\documentclass[11pt,a4paper,reqno]{amsart}

\usepackage{amscd,amsmath}
\usepackage{amssymb,amsmath,latexsym,color,enumerate,tikz,dsfont,tikz-cd,hyperref}
\usepackage[all]{xypic}       
\usepackage[varg]{pxfonts}
\usepackage{tabu,adjustbox}
\usepackage{graphicx}

 \newcommand{\newdownarrow}{{{\rlap{$\ $}\hbox{$\downarrow$}}}}
 \newcommand{\newuparrow}{{{\rlap{$\ $}\hbox{$\uparrow$}}}}
 \newcommand{\twoheaddownarrow}{{\rlap{\rlap{$\ $}\raise .25ex\hbox{$\downarrow$}}\raise-.25ex\hbox{$\downarrow$}}}
 \newcommand{\twoheaduparrow}{{\rlap{\rlap{$\ $}\raise .25ex\hbox{$\uparrow$}}\raise-.25ex\hbox{$\uparrow$}}}
\def\zerohedgehog{\boldsymbol{0}}

\newcommand{\tbigwedge}{\mathop{\textstyle \bigwedge }}
\newcommand{\tbigcap}{\mathop{\textstyle \bigcap }}
\newcommand{\tbigcup}{\mathop{\textstyle \bigcup }}
\newcommand{\tbigprod}{\mathop{\textstyle \prod }}
\newcommand{\tbigvee}{\mathop{\textstyle \bigvee }}
 \newcommand{\tbigsum}{\mathop{\textstyle \sum }}
 
\makeatletter
\newcommand*{\@old@slash}{}\let\@old@slash\slash
\def\slash{\relax\ifmmode\delimiter"502F30E\mathopen{}\else\@old@slash\fi}
\makeatother

\def\R{\mathbb{R}}
\def\Q{\mathbb{Q}}

\def\N{\mathbb{N}}
\def\Z{\mathbb{Z}}

\def\NN{\mathcal{N}}

\newtheorem{theorem}{Theorem}[section]
\newtheorem{proposition}[theorem]{Proposition}
\newtheorem{lemma}[theorem]{Lemma}
\newtheorem{corollary}[theorem]{Corollary}
\newtheorem{fact}[theorem]{Fact}
\newtheorem*{ffact}{Fact \Roman{ffact}}

\theoremstyle{definition}
\newtheorem{definition}[theorem]{Definition}
\newtheorem{example}[theorem]{Example}
\newtheorem{examples}[theorem]{Examples}

\theoremstyle{remark}
\newtheorem{remark}[theorem]{Remark}
\newtheorem{remarks}[theorem]{Remarks}
\newtheorem{properties}[theorem]{Properties}

\newcounter{ffact}
\begin{document}
\title{A tale of three hedgehogs}
\author{Igor Arrieta Torres}
\maketitle

\begin{center}\textbf{Supervisors}: Jon Gonz\'alez S\'anchez and Javier Guti\'errez Garc\'ia\end{center}

\tableofcontents 

\newpage 
\section{Introduction}
In this work we study three topologies defined over the same set: the \emph{hedgehog}. As the name suggests, the hedgehog can be described as a set of spines identified at a single point.  The first topology on the hedgehog will be a quotient topology, and the resulting space will said to be the \emph{quotient hedgehog}. The main feature of the next topology, which we shall refer to as the \emph{compact hedgehog}, will, of course, be compactness.  The third and last topology will be generated by a metric, and thus the resulting space will said to be the \emph{metric hedgehog}.
Each of the spaces has its particular properties and several interesting implications in Topology; so let us start by figuring out the importance of these spaces:
\medskip

\noindent {\bf Motivation}
\medskip 

There are a number of reasons why topological hedgehogs are worthy of being studied in depth. Firstly, hedgehog spaces are a nice source of counterexamples in topology. For instance, the quotient hedgehog will turn out to be one of the easiest examples of a quotient of a first countable space which is not first countable. We will also prove that the quotient hedgehog is an example of a Fr\'echet--Urysohn space which is not first countable. Furthermore, a classical counterexample in topology will be studied, namely the Fr\'echet--Urysohn fan, which is strongly related to the quotient hedgehog. Many other interesting examples will arise as a consequence of the study of the topological properties of the three hedgehogs. \\[2mm]
Secondly, the hedgehogs are of great importance in point-set topology, and they have a number of interesting applications. More precisely, we will prove some deep results concerning the hedgehogs, one of them being the Kowalsky's hedgehog theorem, which surprisingly asserts that every met\-rizable space is embeddable into a countable cartesian power of the metric hedgehog.  Besides, the concept of collectionwise normality will be studied and we will provide a full answer to the task of determining whether a given space is collectionwise normal.  The key in that answer will turn out to be the metric hedgehog. Additionally, we are going to obtain infinitely many different characterizations of normality in topological spaces, all of them based on the metric hedgehog.\\[2mm]
Last but not least, the topological hedgehogs provide a good opportunity to learn about many different facts from general point-set topology. In fact, almost all the concepts learnt in the course in Topology of the UPV/EHU have appeared throughout these notes, and many others have also been introduced and studied. Therefore, our three hedgehogs are the perfect partners for a pleasant journey through the main areas of general topology.
\medskip

\noindent {\bf Organization of the memory}
\medskip

These notes are organized as follows. In Section~\ref{ch1}, we give a number of preliminary results in order to make this notes self-contained.

In Section~\ref{ch2} the hedgehog is defined as a set and its most immediate properties are studied. Besides, the hedgehog is given a structure of partially ordered set and complete semilattice. Further, in Subsection~\ref{axescube} we provide an alternative description of the the hedgehog as a subset of the cube $[0,1]^I$. 

Section~\ref{ch3} deals with the first topology on the hedgehog, namely the quotient topology. Firstly, a complete description of the topology is given, and then its most remarkable topological features are explained. Among others, in Subsection~\ref{secfre} we explain its relation to the Fr\'echet--Urysohn fan.

In Section~\ref{ch4}, we introduce the compact topology on the hedgehog and the main properties are studied.

The last topology on the hedgehog is presented in Section~\ref{ch5}. More precisely, the hedgehog will be seen as a metric space. Among others, we shall give a proof for the Kowalsky's Hedgehog Theorem (Subsection~\ref{secKo}) and a Tietze-type extension theorem for the metric hedgehog will also be shown (Subsection~\ref{secNor}).
\medskip

\noindent {\bf Personal work}
\medskip

These notes were originally inspired by \cite{ASOTH}, a homework exercise set proposed by Mikhail Matveev (George Mason University). However, we have gone much further in the study of the three hedgegos than what it is asked in \cite{ASOTH}. In fact, several properties about the quotient and compact hedgehog have been proved by the author. Regarding the metric hedgehog, we have mainly followed the book \cite{Eng}. On the one hand,  details of the proofs left to the reader have been completed by the author. On the other hand, the dissertation also includes a number of solved exercises, most of them proposed in \cite{Eng}, which have been presented together with the text, since they accompany the theory and the development of these notes. Examples of work carried out by the author include the proofs of Proposition \ref{Rembeds} (whose proof is only sketched in \cite{Prz}), certain properties of the metric hedgehog, some results shown in Section \ref{secKo} (particularly Theorems~\ref{app1} and~\ref{app2}), as well as Theorem~\ref{Fsigmahereditary} (as far as the author knows the proof of this result has only been published in Russian, cf.~\cite{Sed}) and Theorem~\ref{colnorTh}.
\section{Notation}

Given a topological space $(X,\tau_X)$ and $x\in X$, the family of all the neighborhoods of $x$ will be  denoted by $\NN_x$. The closure (resp. interior) of a set $A\subseteq X$ will be denoted by $\overline{A}$ (resp. $\textrm{int } A$). If $B\subseteq A \subseteq X$, we may write $\overline{B}^{A}$ for the closure of $B$ in $A$, i.e. $\overline{B}^{A}= \overline{B}\cap A$, where $\overline{B}$ is the closure of $B$ in $X$. If $(Y,\tau_Y)$ is a further topological space, we may write $X\cong Y$ if $X$ and $Y$ are homeomorphic.

Let $A$ be a set and $f\colon A\longrightarrow[0,1]$ a mapping. Throughout these notes, we shall write  
\begin{align*}
 &[f > t] = f^{-1} \left( (t,1] \right)\ \text{and}\ [f \geq t] = f^{-1} \left( [t,1] \right)\ \text{for each }t\in[0,1)\text{ and}\\
 &[f<t] =f^{-1}\left( [0,t) \right)\ \text{and}\ [f\leq t] =f^{-1}\left( [0,t] \right)\ \text{for each }t\in(0,1].
\end{align*}

Given a partially ordered set $P$ and $x,y \in P$, we will denote by $x\vee y$ (resp. $x \wedge y$)  the supremum (resp. infimum) of $x$ and $y$ when it exists. Similarly, if $S\subseteq P$, we will write  $\tbigvee S$ (resp. $\tbigwedge S$) for the supremum (resp. infimum) of $S$, whenever it exists.

For $X$ a set and $A\subseteq X$, the symbol $\chi_A$ stands for the characteristic function of $A$, i.e. the mapping $\chi_A \colon X\longrightarrow \{0,1\}$ given by  $\chi_A (x)=0$ if $x\not\in A$ and $\chi_A(x)=1$ if $x\in A$. 

In these notes, $\kappa$ will always denote some cardinal. Besides, the symbol $\aleph_0$ will be written for the cardinality of the set of natural numbers $\N$, i.e. it will denote a countably infinite cardinal. Similarly, the symbol $\mathfrak{c}$ will be written for the cardinality of the set of all  real numbers. Given a set $X$, we will usually write $|X|$ for the cardinality of $X$.
\section{Preliminaries}
\label{ch1}
In this section we provide all the preliminary results needed for a smooth development and understanding of this work. We will specially focus on the concepts and results of general topology which are not covered in the degree in mathematics of the UPV/EHU.

Recall that a topological space $(X,\tau)$ is \emph{separable} if there is a countable dense subset $D\subseteq X$.

 \begin{lemma}\label{lemmasep} Let $(X,\tau_X)$ and $(Y,\tau_Y)$ be topological spaces, $f\colon X \longrightarrow Y$ a continuous map and assume that $X$ is separable. Then $f(X)$ is separable. 
\end{lemma}

\begin{proof}
Let $D$ be  a countable dense subset of $X$. Since $D$ is countable, $f(D)$ is countable. By continuity of $f$, $f(\overline{D})\subseteq \overline{f(D)}$.  Note that $f(\overline{D})=f(X)$ because $D$ is dense. Let $\overline{f(D)}^{f(X)}$ denote the closure of $f(D)$ in $f(X)$. One has
$$ f(X)=  f(X)\cap f(\overline{D}) \subseteq f(X)\cap \overline{f(D)} =  \overline{f(D)}^{f(X)} \subseteq f(X), $$
which shows that $f(D)$ is dense in $f(X)$, and hence $f(X)$ is separable.
\end{proof}

\begin{lemma}\label{SepCountProd} 
The product of countably many separable topological spaces is separable with the Tychonoff  topology.
\end{lemma} 

\begin{proof}
Let $\{X_n\}_{n\in\N}$ be a countable family of separable topological spaces and set $X=\tbigprod_{n\in\N}X_n$. For each $n\in\N$, let $D_n$ denote a countable dense subset of $X_n$.  Fix $x=\{x_n\}_{n\in\N}$ in $X$ and define
$$D= \tbigcup_{n\in\N} \left( \tbigprod_{m=1}^{n} D_m \right)\times\left(\tbigprod_{m=n+1}^{+\infty}\{x_m\}\right).$$
Note that $D$ is countable because it is a countable union of countable sets. We now show that $D$ is also dense in $X$. Let $y=\{y_n\}_{n\in\N}$ in $X$. Let $U= \tbigprod_{n\in\N} U_n$ be a basic open neighborhood of $y$, where $U_n$ is open in $X_n$ for every $n\in\N$ and $U_j=X_j$ whenever $j\not\in J$ for some finite subset $J\subseteq \N$. Let $n_0=\tbigvee J \in \N$. For every $n=1,\dots, n_0$, since $D_n$ is dense, there is $z_n \in D_n \cap U_n$. For each $n>n_0$, set $z_n=x_n$. Then $z=\{z_n\}_{n\in\N} \in D\cap U$ and $D$ is dense.
\end{proof}

\begin{definition} 
Let $X$ be a set,  $\{(Y_i,\tau_i)\}_{i\in I}$  a family of topological spaces, and $\{ f_i \colon X\longrightarrow (Y_i,\tau_i) \}_{i\in I}$ a family of maps. The \emph{initial topology} on $X$  is the coarsest topology which makes all the maps $f_i$ continuous. 
\end{definition}

It is easy to show that the initial topology coincides with the topology on $X$ which is generated by the subbasis $\{ f_{i}^{-1} (U) \mid  U\in \tau_i,i\in I \}$. 

\begin{examples}\label{exampinitial}
(1) Let $\{(X_i,\tau_i)\}_{i\in I}$ be a family of topological spaces. The product topology is the initial topology with respect to the coordinate projections $\left\{ \varphi_i \colon \tbigprod_{j\in I} X_j \longrightarrow (X_i,\tau_i) \right\}_{i\in I}$.\\[2mm]
(2) Let $(X,\tau_X)$ be a topological space and $A\subseteq X$. The initial topology with respect to the inclusion map $\iota \colon A\longrightarrow (X,\tau_X)$ is precisely the subspace topology of $A$. 
\end{examples}

Later, the following universal property concerning the initial topology will be fundamental. 

\begin{proposition}\label{univpropinitial} Let $(X,\tau_X)$ be the initial topology with respect to the family $\{(Y_i,\tau_i)\}_{i\in I}$  of topological spaces and the family $\{ f_i \colon X\longrightarrow (Y_i,\tau_i) \}_{i\in I}$ of maps. Let $(Z,\tau_Z)$ be a further topological space. Then, a mapping $g\colon (Z,\tau_Z)\longrightarrow (X,\tau_X)$ is continuous if and only if $f_i \circ g$ is continuous for every $i\in I$.
\[\begin{tikzcd}
X \arrow{r}{f_i} &Y_i\\
Z\arrow{u}{g}\arrow{ur}[swap]{f_i \circ g}
\end{tikzcd}\] 
\end{proposition}

\begin{proof}
$\Rightarrow$) This implication is clear because the composition of continuous maps is continuous.\\[2mm]
$\Leftarrow$) Assume that $f_i \circ g$ is continuous for every $i\in I$. Let $f_{i}^{-1}(U)$ be a subbasic open set in $(X,\tau_X)$, where $U\in \tau_i$ and $i\in I$. Then $g^{-1} \left( f_{i}^{-1}(U) \right) = (f_i \circ g)^{-1}(U)$, which is open by continuity of $f_i \circ g$.  Hence, $g$ is continuous.
\end{proof}

\begin{lemma}\label{bijecthomeom}
Let $X$ be a set and $(Y,\tau_Y)$  a topological space. Assume that $f\colon X\longrightarrow Y$ is a bijective map. Then, the family $\tau_X=\{ f^{-1}(U) \mid U\in \tau_Y \}$ is a topology defined on $X$ which makes $f\colon (X,\tau_X) \longrightarrow (Y,\tau_Y)$ a homeomorphism. Moreover, if $\beta_Y$ is a base (resp. subbase) of $(Y,\tau_Y)$, then $\beta_X=\{ f^{-1}(B) \mid B\in \beta_Y\}$ is a base (resp. subbase) of $( X,\tau_X)$. 
\end{lemma}

\begin{proof}
It is very easy to show that $\tau_X$ is a topology on $X$, which makes $f\colon (X,\tau_X) \longrightarrow (Y,\tau_Y)$ continuous.  Since $f$ is surjective, we deduce that $f\colon (X,\tau_X) \longrightarrow (Y,\tau_Y)$ is open, and thus a homeomorphism. Assume now that $\beta_Y$ is a basis of $(Y,\tau_Y)$ and set $\beta_X = \{f^{-1}(B) \mid B\in \beta_Y\}$. Then $\beta_X\subseteq \tau_X$, and for each $V=f^{-1}(U) \in\tau_X$ (where $U\in\tau_Y$), one can write $U=\tbigcup_{i\in I} B_i$ where $B_i\in \beta_Y$ because $\beta_Y$ is a basis. Thus $V=\tbigcup_{i\in I} f^{-1}(B_i)$ and $\beta_X$ is a basis. The assertion corresponding to the subbase can be proved in the same way.
\end{proof}

\begin{definition}
A topological space is $X$ is \emph{universal} in a class $\mathcal{C}$ of topological spaces if
\begin{enumerate}[\normalfont (i)]
\item $X$ belongs to the class $\mathcal{C}$;
\item Every topological space in $\mathcal{C}$ can be embedded into $X$.
\end{enumerate}
\end{definition}

\begin{definition}
A subset in a topological space is said to be an \emph{$F_{\sigma}$-set} if it is a countable union of closed sets. Dually, a subset is said to be a \emph{$G_{\delta}$-set} if it is a countable intersection of open sets. 
\end{definition}

Let us also recall the following fact from general topology.
\begin{proposition}\label{comhaus}
Every compact Hausdorff space is normal.
\end{proposition}

\subsection{Metric spaces}We shall need to recall now some definitions and properties concerning metric spaces.

\begin{definition}
A metric space $X$ is said to be \emph{totally bounded} if for every $\varepsilon>0$ there exists a finite collection of open balls of radius $\varepsilon$ whose union equals $X$. 
\end{definition}

\begin{lemma}\label{2ndcountiffsep}
A metric space is separable if and only if it is second countable.
\end{lemma}

\begin{proof}
$\Rightarrow)$ Let $D=\{x_n\mid n\in\N\}$ be a countable dense subset in a separable metric space $(X,d)$. Denote by $B(x,r)$ the ball centered at $x\in X$ of radius $r>0$, and define
$$\beta = \{ B(x_n , 1/k ) \mid n,k\in \N\}.$$
We will show that $\beta$ is a countable basis of $X$. The countability is clear, so let $U$ be an open subset of $X$ and take $x\in U$. Then there is an $r>0$ such that $B(x,r) \subseteq U$. Select $k\in \N$ such that $1/k < r/2$. By density of $D$, one has $D \cap B(x,1/k) \ne \varnothing$,  so there is an $n\in \N$ which satisfies $x_n \in D \cap B(x,1/k)$.  Note that $x\in B(x_n, 1/k)$, so the only task remaining is to show that $B(x_n,1/k) \subseteq U$. Indeed, let $y\in B(x_n,1/k)$. One has 
$$d(y,x) \leq d(y,x_n)+ d(x_n,x) < 1/k +1/k = 2/k < r,$$
 and hence $y\in B(x,r) \subseteq U$, as desired. \\[2mm]
$\Leftarrow$) Let $\beta=\{B_n\}_{n\in\N}$ be a countable basis of a second countable metric space $(X,d)$. For all $n\in\N$, choose $x_n\in B_n$ and set $D=\{x_n\mid n\in\N\}$. It is clear that $D$ is a countable dense subset of $X$. 
\end{proof}

\begin{corollary}\label{SepHered}
Every subspace of a separable metric space is separable.
\end{corollary}

\begin{proof}
Let $(X,d)$ be a separable metric space and $A\subseteq X$. By the previous lemma, $X$ is second countable, which is an hereditary property. Thus $A$ is second countable, and again by  the previous lemma, $A$ is separable.
\end{proof}

Let $(X,d)$ be a metric space. Recall that one defines $ d(x,A) = \inf\{ d(x,y) \mid y\in A\}$ for all nonempty subsets $A\subseteq X$ and $x\in X$.  For convenience, we set $d(x,\varnothing)=1$ for all $x\in X$. 

\begin{lemma}{\rm (Cf.\ \cite[page\ 254,\ Proposition\ 4.1.9,\ Corollary\ 4.1.12]{Eng})} \label{distcontclosed}
Let $(X,d)$ be a metric space and $A\subseteq X$. Then the map 
 \[\begin{array}{rcclll}
 f_A \colon&X&\longrightarrow&\R &&\cr
 &x&\longmapsto&f_A(x) = d(x,A)
 \end{array}\]
 is continuous. Moreover, if $A$ is closed, then $A=f_{A}^{-1}\left(\{0\}\right)$.
\end{lemma}

\begin{proof}
Clearly we can assume that $A\ne\varnothing$. First we shall show that for all $x,y\in X$
$$| d(x,A) -d(y,A) | \leq d(x,y).$$
Let $x,y\in X$. For all $a\in A$ the triangular inequality yields $d(x,A)\leq d(x,a)  \leq d(x,y)+ d(y,a)$, and taking the infimum, one gets $d(x,A)\leq d(x,y)+ d(y,A)$, that is,  $d(x,A)-d(y,A)\leq d(x,y)$. By symmetry, we also have $d(y,A)-d(x,A)\leq d(x,y)$ and thus $| d(x,A) -d(y,A) | \leq d(x,y)$ follows. Then, for every $\varepsilon > 0$, we have $|f_A(x) - f_A(y)| \leq  d(x,y)<\varepsilon$ whenever $d(x,y)<\varepsilon$, i.e. $f_A$ is continuous. 

Now we show the second part of the statement. Suppose that $A$ is closed. The inclusion $A\subseteq f_A^{-1}\left( \{0\} \right)$ is clear  and always holds, so let us prove the reverse one.  Suppose that $x\in f_A^{-1}\left( \{0\} \right)$, i.e.  $d(x,A)=0$. Then for each $n\in \N$ there is a $y_n\in A$ such that $d(x,y_n)<1/n$. This means that the sequence $\{y_n\}_{n\in \N}\subseteq A$ converges to $x$. Hence $x\in \overline{A}=A$, as we wanted to prove. 
\end{proof}

The following result asserts that there is no loss of generality in assuming that the metric in a metric space is bounded.

\begin{lemma}{\rm (Cf.\ \cite[page\ 250,\ Theorem\ 4.1.3,\ Corollary\ 4.1.12]{Eng})} \label{lemmabd1}
For every metric space $(X,d)$ there exists a metric $d_1$ on the set $X$ which is bounded by $1$ and induces the same topology as $d$ does.
\end{lemma}

\begin{proof}
It is straightforward to verify that 
$$d_1(x,y) = \min \{ 1 , d(x,y) \} \quad \textrm{ for } x,y\in X$$ 
defines a metric on $X$. Let us check that they induce the same topology. Denote by $B_{d}(x,r)$ and $B_{d_1}(x,r)$ the open balls of radius $r$ centered at $x$ with respect to the metrics $d$ and $d_1$ respectively. Since $d_1 \leq d$, one has $B_{d} (x,r) \subseteq B_{d_1} (x,r)$ for every $r>0$ and $x\in X$.  Thus the topology induced by $d$ is finer than the one induced by $d_1$. Let us now check that the topology induced by $d_1$ is finer than the one induced by  $d$. Let $x\in X$ and $r>0$. Select $r_1 = \min \{r,1\}$. Then one can easily check that $B_{d_1}(x,r_1) \subseteq B_d (x,r)$.
\end{proof}

\begin{theorem}{\rm (Cf.\ \cite[page 259, Theorem~4.2.2]{Eng})} \label{countablemetrizable} Countable products of met\-rizable spaces are met\-rizable with the Tychonoff topology.
\end{theorem}

\begin{proof}
Let $\{(X_n,d_n)\}_{n\in \N}$ be a countable family of metric spaces and set \linebreak $X=\tbigprod_{n\in \N} X_n$.  By the previous lemma we may assume that $d_n$ is bounded by $1$ for every $n\in\N$.  We define
$$d(x,y) = \tbigsum_{n\in \N} \frac{d_n(x_n,y_n)}{2^n} , \quad \textrm{for } x=\{x_n\}_{n\in \N} \textrm{ and } y=\{y_n\}_{n\in \N} \textrm{ in } X.$$
First note that the series converges because of the term $1/2^n$ and because $d_n(x_n,y_n) \leq 1$ for every $n\in \N$ . It is straightforward to verify that $d$ is a metric in the cartesian product. Let us now check that it precisely induces the product topology. 

Denote by $\varphi_n$ the usual coordinate projection for each $n\in \N$. Given $n\in\N$,  $\varepsilon> 0$, $x=\{x_m\}_{m\in \N}$ and $y=\{y_m\}_{m\in \N}$ in $X$, we clearly have $d_n( \varphi_n(x),\varphi_n(y))=d_n(x_n,y_n) < \varepsilon$ whenever $d(x,y) <\varepsilon /2^n$. Therefore,  \linebreak$\varphi_n \colon (X,d) \longrightarrow (X_n,d_n)$ is continuous for every $n\in \N$.  Since the Tychonoff topology is the initial topology with respect to the coordinate projections (see Examples~\ref{exampinitial}), we deduce that the Tychonoff topology is coarser than the topology generated by $d$. Now let us show that any open set $U$ in $(X,d)$ is also open with the topology of the Cartesian product.

Let $x=\{x_n\}_{n\in\N} \in U$. Then there is an $r>0$ such that $B_d(x,r) \subseteq U$. Select  $k\in \N$ such that
$$ \tbigsum_{n=k+1}^\infty \frac{1}{2^n} = \frac{1}{2^k} < \frac{r}{2}.$$
For each $n=1,\dots k$, define $U_n = B_{d_n} (x_n, r/2)$. We then have
$$x\in \tbigcap_{n=1}^{k} \varphi_{n}^{-1}(U_n) \subseteq B(x,r) \subseteq U.$$
Indeed, for every $y= \{y_m\}_{m\in \N}\in \tbigcap_{n=1}^{k} \varphi_{n}^{-1}(U_n)$, we have   $d_n(x_n,y_n)< r/2$ whenever $n \leq k$.
Hence,
$$d(x,y) =\tbigsum_{n=1}^k \frac{d_n (x_n,y_n)}{2^n}  +  \tbigsum_{n=k+1}^\infty \frac{d_n (x_n,y_n)}{2^n}  < \frac{r}{2} + \frac{r}{2}=r,$$
and so $y\in B(x,r)$. Since $ \tbigcap_{n=1}^{k} \varphi_{n}^{-1}(U_n)$ is open in the Cartesian product, we conclude that so is $U$.
\end{proof}

We recall the following elementary result concerning Cauchy sequences.

\begin{lemma}\label{ChConv}
A Cauchy sequence in a metric space is convergent if and only if it has a convergent subsequence.
\end{lemma}
 
 \begin{lemma}\label{lemmaconvergentlimit}
 Let $\{x_n\}_{n\in\N}$ be a sequence in a metric space $(X,d)$,  set $S=\{x_n \mid n\in\N\}\subseteq X$, and let $x$ be a limit point of $S$. Then, $\{x_n\}_{n\in\N}$ has a subsequence converging to $x$. 
 \end{lemma}
 
 \begin{proof}
 We build the desired subsequence $\{x_{n_k}\}_{k\in\N}$ iteratively, as follows. Let $n_1=1$ and for every $k\geq 2$ we define 
 $$n_{k+1}=\min \bigl\{n\in\N \mid n>n_k 	\textrm{ and } d(x,x_n)<\frac{1}{k+1}\bigr\}.$$
 We have to check that $n_{k+1}$ exists, i.e. that the set
$$\bigl\{n\in\N \mid n>n_k 	\textrm{ and } d(x,x_n)<\frac{1}{k+1}\bigr\}$$
 is nonempty. Indeed, since $x$ is a limit point (in a metric space) of $S$, the open ball $B(x,1/(k+1))$ has infinitely many points of $S$, and, in particular it contains a point $x_{n}$ of $S$ with $n>n_k$. Thus, $n_{k+1}$ is well-defined.  
 
 Now we show that the subsequence $\{x_{n_k}\}_{k\in\N}$ converges to $x$. Let $\varepsilon>0$ and choose $k_0\in\N$ with $1/k_0 < \varepsilon$. Then $d(x,x_{n_k})<1/k \leq 1/k_0<\varepsilon$ whenever $k\geq k_0$, which concludes the proof.
  \end{proof}

\subsection{Complete metric spaces}

\begin{definition}
A metric space $(X,d)$ is said to be \emph{complete} if every Cauchy sequence converges to a point of $X$. In that case we say that $d$ is a \emph{complete metric}.
\end{definition}

\begin{definition}
A topological space $(X,\tau)$ is said to be \emph{completely met\-rizable} if there exists a complete metric defined on $X$ which induces the topology $(X,\tau)$.
\end{definition}

\begin{remark}\label{compmetrtopprop}
Completeness is not a topological property (since it does not make sense in non-met\-rizable spaces). However, one can easily check that complete metrizability is a topological property. 
\end{remark}
Finally we recall here two results on $G_\delta$ subsets in met\-rizable spaces that we will need later on. We omit the proofs due to lack of space.

 \begin{lemma}{\rm (Cf.\ \cite[page\ 274,\ Lemma\ 4.3.22]{Eng})} \label{closedofcartesian}
Every $G_\delta$ subset in a met\-rizable space $X$ is homeomorphic to a closed subspace of the cartesian product $X \times \mathbb{R}^{\aleph_0}$. 
\end{lemma}

\begin{theorem}{\rm (Cf.\ \cite[page\ 274,\ Theorem\ 4.3.24]{Eng})} \label{complmetrizGdelta}
If a subspace $M$ of a met\-rizable space $X$ is completely met\-rizable, then $M$ is a $G_\delta$-set in $X$. 
\end{theorem}

\subsection{Fr\'echet--Urysohn spaces} 
In general topological spaces, the following result is well known.

\begin{proposition}\label{sequenceclosure}
Let $X$ be a topological space, $A\subseteq X$ and $\{x_n\}_{n\in \N}\subseteq A$ a sequence converging to $x$. Then $x\in \overline{A}$.
\end{proposition}

The converse is not true in general. However, for the class of first countable spaces, one can easily prove that the result is positive.

\begin{proposition}
If $X$ is a first countable topological space,  for all $A\subseteq X$ and $x\in \overline{A}$ there is a sequence $\{x_n\}_{n\in \N}\subseteq A$ that converges to $x$.
\end{proposition}

Later, we will show that there are topological spaces which are not first countable but where the converse of Proposition~\ref{sequenceclosure} is still true. Hence, we introduce the following:

\begin{definition}\label{defFU}
A topological space $X$ is said to be a \emph{Fr\'echet--Urysohn space} if for all $A\subseteq X$ and $x\in \overline{A}$ there is a sequence $\{x_n\}_{n\in \N}\subseteq A$ that converges to $x$.
\end{definition}

\begin{proposition}\label{FUhered}
The property of being a Fr\'echet--Urysohn space is hereditary.
\end{proposition}

\begin{proof}
Let $X$ be a Fr\'echet--Urysohn space and take $A\subseteq X$. Our goal is to prove that $A$ is also Fr\'echet--Urysohn. Let $B\subseteq A$ and $x\in \overline{B}^A$. Since $\overline{B}^A= A\cap \overline{B}^X$, we have $x\in \overline{B}$, and thus there is a sequence $\{x_n\}_{n\in\N}\subseteq B$ converging to $x$ (in $X$). Let us check that the sequence $\{x_n\}_{n\in\N}$ also converges to $x$ in $A$. Indeed, let $N$ be a neighborhood of $x$ in $A$. Then there is a neighborhood $M$ of $x$ in $X$ such that $N=M\cap A$. Since $\{x_n\}_{n\in\N}$ converges to $x$ in $X$, there is an $n_0 \in \N$ such that $x_n \in M$ whenever $n\geq n_0$. Since $x_n \in B\subseteq A$ for all $n\in \N$, one has that $x\in N$ whenever $n\geq n_0$, and thus $\{x_n\}_{n\in\N}$ converges to $x$ in $A$. 
\end{proof}

\subsection{Weight of a topological space} 

\begin{definition}\label{defweig}
For $X$ a topological space, the \emph{weight} of $X$ is defined to be the minimum cardinality of a basis of $X$.  
\end{definition}

In what follows, we shall frequently denote the weight of $X$ by $\omega(X)$.  We will later need the following lemmas only in the case $\kappa\geq \aleph_0$.

\begin{lemma}{\rm (Cf.\ \cite[page\ 17,\ Theorem\ 1.1.14]{Eng})}
Let $X$ be a topological space and $\omega(X)\leq \kappa$. Then for every nonempty family $\{U_i\}_{i\in I}$ of open sets there exists $I_0\subseteq I$ such that $|I_0| \leq\kappa$ and $\tbigcup_{i\in I_0} U_i = \tbigcup_{i\in I} U_i$. 
\end{lemma}

\begin{proof} Let $\{U_i\}_{i\in I}$ be a nonempty family of open sets.
Since $\omega(X) \leq \kappa$, there exists a basis $\beta$ of $X$ with $|\beta|\leq \kappa$. Define
$$\beta_0 = \{ B\in \beta \mid B\subseteq U_i \textrm{ for some } i\in I\}.$$
For each $B\in \beta_0$ choose $i(B)\in I$ such that $B\subseteq U_{i(B)}$. This allows us to define a map $f\colon \beta_0 \longrightarrow I$ such that $f(B)=i(B)$. Set $I_0=f(\beta_0)$. Let us show that $I_0$ satisfies the required property. On the one hand, note that $\beta_0\subseteq \beta$, and thus 
$$| I_0 | = |f(\beta_0)|\leq |\beta_0|\leq |\beta|\leq \kappa.$$ 
The inclusion $\tbigcup_{i\in I_0} U_i \subseteq \tbigcup_{i\in I} U_i$ is obvious, so let us check the reverse one. Let $x\in \tbigcup_{i\in I} U_i$. Then there is an $i\in I$ such that $x\in U_i$. Since $\beta$ is a basis and $U_i$ is open, there exists $B\in\beta$ such that $x\in B\subseteq U_i$. Clearly,  $B\in \beta_0$ and thus $f(B)=i(B)\in I_0$. Hence, $B\subseteq U_{i(B)} \subseteq \tbigcup_{i\in I_0} U_i$, as desired. 
\end{proof}

\begin{lemma}{\rm(Cf.\ \cite[page\ 17,\ Theorem\ 1.1.15]{Eng})} \label{smallerbasis} Let $X$ be a topological space and $\omega(X) \leq \kappa$. Then for every basis $\beta$ of $X$ there is a basis $\beta_0$ such that $|\beta_0|\leq \kappa$ and $\beta_0 \subseteq  \beta$. 
\end{lemma}

\begin{proof}
Let $\beta=\{U_j\}_{j\in J}$ be a basis of $X$. Since $\omega(X)\leq \kappa$, there exists a basis $\beta_1=\{B_i\}_{i\in I}$ such that $| I |\leq \kappa$.\\[2mm]
(1) Suppose first that $\kappa\geq \aleph_0$.  For every $i\in I$, set $J(i) = \{ j\in J\mid U_j\subseteq B_i\}$. It is clear that $\tbigcup_{j\in J(i)}U_j = B_i$ because $\beta$ is a basis. The previous lemma yields $J_0 (i) \subseteq J(i)$ such that $|J_0 (i) | \leq \kappa$ and
$$B_i =\tbigcup_{j\in J(i)}U_j= \tbigcup_{j\in J_0 (i)}U_j .$$
Define $\beta_0=\{U_j \mid j\in J_0(i),\, i\in I\}$. Since $|I|\leq\kappa$ and $|J_0(i)|\leq\kappa$, one has $|\beta_0|\leq \kappa$ (because $\kappa$ is infinite). Now we check that $\beta_0$ is a basis. Let $U$ be open and $x\in U$. Since $\beta_1$ is a basis, there is $i\in I$ such that $x\in B_i \subseteq U$. Since $B_i =  \tbigcup_{j\in J_0 (i)}U_j$, there is some $j\in J_0 (i)$ such that $x\in U_j \subseteq B_i$. Thus $x\in U_j \subseteq U$ and $\beta_0$ is a basis.  \\[2mm]
(2) Now we deal with the case $\kappa<\aleph_0$. In this case we will show that $\beta_1 \subseteq \beta$. Let $B_i\in\beta_1$ and $J(i) = \{ j\in J\mid U_j\subseteq B_i\}$. Since $B_i$ is open and $\beta$ is a basis it is clear that $\tbigcup_{j\in J(i)}U_j = B_i$. Similarly, since each $U_j$ is open and $\beta_1$ is a basis, one has $U_j = \tbigcup_{k(j) \in I(j)} B_{k(j)}$, where $I(j) = \{k(j)\in I \mid B_{k(j)} \subseteq  U_j\}$. Thus,
$$ B_i = \tbigcup_{j\in J(i)} \tbigcup_{k(j)\in I(j)} B_{k(j)}.$$
We distinguish two cases: first, assume that $B_i\ne B_{k(j)}$ for every $k(j)\in I(j)$ and $j\in J(i)$. Then one can remove the element $B_i$ from the basis $\beta_1$, obtaining a new basis with strictly smaller cardinality, a contradiction. Therefore, there is a $j_0\in J(i)$ and $k_0(j_0)\in I(j_0)$ such that $B_{k_0(j_0)}=B_i$, from which follows that
$$B_{i}= B_{k_0(j_0)} \subseteq \tbigcup_{k(j_0)\in I(j_0)} B_{k(j_0)} = U_{j_0} \subseteq B_i,$$
that is, $B_i= U_{j_0} \in \beta$. 
\end{proof}

\begin{lemma}\label{finitecount}
The family consisting of all finite subsets of a countable set is countable.
\end{lemma}

\begin{proof}
 Let  $X$ be a  countable set and  $\mathcal{J}=\{ J\subseteq X\mid J \textrm{ is finite}\}$. 
 Let $X=\{x_n \mid n\in \N\}$ be an enumeration of $X$.
 Define
    \[\begin{array}{rcclll}
 \varphi \colon&\mathcal{J}&\longrightarrow&\Q&&\cr
 &J&\longmapsto&\varphi \left(J\right)  = 0.z_1z_2\dots  \textrm{ where } z_i=\begin{cases} 1, & \textrm{if } x_i\in J; \\ 0, & \textrm{if }x_i\not\in J.\end{cases}
 \end{array}\]
 It is clear that $\varphi$ is one-to-one, and thus $\mathcal{J}$ is countable. 
\end{proof}

The following lemma establishes the relation between the weight of a space and the weight of its countable cartesian power (with respect to the Tychonoff  topology).

\begin{lemma}\label{prodwrithk}
Let $\kappa \geq \aleph_0$ be some cardinality and $(X,\tau_X)$ a topological space of weight $\omega (X) = \kappa$.  Then $\omega (X^\N) =\kappa$.
\end{lemma}

\begin{proof}
Let $\beta$ be a basis of $X$ with $|\beta|=\kappa$.  For each $n\in\N$, denote by $\varphi_n \colon X^\N \longrightarrow X$ the $n$th coordinate projection. A basis of the product $X^\N$ is given by
\begin{align*} \tilde\beta= \bigl\{ \tbigcap_{n\in J} \varphi^{-1}_n (B_n) \mid J\subseteq \N \textrm{ finite, } B_n\in\beta\bigr\} = \tbigcup_{J\in \mathcal{J}} K_J , \end{align*}
where $\mathcal{J}=\{ J\subseteq \N \mid J \textrm{ is finite} \}$ and
 $$K_J =\bigl\{ \tbigcap_{n\in J} \varphi^{-1}_n (B_n) \mid B_n\in\beta\bigr\} = \bigl\{ \tbigprod_{n\in \N} B_n \mid B_n\in \beta, B_j=X \quad \forall j\not\in J\bigr\}. $$
 It is clear that each $K_J$ is in bijection with $\beta ^J$, via the mapping 
   \[\begin{array}{rcclll}
 f \colon&K_J&\longrightarrow&\beta^J&&\cr
 &\tbigprod_{n\in \N} B_n &\longmapsto&f\left( \tbigprod_{n\in \N} B_n  \right)  = \tbigprod_{j\in J} B_j.
 \end{array}\]
 Then, $|K_J| = |\beta^J|=\kappa ^{|J|} =\kappa$ (note that $\kappa$ is infinite and $|J|$ is finite). 
  By the previous lemma, $\mathcal{J}$ is countable and hence we have proved that $\tilde\beta$ is a countable union of sets of cardinality $\kappa\geq \aleph_0$, thus $|\tilde\beta| =\kappa$ and $\omega (X^\N) \leq \kappa$.   
  
  Finally assume that there is another basis $\beta '$ of $X^\N$ with $|\beta ' | = \kappa' < \kappa$. Now, $X$ is embedded in the product $X^\N$ (i.e. it is homeomorphic to a subspace of $X^\N$), from which follows that $X$ has a basis of cardinality less than $\kappa$, a contradiction. We therefore have $\omega(X^\N)= \kappa$. 
\end{proof}

\subsection{Discrete and pairwise disjoint families} Recall that a family $\{A_i\}_{i\in I}$ of subsets of a given $X$ is said to be \emph{pairwise disjoint} if $A_i \cap A_j =\varnothing$ whenever $i\ne j$.

\begin{definition}
Let $X$ be a topological space. A family $\{A_i\}_{i\in I}$ of subsets of $X$ is said to be \emph{discrete} if for all $x\in X$ there exists $N\in\mathcal{N}_x$ such that 
$$| \{ i\in I \mid A_i \cap N \neq \varnothing \} | \leq 1.$$
\end{definition}

Discreteness is a stronger condition than pairwise disjointness. More precisely, both concepts are related as follows:

\begin{lemma}\label{propertiesdisjointdiscrete} Let $X$ be a topological space.\begin{enumerate}[\normalfont (i)]
\item  Every discrete family of $X$ is pairwise disjoint.
\item  A finite family of closed subsets of $X$ is discrete if and only if it is pairwise disjoint. \end{enumerate}
\end{lemma}

\begin{proof}
(i) Let $\{A_i\}_{i\in I}$ be a discrete family and $x\in A_i \cap A_j$. By discreteness there is a neighborhood $N$ of $x$ such that $A_k\cap N\neq \varnothing$ for at most one $k\in I$. But since $x\in N$ then, it must be $i=j$.\\[2mm]
(ii) By part (i), we only need to show the ``if'' part.
Let $\{F_n\}_{n=1}^k$ be a finite family of closed pairwise disjoint subsets and take $x\in X$. We distinguish two cases. First, if $x\not\in \tbigcup_{n=1}^k F_n$, by finiteness  $U=X\smallsetminus \tbigcup_{n=1}^k F_n$ is an open neighborhood of $x$ and clearly it does not intersect any of the $F_1,\dots ,F_k$. 

Assume now that $x\in F_{n_0}$ for an $n_0\in\{1,\dots k\}$. Let $U=X\smallsetminus \tbigcup_{n\neq n_0} F_n$. Since $F_1,\dots, F_k$ are pairwise disjoint, we have $x\in U$. By finiteness, $U$ is open, and thus $U$ is an open neighborhood of $x$. Finally, by construction, $U$ only intersects $F_{n_0}$. Hence, the discreteness condition is verified.
\end{proof}

An infinite union of closed sets is not necessarily closed. However, in the case of discrete families of closed sets, we have the following result:

\begin{proposition}\label{unionclosed}
For every discrete family $\{A_i\}_{i\in I}$ in a topological space, we have the equality $\tbigcup_{i\in I} \overline{A_i} = \overline{\tbigcup_{i\in I} A_i}$. In particular, the union of a discrete family of closed sets is closed.
\end{proposition}

\begin{proof}
Let $x\in \overline{\tbigcup_{i\in I}A_i}$. Then there is $N\in \mathcal{N}_x$ such that 
$$| \{i\in I \mid A_i \cap N \neq\varnothing\}|\leq1.$$
Since $x\in \overline{\tbigcup_{i\in I} A_i}$ it follows that $N\cap\bigl(\overline{\tbigcup_{i\in I} A_i}\bigr)\ne\varnothing$ and so $| \{i\in I \mid A_i \cap N \neq\varnothing\}|=1$. Let $i_0\in I$ be the index such that $A_{i_0} \cap N \neq \varnothing$ and  $A_i \cap N=\varnothing$ for each $i \neq i_0.$ 
Note that therefore $\left(\tbigcup_{i\neq i_0} A_i \right)\cap N = \varnothing$ and so $x\not\in \overline{\tbigcup_{i\neq i_0}A_i}$. Moreover, since
$$x\in \overline{\tbigcup_{i\in I}A_i} =\overline{\tbigcup_{i \neq i_0} A_i \cup A_{i_0}}=\overline{\tbigcup_{i\neq i_0}A_i}\cup \overline{A_{i_0}},$$  we obtain that $x\in \overline{A_{i_0}} \subseteq \tbigcup_{i\in I} \overline{A_i}$.
The reverse inclusion always holds: for every $i\in I$ one has $A_i \subseteq \tbigcup_{i\in I} A_i$, from which follows that $\overline{A_i} \subseteq \overline{\tbigcup_{i\in I} A_i}$, and thus $\tbigcup_{i\in I} \overline{A_i} \subseteq \overline{\tbigcup_{i\in I} A_i}$.
\end{proof}

\begin{lemma}\label{closdiscrete}
If $\{A_i\}_{i\in I}$ is a discrete family in a topological space, then $\bigl\{\overline{A_i}\bigr\}_{i\in I}$ is also discrete.
\end{lemma}
\begin{proof}
By way of contradiction assume that there exists $x\in X$ such that for every $N\in\mathcal{N}_x$ there are $i\ne j$ in $I$ satisfying $\overline{A_i} \cap N\neq \varnothing$ and $\overline{A_j}\cap N\neq \varnothing$. Now, for each  $N\in \mathcal{N}_x$ take an open subset $U$ with $x\in U \subseteq N$.  We also have $U\in \mathcal{N}_x$, and hence $\overline{A_i} \cap U\neq \varnothing$ and $\overline{A_j}\cap U\neq \varnothing$ for some $i\ne j$ in $I$. Let $y\in \overline{A_i}\cap U$. Since $U$ is open, one has $U\in\mathcal{N}_y$, and since $y\in\overline{A_i}$ (by definition of closure) we get $A_i \cap U\ne \varnothing$ and so $A_i \cap N\ne\varnothing$. Similarly, we have $A_j \cap N\ne \varnothing$,  a contradiction with the discreteness of $\{A_i\}_{i\in I}$. 
\end{proof}

\begin{definition}
A family $\{A_i\}_{i\in I}$ of subsets of a topological space $X$ is said to be \emph{locally finite} if for every $x\in X$ there is $N\in\NN_x$ such that $\{ i\in I \mid  A_i \cap N \ne \varnothing \}$ is finite.
\end{definition}

\begin{lemma}\label{charactDisc}
Let $X$ be a topological space and $\{A_i\}_{i\in I}$ a family of subsets of $X$. Then, $\{A_i\}_{i\in I}$ is a discrete family if and only if it is locally finite and $\overline{A_i}\cap \overline{A_j} =\varnothing$ whenever $i\ne j$ in $I$.
\end{lemma}

\begin{proof}
In the ``only if'' part, local finiteness is clear and the second condition follows from Lemmas \ref{closdiscrete} and \ref{propertiesdisjointdiscrete}. Let us now show the ``if'' part. Assume that $\{A_i\}_{i\in I}$ is locally finite and that $\overline{A_i}\cap \overline{A_j} =\varnothing$ whenever $i\ne j$ in $I$.  Let $x\in X$. Then there is a neighborhood $N\in\NN_x$ such that $\{i\in I\mid A_i\cap N\}=\{i_1,\dots,i_n\}$ is finite. If $x\not\in \overline{A_{i_1}}\cup\cdots \cup \overline{A_{i_n}}$, then $M=N\cap (X\smallsetminus  \overline{A_{i_1}}\cup\cdots \cup \overline{A_{i_n}})$ is a neighborhood of $x$ which does not meet any member of the family $\{A_i\}_{i\in I}$. Assume otherwise $x\in \overline{A_{i_k}}$ for some $k\in\{ 1,\dots,n\}$. By hypothesis, $\overline{A_{i_k}} \subseteq X\smallsetminus \overline{A_{i_j}}$ for all $j\in \{1,\dots,n\}\smallsetminus \{k\}$. Then 
$$M= N \cap \tbigcap_{\substack{j=1\\ j\ne k}}^{n} (X\smallsetminus \overline{A_{i_j}})$$ 
is an neighborhood of $x$ which intersects at most one member of $\{A_i\}_{i\in I}$ (namely $A_{i_{k}}$). 
\end{proof}

Let $(X,\tau_X)$ and $(Y,\tau_Y)$ be topological spaces. Suppose that $\{X_i\}_{i\in I}$ is a cover of $X$ and  take a family $\{f_i \colon X_i \longrightarrow Y\}_{i\in I}$ of continuous mappings. Recall that the maps $\{f_i\}_{i	\in I}$ are said to be \emph{compatible} if ${f_{i}}_{|X_i \cap X_j}={f_{j}}_{|X_i \cap X_j}$ for all $i,j\in I$.  In that case, a mapping $f\colon X\longrightarrow Y$ arises, given by $f(x)=f_i(x)$ where $x\in X_i$. This function is said to be the \emph{combination} of the mappings $\{f_i\}_{i\in I}$.

\begin{remark}
If the family $\{X_i\}_{i\in I}$ is pairwise disjoint, the maps $\{f_i\}_{i\in I}$ are always compatible. In particular, because of Lemma~\ref{propertiesdisjointdiscrete}, if $\{X_i\}_{i\in I}$ is discrete  the maps $\{f_i\}_{i\in I}$ are always compatible.
\end{remark}

The following result is an extension of the Pasting Lemma. It guarantees the continuity of a combined map with respect to a (possibly infinite) discrete family of sets.

\begin{proposition}\label{propcombinedmap}
Let $(X,\tau_X)$ and $(Y,\tau_Y)$ be topological spaces. Suppose that $\{F_i\}_{i\in I}$ is a closed discrete  cover of $X$ and  let $\{f_i \colon F_i \longrightarrow Y\}_{i\in I}$ be a family of continuous  mappings. Then the combined map is continuous. 
\end{proposition}

\begin{proof}
Let $f$ be the combined mapping. We shall prove that $f$ is continuous by showing that inverse images of closed sets are closed. Indeed, let $F\subseteq X$ be closed. Note that
$$f^{-1} (F) =\{ x\in X\mid f(x) \in F\} = \tbigcup_{i\in I} \{x\in F_i \mid f_i(x)\in F\}=\tbigcup_{i\in I} f_i^{-1}(F).$$
By continuity of $f_i$ it follows that $f_i^{-1}(F)$ is closed in $F_i$  for every $i\in I$. Further, for each $i\in I$, since $F_i$ is closed in $X$,  $f_i^{-1}(F)$  is also closed  in $X$. By Proposition~\ref{unionclosed},  $f^{-1}(F)$ is closed in $X$. 
\end{proof}

We will also be interested in certain families consisting of a union of countably many discrete families:

\begin{definition}\label{sigmadiscret}
A family of subsets of a topological space is called $\sigma$-\emph{discrete} if it can be represented as a countable union  of discrete families. 
\end{definition}

\subsection{The Diagonal Theorem}

\begin{definition} Let $(X,\tau_X)$ be a topological space, $\{ (Y_i,\tau_i) \}_{i\in I}$ a family of 
topological spaces and $\{ f_i \colon (X,\tau_X) \longrightarrow (Y_i,\tau_i) \}_{i\in I}$ a family of continuous maps. The map
 \[\begin{array}{rcclll}
 f \colon&X&\longrightarrow&\tbigprod_{i\in I} Y_i &&\cr
 &x&\longmapsto&f(x) = \left\{ f_i(x) \right\}_{i\in I}
 \end{array}\]
 is said to be the \emph{diagonal} of the mappings $\{f_i\}_{i\in I}$, and it is usually denoted by $\Delta_{i\in I} f_i$. 
\end{definition}

\begin{lemma}
 Let $(X,\tau_X)$ be a topological space, $\{ (Y_i,\tau_i) \}_{i\in I}$ a family of 
topological spaces and $\{ f_i \colon (X,\tau_X) \longrightarrow (Y_i,\tau_i) \}_{i\in I}$ a family of continuous maps.  Then the diagonal map $\Delta_{i\in I} f_i$ is continuous.
\end{lemma}

\begin{proof}
Let $\varphi_i\colon \tbigprod_{j\in I} Y_j \longrightarrow Y_i$ denote the $i$th coordinate projection for each $i\in I$. Because of Examples~\ref{exampinitial} and Proposition~\ref{univpropinitial}, $f$ is continuous if and only if $\varphi_i \circ f$ is continuous for every $i\in I$. Now, one has $\varphi_i \circ f = f_i$, which is continuous for all $i\in I$, and the proof is complete.
\end{proof}

\begin{definition}
Let $(X,\tau_X)$ be a topological space, $\{ (Y_i,\tau_i) \}_{i\in I}$ a family of 
topological spaces and $\{ f_i \colon (X,\tau_X) \longrightarrow (Y_i,\tau_i) \}_{i\in I}$ a family of continuous maps.  
\begin{enumerate}[\normalfont (i)]
\item $\{f_i\}_{i\in I}$ is said to \emph{separate points} if for every pair of distinct points $x,y\in X$ there exists an $i\in I$ such that $f_i(x)\ne f_i(y)$.
\item $\{f_i\}_{i\in I}$ is said to \emph{separate points and closed sets} if for every $x\in X$ and for every closed subset $F\subseteq X$ such that $x\not\in F$ there is an $i\in I$ such that $f_i(x)\not\in \overline{f_i(F)}$. 
\end{enumerate}
\end{definition}

\begin{remark}\label{remT1}
Recall that in a $T_1$ space singletons are closed. Thus,  if $X$ is $T_1$, condition (ii) in the previous definition automatically implies condition (i). 
\end{remark}

The following theorem asserts that under certain circumstances the diagonal map is one-to-one or, further,  an embedding.

\begin{theorem}[The diagonal theorem]\label{diagonal} Let $(X,\tau_X)$ be a topological space, $\{ (Y_i,\tau_i) \}_{i\in I}$ a family of 
topological spaces and $\{ f_i \colon (X,\tau_X) \longrightarrow (Y_i,\tau_i) \}_{i\in I}$ a family of continuous maps.  Then,
\begin{enumerate}[\normalfont (i)]
\item If $\{f_i\}_{i\in I}$ separates points, then $\Delta_{i\in I} f_i$ is one-to-one.
\item If $\{f_i\}_{i\in I}$ separates points and also separates points and closed sets, then $\Delta_{i\in I} f_i$ is an embedding, i.e.  it is a homeomorphism onto its image.
\end{enumerate}
\end{theorem}

\begin{proof}
Throughout the proof we shall write $f=\Delta_{i\in I} f_i$.\\[2mm]
(i) If $x\ne y$ in $X$, there is an $i\in I$ with $f_i(x)\ne f_i(y)$, and then $f(x)\ne f(y)$.\\[2mm]
(ii)  By part (i) $f$ is one-to-one, and we already know that $f$ is continuous. Thus it is enough to show that $f\colon (X,\tau_X)\longrightarrow ( f(X), \tau_{f(X)})$ is a closed map, where $(f(X), \tau_{f(X)})$ denotes the subspace topology inherited from the cartesian product. Let $F$ be closed in $X$. Our goal is to show that $f(F)$ is closed in $f(X)$. We will show that
$$f(F) = \overline{f(F)} \cap f(X)$$
where $\overline{f(F)}$ denotes the closure of $f(F)$ in the whole cartesian product $\tbigprod_{i\in I} Y_i$. Clearly, it is enough to show that the right hand side is contained in the left hand side.  
Let $f(x)\in  \overline{f(F)} \cap f(X)$. Denote by $\varphi_i\colon \tbigprod_{j\in I} Y_j\longrightarrow Y_i$ the $i$th coordinate projection. Since $f(x) \in  \overline{f(F)}$,  one has 
$$f_i(x) =\varphi_i (f(x))\in \varphi_i \left(  \overline{f(F)} \right) \subseteq \overline{ \varphi_i \left(  f(F) \right) }=\overline{f_i(F)}$$
for every $i\in I$. Thus, the hypothesis tells us that necessarily $x\in F$, which implies $f(x)\in f(F)$, as desired.
\end{proof}

 \section{The Hedgehog}
 \label{ch2}
 Let $\kappa $ be some cardinal and $I$ be a set with $|I|=\kappa$. Let $\sim$ be an equivalence relation on the product $X=[0,1]\times I$ defined by $(t,i)\sim (s,j)$ if and only if $t=0=s$ or $(t,i)=(s,j)$. The \emph{hedgehog with $\kappa$ spines} $J(\kappa)$ is the set of equivalence classes $X\slash\sim$ of $[0,1]\times I$ under $\sim$.

 Let $p$ denote the quotient map $p\colon X\longrightarrow J(\kappa)$.
 In what follows we shall identify equivalence classes $p(t,i)$ with their representatives $(t,i)$, and we let $\zerohedgehog$ denote the equivalence class $p(0,i)$.

  \begin{figure}[htbp]
\begin{center}
\begin{tikzpicture}
\draw[line width=2pt] (0,0) -- (5,0);
\foreach \x in {1,...,9}
\draw[line width=2pt,rotate=(60/sqrt(\x)-20)] (0,0) -- (5,0);
\filldraw[black,rotate=(40)] (3,0) circle (2pt);
\draw[rotate=(40)] (3,0)  node[anchor=south east] {$(t,i)$};
\filldraw (0,0) circle (2pt);
\draw (0,0)  node[anchor=south east] {$\zerohedgehog$};
\end{tikzpicture}
\caption{The hedgehog}
\label{irudi}
\end{center}
\end{figure}
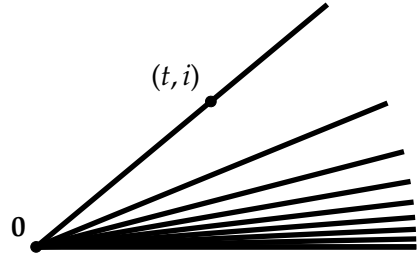

 \subsection{Projections}

Now we introduce  a new family of useful mappings. For each $i \in I$, let $\pi_i \colon J(\kappa)\longrightarrow [0,1]$ the $i$th projection given by
\[\pi_i (t,j)  =\begin{cases}t,&\text{ if  }j=i;\\[0.1cm] 0,&\text{ if  }j\ne i,\cr\end{cases},\qquad (t,j)\in J(\kappa)  \]

We also have the projection $\pi_\kappa \colon J(\kappa)\longrightarrow [0,1]$ given by
\[\pi_\kappa (t,j)  =t,\qquad (t,j)\in J(\kappa)  \]
  It is clear that the equalities
$$\pi_\kappa = \tbigsum_{i\in I} \pi_i = \tbigvee_{i\in I} \pi_i$$
hold (note that for each $(t,j)\in J(\kappa)$ the sum $\tbigsum_{i\in I} \pi_i(t,j)$ has only one nonzero term).

 \subsection{Partial order on $J(\kappa)$}

The hedgehog $J(\kappa)$ can be seen as a partially ordered set. More precisely,  we define a binary relation on $J(\kappa)$ as follows:
 \[
 (t,i)\leq (s,j)\quad\text{ if }(t,i)=\zerohedgehog\quad\text{ or }\quad i=j\text { and } t\leq s.
 \]
 
 It turns out that $\leq$ is a partial order on $J(\kappa)$. We set
$$\newuparrow (t,i)=\{ (s,j)\in J(\kappa) \mid(t,i) \leq (s,j) \}\quad \textrm{and} \quad  \newdownarrow (t,i)=\{ (s,j)\in J(\kappa) \mid(s,j) \leq (t,i) \}.$$ Figure~\ref{partial} shows how the sets $\newdownarrow (t,i)$ and $\newuparrow (t,i)$ look like for $(t,i)\in J(\kappa)$.
   
  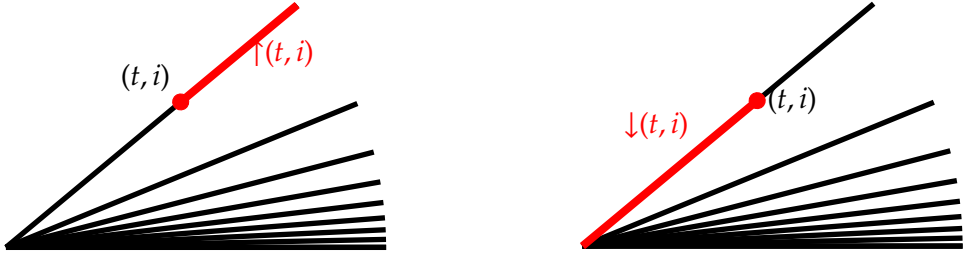
\begin{figure}[htbp]
\begin{center}
\begin{tikzpicture}
\draw[line width=2pt] (0,0) -- (5,0);
\foreach \x in {2,...,9}
\draw[line width=2pt,rotate=(60/sqrt(\x)-20)] (0,0) -- (5,0);
\draw[line width=2pt,rotate=(40)] (0,0) -- (5,0);
\draw[line width=3pt,red,rotate=(40)] (3,0) -- (5,0);
\filldraw[red,rotate=(40)] (3,0) circle (3pt);
\draw[rotate=(40)] (3,0)  node[anchor=south east] {$(t,i)$};
\draw[red,rotate=(40)] (4,0)  node[anchor=west] {$\newuparrow (t,i)$};
\end{tikzpicture}
\hfill
\begin{tikzpicture}
\draw[line width=2pt] (0,0) -- (5,0);
\foreach \x in {2,...,9}
\draw[line width=2pt,rotate=(60/sqrt(\x)-20)] (0,0) -- (5,0);
\draw[line width=2pt,rotate=(40)] (0,0) -- (5,0);
\draw[line width=3pt,red,rotate=(40)] (3,0) -- (0,0);
\filldraw[red,rotate=(40)] (3,0) circle (3pt);
\draw[rotate=(40)] (3,0)  node[anchor=west] {$(t,i)$};
\draw[red,rotate=(40)] (2,0)  node[anchor=south east] {$\newdownarrow (t,i)$};
\end{tikzpicture}
\caption{The partial order}
\label{partial}
\end{center}
\end{figure}

We begin by stating some properties of the poset $(J(\kappa),\leq)$. Recall that a nonempty subset $D$ of a partially ordered set is said to be \emph{directed} if every pair of elements has an upper bound.
  
 \begin{fact}\label{directedfact} A subset $D\subset J(\kappa)$ is directed in $(J(\kappa),\le)$ if and only if $D$ is nonempty and there is an $i_0\in I$ such that $D\subset p([0,1]\times \{i_0\})$, i.e. iff it is directed in one spine.
 \end{fact}
 
 \begin{proof}
 First we show the ``only if'' part. Assume that $D$ is directed. Assume by contradiction that there are $(t,i)$ and $(s,j)$ in $D$ such that $t,s>0$ and $i\neq j$. Let $(r,k)$ be an upper bound of $(t,i)$ and $(s,j)$. Since $(t,i),(s,j)\neq \zerohedgehog$, necessarily we have $i=j$, a contradiction. For the converse, assume that $D$ is nonempty and $D\subset p([0,1]\times \{i_0\})$ for an $i_0\in I$. Let $(t,i_0),(s,i_0)\in D$. Then $(t\vee s, i_0)$ in an upper bound in $D$. 
 \end{proof}

\begin{definition}{\rm(Cf.\ \cite[page~9, Definition~0-2.1\,(iv)]{GierzDom})}
A partially ordered set is said to be a \emph{complete semilattice} if every nonempty subset has an infimum and every directed subset has a supremum.
\end{definition}

 \begin{fact} The partially ordered set $(J(\kappa),\le)$ is a complete semilattice.
 \end{fact}
 
 \begin{proof}
 First we show that every directed subset has a sup. Let $D$ be a directed subset. By Fact~\ref{directedfact} it follows that there is an $i_0\in I$ such that $D\subset p([0,1]\times \{i_0\})$. Set $d_0= \tbigvee \{ d \mid (d,i_0)\in D\}$ (the sup is taken in the complete lattice $[0,1]$). Then $(d_0, i_0)$ is the sup of $D$ in $(J(\kappa),\leq)$.  
 
 We now prove that every nonempty subset has an inf. Let $S\subset J(\kappa)$ a nonempty subset. We distinguish two cases. If there is $i_0\in I$ such that $S\subseteq p([0,1]\times \{i_0\})$, let $s_0= \tbigwedge\{ s \mid (s,i_0)\in S\}$. Then $(s_0,i_0)$ is the inf of $S$. Indeed, it is clear that $(s_0,i_0)$ is a lower bound of $S$. Further, assume that $(s_1,i_1)$ is another lower bound of $S$. If $i_1=i_0$, one has $s_1 \in \{ s \mid (s,i_0)\in S\}$ and thus $s_1\leq s_0$ and $(s_1,i_1)\leq (s_0,i_0)$. Suppose now that $i_1\neq i_0$.  Take $(s,i_0)\in S$ ($S$ is nonempty). Since $(s_1,i_1)$ is a lower bound, we have $(s_1,i_1)\leq (s,i_0)$. Thus, $(s_1,i_1)=\zerohedgehog \leq (s_0,i_0)$.  Finally assume that there are two elements $(s,i),(s',i')$ in $S$ such that $i\neq i'$. Clearly, $\zerohedgehog$ is an inf of $S$. 
 \end{proof}
 \subsection{The hedgehog as a subset of the cube $[0,1]^{I}$}\label{axescube}
 
We now give an alternative description of the poset $(J(\kappa),\leq)$. More precisely, we show that $J(\kappa)$ is order-isomorphic to the axes of the cube $[0,1]^I$
 \[L(\kappa)=\underset{i\in I}\tbigcup\, \left\{\varphi\in [0,1]^{I}\mid \varphi(j)=0\quad \forall j\ne i\right\}
 \]
endowed with the componentwise order inherited from $[0,1]^{I}$, i.e.
\[\varphi\le \psi\text{ in }[0,1]^{I} \iff \varphi(i)\le \psi(i)\text{ for each } i\in I.\] 
 \begin{fact}\label{orderisom} The hedgehog $(J(\kappa),\le)$ is order-isomorphic to $(L(\kappa),\le)$.
  \end{fact}
 
 \begin{proof}
  Define an order-isomorphism $\Phi\colon J(\kappa)\longrightarrow L(\kappa)$ by
\[\begin{array}{clll}
 &\Phi(t,i)\colon&I\longmapsto &[0,1]\cr
 &&j\longmapsto &\Phi(t,i)(j)=\begin{cases}t,&\text{ if  }j=i;\\[0.1cm] 0,&\text{ if  }j\ne i;\cr\end{cases} 
 \end{array}\]
for each $(t,i)\in J(\kappa)$. Note that $\Phi$ is well defined, that is, it does not depend on the representatives chosen. Let us begin by proving that $\Phi$ is order-preserving. Let $(t,i) \leq (s,j)$ in $J(\kappa)$. Assume first that $i=j$ and $t\leq s$.  Let $k\in I$. If $k=i=j$, then $\Phi(t,i)(k) = t \leq s = \Phi(s,j)(k)$. If $k\neq i,j$, we have $\Phi(t,i)(k)=0 = \Phi(s,j)(k)$. Thus $\Phi(t,i)\leq \Phi(s,j)$.  Assume now that $(t,i)=\zerohedgehog$, i.e. that $t=0$. Then it is clear that $\Phi(t,i)(k)=0 \leq \Phi(s,j)(k)$ for all $k\in I$, that is, $\Phi(t,i)\leq \Phi(s,j)$. Define a new map $\Psi$ as follows:
   \[\begin{array}{rcclll}
 \Psi \colon&L(\kappa)&\longrightarrow&J(\kappa)&&\\
 &\varphi&\longmapsto&\Psi(\varphi)=\begin{cases}(\varphi(i),i),&\text{ if  }\varphi(i)\ne 0;\\[0.1cm] \zerohedgehog,&\text{ if  }\varphi(i)=0\quad \forall i\in I.\cr\end{cases} \end{array}\]
 We will show that $\Phi$ and $\Psi$ are mutually inverse, i.e. that $\Phi \circ \Psi =\text{\rm id}_{J(\kappa)}$ and  $\Phi\circ \Psi =\text{\rm id}_{L(\kappa)}$. We will only prove the former equality, because the latter may be proved similarly.  Let $(t,i)\in J(\kappa)$.  If $t = 0$, we have $\Phi(t,i)(j)=0$ for all $j\in I$. Hence, $(\Psi \circ \Phi)(t,i)=\Psi( \Phi(t,i) )=\zerohedgehog=(t,i).$ Now assume that $t\ne 0$. It follows that $\Phi(t,i)(i)=t\ne 0$, and so we obtain $(\Psi \circ \Phi)(t,i)=\Psi( \Phi(t,i) )=(\Phi(t,i)(i),i)=(t,i)$, as desired. 
 
 Finally we show that $\Psi= \Phi^{-1}$ is also order-preserving. Let $\varphi \leq \psi$ in $L(\kappa)$. Suppose first that $\varphi (i) = 0$ for all $i\in I$. Then, $\Psi(\varphi)= \zerohedgehog \leq \Psi(\psi)$. Assume otherwise that $\varphi(i)\neq 0$ (note that we also have $\psi(i)\neq 0$ because $\varphi(i)\leq \psi(i)$). Then $\Psi(\varphi)= (\varphi(i),i) \leq (\psi(i),i)=\Psi(\psi)$.
 \end{proof}

\section{The Quotient Hedgehog}
\label{ch3}
We consider the set $X=[0,1]\times I$ endowed with the product topology of the usual topology on $[0,1]$ and the discrete topology on $I$.
 In view of the description of the hedgehog with $\kappa$ spines as a quotient defined on the product  $X=[0,1]\times I$ (with $|I|=\kappa$), it is natural to consider the quotient topology on it. The \emph{quotient hedgehog with $\kappa$ spines} is the quotient space with respect to this topology.

One may ask whether it is true that the quotient mapping $p\colon X \longrightarrow J(\kappa)$  is closed or open. Firstly, it turns out that $p$ is not open. Indeed,  choose an $i_0\in I$. Then $U=[0,1]\times \{i_0\}$ is open in $X$ but $p\left(U \right)$ is not open in $J(\kappa)$ because $p^{-1}(p(U)) =\left( [0,1]\times \{i_0\} \right) \cup \left( \{0\}\times I\right) $ is not open in $X$. Thus $p$ fails to be open. On the contrary, the quotient mapping $p$ is always closed.

\begin{fact}\label{projclosed}
The quotient map $p\colon X \longrightarrow J(\kappa)$ is closed.
\end{fact}

\begin{proof} Let $F$ be closed in $X$. We want to see that $p(F)$ is closed in $J(\kappa)$, i.e. that $p^{-1}(p(F))$ is closed in $X$. We have that $p^{-1}(p(F))$ is the union of the equivalence classes intersecting $F$, and thus, 
$$p^{-1}(p(F)) = \begin{cases}
F & \textrm{if } (0,i)\not\in F \text{ for all } i\in I;\\
F\cup \left(\{0\}\times I\right) & \textrm{otherwise.}\end{cases}$$
In both cases, it is clear that $p^{-1}(p(F))$ is closed (in the last case because it is the union of two closed subsets). 
\end{proof}

Once that we have defined the quotient topology on hedgehog, we can give a subbasis. 

\begin{fact}\label{subbasequot}
The family 
 \[
 {\mathcal{S}}(\kappa) =\bigl\{\underset{i\in I}\tbigcup\, p\left(\left[0,t_{i}\right)\times\{i\}\right)\mid \{t_{i}\}_{i\in I}\in (0,1]^{I}\bigr\}\cup\bigl\{p\left((s,1]\times \{i\}\right)\mid s<1,i\in I\bigr\}
 \] 
is a subbase of the quotient hedgehog $J(\kappa)$. 
\end{fact}

\begin{proof}
We have to show that $\mathcal{S}(\kappa)$ is a family of open sets whose finite intersections form a basis of the quotient topology on $J(\kappa)$. Note that a basis of the product $X=[0,1]\times I$ is given by
\begin{align*} \beta &= \bigl\{ (a,b)\times \{i\} \mid 0\leq a < b \leq 1, i\in I\bigr\}\\& \cup \bigl\{ [0,b)\times\{i\} \mid 0<b\leq 1,i\in I\bigr\} \cup \bigl\{ (a,1]\times\{i\} \mid 0\leq a< 1,i\in I\bigr\}.\end{align*}

Let $S=\underset{i\in I}\tbigcup\, p\left(\left[0,t_{i}\right)\times\{i\}\right) \in \mathcal{S}(\kappa)$ with $\{t_i\}_{i\in I} \in (0,1]^I$.  We have 
$$p^{-1} (S)= \underset{i\in I}\tbigcup\,\left[0,t_{i}\right)\times\{i\},$$ which is open in $X$. Thus $S$ is open in $J(\kappa)$. Take now $S=p\left((s,1]\times \{i\}\right) \in \mathcal{S}(\kappa)$. It turns out that $p^{-1}(S)=(s,1]\times \{i\}$ which is open in $X$, i.e. $S$ is open in $J(\kappa)$.  Hence $\mathcal{S}(\kappa)$ is a family of open sets.

Let $U\subseteq J(\kappa)$ be open in the quotient hedgehog, that is, $p^{-1} (U)$ is open in $X$. Let $(t,i)\in U$. We distinguish two cases. First, assume that $(t,i)\neq \zerohedgehog$. If $t<1$, Since $(t,i)\in p^{-1}(U) \subseteq X $ and $\beta$ is a basis of $X$, we have $(t,i)\in (a,b)\times \{i\} \subseteq p^{-1}(U)$ for some $0\leq a < b \leq 1$. It follows that  $(t,i) \in p\left( (a,b)\times \{i\} \right)  \subseteq U$.
Note that
\[p\left( (a,b)\times \{i\} \right) = \tbigcup_{j\in I} p\left( [0,b)\times \{j\} \right) \cap p\left((a,1]\times \{i\}\right)\]  
is a finite intersection of elements in $\mathcal{S}(\kappa)$. If $t=1$, since $\beta$ is a basis, one has $(t,i)\in (a,1]\times\{i\} \subseteq p^{-1}(U)$ for some $0\leq a <1$, and so $(t,i)\in p((a,1]\times\{i\}) \subseteq U$, where $p((a,1]\times\{i\})\in\mathcal{S}(\kappa)$, as desired.  Assume now that $(t,i)=\zerohedgehog$. Then one has $(0,i)\in p^{-1}(U)\subseteq X$ for every $i\in I$. Since $\beta$ is a basis, for each $i\in I$ there is $b_i$ such that $0< b_i \leq 1$ and $(0,i)\in [0,b_i) \times \{i\} \subseteq p^{-1}(U)$. Thus $\zerohedgehog \in p([0,b_i)\times \{i\})\subseteq  U$ for all $i\in I$, which implies
$$\zerohedgehog \in \tbigcup_{i\in I} p([0,b_i)\times \{i\})\subseteq  U.$$
Note that $\tbigcup_{i\in I} p([0,b_i)\times \{i\}) \in \mathcal{S}(\kappa)$. Thus, $\mathcal{S}(\kappa)$ is a subbase of $J(\kappa)$. 
\end{proof}

 \begin{figure}[htbp]
\begin{center}
\begin{tikzpicture}
\draw[line width=2pt] (0,0) -- (5,0);
\foreach \x in {2,...,9}
\draw[line width=2pt,rotate=(60/sqrt(\x)-20)] (0,0) -- (5,0);
\draw[line width=3pt,red, rotate=(60/sqrt(2)-20)] (0,0) -- (2,0);
\draw[line width=3pt,red, rotate=(60/sqrt(3)-20)] (0,0) -- (4,0);
\draw[line width=3pt,red, rotate=(60/sqrt(4)-20)] (0,0) -- (3.25,0);
\draw[line width=3pt,red, rotate=(60/sqrt(5)-20)] (0,0) -- (4.25,0);
\draw[line width=3pt,red, rotate=(60/sqrt(6)-20)] (0,0) -- (3.5,0);
\draw[line width=3pt,red, rotate=(60/sqrt(7)-20)] (0,0) -- (4,0);
\draw[line width=3pt,red, rotate=(60/sqrt(8)-20)] (0,0) -- (3,0);
\draw[line width=3pt,red, rotate=(60/sqrt(9)-20)] (0,0) -- (4.5,0);
\draw[line width=2pt,rotate=(40)] (0,0) -- (5,0);
\draw[line width=3pt,red,rotate=(40)] (0,0) -- (4.2,0);
\filldraw[red,rotate=(40)] (4.2,0) circle (3pt);
\filldraw[red,rotate=(60/sqrt(2)-20)] (2,0) circle (3pt);
\filldraw[red,rotate=(60/sqrt(3)-20)] (4,0) circle (3pt);
\filldraw[red,rotate=(60/sqrt(4)-20)] (3.25,0) circle (3pt);
\filldraw[red,rotate=(60/sqrt(5)-20)] (4.25,0) circle (3pt);
\filldraw[red,rotate=(60/sqrt(6)-20)] (3.5,0) circle (3pt);
\filldraw[red,rotate=(60/sqrt(7)-20)] (4,0) circle (3pt);
\filldraw[red,rotate=(60/sqrt(8)-20)] (3,0) circle (3pt);
\filldraw[red,rotate=(60/sqrt(9)-20)] (4.5,0) circle (3pt);
\filldraw[white,rotate=(40)] (4.2,0) circle (2pt);
\filldraw[white,rotate=(60/sqrt(2)-20)] (2,0) circle (2pt);
\filldraw[white,rotate=(60/sqrt(3)-20)] (4,0) circle (2pt);
\filldraw[white,rotate=(60/sqrt(4)-20)] (3.25,0) circle (2pt);
\filldraw[white,rotate=(60/sqrt(5)-20)] (4.25,0) circle (2pt);
\filldraw[white,rotate=(60/sqrt(6)-20)] (3.5,0) circle (2pt);
\filldraw[white,rotate=(60/sqrt(7)-20)] (4,0) circle (2pt);
\filldraw[white,rotate=(60/sqrt(8)-20)] (3,0) circle (2pt);
\filldraw[white,rotate=(60/sqrt(9)-20)] (4.5,0) circle (2pt);
\end{tikzpicture}
\hfill
\begin{tikzpicture}
\draw[line width=2pt] (0,0) -- (5,0);
\foreach \x in {2,...,9}
\draw[line width=2pt,rotate=(60/sqrt(\x)-20)] (0,0) -- (5,0);
\draw[line width=2pt,rotate=(40)] (0,0) -- (5,0);
\draw[line width=3pt,red,rotate=(40)] (3,0) -- (5,0);
\filldraw[red,rotate=(40)] (3,0) circle (3pt);
\filldraw[white,rotate=(40)] (3,0) circle (2pt);
\draw[rotate=(40)] (3,0)  node[anchor=south east] {$(t,i)$};
\end{tikzpicture}
\caption{Subbase of $J(\kappa)$}
\label{picsubq}
\end{center}
\end{figure}
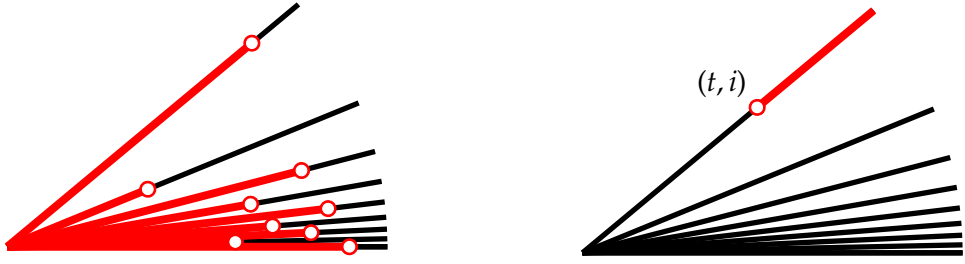
A base of neighborhoods of $\zerohedgehog$ (see Figure \ref{basisnbqu}) is precisely given by
$$\mathcal{B}_{\zerohedgehog}(\kappa) =\bigl\{ \underset{i\in I}\tbigcup\,  p\left( \left[0,t_i \right)\times \{i\} \right) \mid \{t_i\}_{i\in I} \in (0,1]^I \bigr\}.$$

 \begin{figure}[htbp]
\begin{center}
\begin{tikzpicture}
\draw[line width=2pt] (0,0) -- (5,0);
\foreach \x in {2,...,9}
\draw[line width=2pt,rotate=(60/sqrt(\x)-20)] (0,0) -- (5,0);
\draw[line width=3pt,red, rotate=(60/sqrt(2)-20)] (0,0) -- (2,0);
\draw[line width=3pt,red, rotate=(60/sqrt(3)-20)] (0,0) -- (4,0);
\draw[line width=3pt,red, rotate=(60/sqrt(4)-20)] (0,0) -- (3.25,0);
\draw[line width=3pt,red, rotate=(60/sqrt(5)-20)] (0,0) -- (4.25,0);
\draw[line width=3pt,red, rotate=(60/sqrt(6)-20)] (0,0) -- (3.5,0);
\draw[line width=3pt,red, rotate=(60/sqrt(7)-20)] (0,0) -- (4,0);
\draw[line width=3pt,red, rotate=(60/sqrt(8)-20)] (0,0) -- (3,0);
\draw[line width=3pt,red, rotate=(60/sqrt(9)-20)] (0,0) -- (4.5,0);
\draw[line width=2pt,rotate=(40)] (0,0) -- (5,0);
\draw[line width=3pt,red,rotate=(40)] (0,0) -- (4.2,0);
\filldraw[red,rotate=(40)] (4.2,0) circle (3pt);
\filldraw[red,rotate=(60/sqrt(2)-20)] (2,0) circle (3pt);
\filldraw[red,rotate=(60/sqrt(3)-20)] (4,0) circle (3pt);
\filldraw[red,rotate=(60/sqrt(4)-20)] (3.25,0) circle (3pt);
\filldraw[red,rotate=(60/sqrt(5)-20)] (4.25,0) circle (3pt);
\filldraw[red,rotate=(60/sqrt(6)-20)] (3.5,0) circle (3pt);
\filldraw[red,rotate=(60/sqrt(7)-20)] (4,0) circle (3pt);
\filldraw[red,rotate=(60/sqrt(8)-20)] (3,0) circle (3pt);
\filldraw[red,rotate=(60/sqrt(9)-20)] (4.5,0) circle (3pt);
\filldraw[white,rotate=(40)] (4.2,0) circle (2pt);
\filldraw[white,rotate=(60/sqrt(2)-20)] (2,0) circle (2pt);
\filldraw[white,rotate=(60/sqrt(3)-20)] (4,0) circle (2pt);
\filldraw[white,rotate=(60/sqrt(4)-20)] (3.25,0) circle (2pt);
\filldraw[white,rotate=(60/sqrt(5)-20)] (4.25,0) circle (2pt);
\filldraw[white,rotate=(60/sqrt(6)-20)] (3.5,0) circle (2pt);
\filldraw[white,rotate=(60/sqrt(7)-20)] (4,0) circle (2pt);
\filldraw[white,rotate=(60/sqrt(8)-20)] (3,0) circle (2pt);
\filldraw[white,rotate=(60/sqrt(9)-20)] (4.5,0) circle (2pt);
\end{tikzpicture}
\caption{Open neighborhood of $\zerohedgehog$}
\label{basisnbqu}
\end{center}
\end{figure}

 This topological space (when $|I|=\aleph_{0}$) is one of the easiest examples of a quotient of a first countable space which is not first countable.

\begin{proposition}\label{NotFirst}
The quotient hedgehog $J(\kappa)$ is first countable if and only if $\kappa < \aleph_{0}$. 
\end{proposition}
\begin{proof}
$\Rightarrow$) We will prove that $J(\kappa)$ is not be first countable whenever $\kappa \geq \aleph_0$. First we show the case where $I=\N$.  By way of contradiction, assume that $\zerohedgehog$ has a countable basis of neighborhoods, say $\{N_n\}_{n\in\N}$. If we define
$$B(t)=\tbigcup_{k\in \N} p\left( \left[0,t_k \right)\times \{k\} \right)$$ with $t=\{t_k\}_{k\in \N} \in (0,1]^\N$, we know that $\mathcal{B}_{\zerohedgehog}=\{ B(t) \mid t \in (0,1]^\N\}$ is a basis of neighborhoods of $\zerohedgehog$.

For each $n\in \N$, there is $t^n =\{t_{k}^{n}\}_{k\in\N}\in (0,1]^\N$ such that $B(t^n) \subseteq N_n$, because $\mathcal{B}_{\zerohedgehog}$ is a  basis of neighborhoods of $\zerohedgehog$ and $N_n \in \NN_{\zerohedgehog}$.  We now construct a new sequence $t=\{ t_n \}_{n\in \N}\in (0,1]^\N$ as follows: for each $n\in \N$ set $t_n= t_{n}^{n} / 2$.

Since $B(t) \in \NN_{\zerohedgehog}$ and $\{N_n\}_{n\in \N}$ is a basis of neighborhoods of $\zerohedgehog$, there exists $n_0\in\N$ such that $B(t^{n_0})\subseteq N_{n_0} \subseteq B(t)$. Note that in the $n_0$th spine we have 
$$p([0,t_{n_0}^{n_0}) \times \{n_0\})\subseteq p([0,t_{n_0}^{n_0}/2)\times \{n_0\}),$$
which is impossible. Thus such countable basis cannot exist.

For the general case, assume that $\kappa \geq \aleph_0$. Clearly, $J(\aleph_0)$ is embedded in $J( \kappa )$. Now, if $J( \kappa)$ were first countable,  so would be $J(\aleph_0) $ (because being first countable is a hereditary property), which contradicts what already has been proved.\\[2mm]
$\Leftarrow$) Assume now that $\kappa =k< \aleph_0$, i.e. that $I=\{i_1,\dots i_k\}$ is finite. Then, for every
 $(t,i)$ in $J(k)$  the family
$$\mathcal B_{(t,i)}= 
\begin{cases} \bigl\{ \tbigcup_{i=1}^{k} p([0,\frac{1}{n} )\times \{i\} ) \mid  n \in \N \bigr\} & \textrm{if } t=0,
\\[1mm]
  \bigl\{ p((t-\frac{1}{n}, t+\frac{1}{n})\times\{i\})  \mid n\in \N, \frac{1}{n} \leq t \le 1-\frac{1}{n} \bigr\} & \textrm{otherwise,}\end{cases}$$
is a countable basis of neighborhoods of $(t,i)$. 
\end{proof}

In what follows, we outline the most important topological properties of the quotient hedgehog.

\begin{properties}
(1) The quotient hedgehog is Hausdorff. Indeed, let $(t,i)\neq (s,j)$ in $J(\kappa)$.  Assume first that  $(t,i),(s,j)\neq \zerohedgehog$. If $i\neq j$, the sets $U=p\left((t/2,1]\times\{i\}\right)$ and $V=p\left((s/2,1]\times \{j\}\right)$ are open and disjoint  and satisfy $(t,i)\in U$ and $(s,j)\in V$.  If $i=j$, set 
$$r=\min \bigl\{ |t-s|/2, t, 1-t, s, 1-s\bigr\}.$$
Then, if $t,s<1$, each of the open and disjoint sets $U=p\left((t-r,t+r)\times\{i\}\right)$ and $V=p\left((s-r,s+r)\times\{i\}\right)$  contains one of the points. The case where $t=1$ or $s=1$ may be shown similarly. Finally, assume that $(t,i)=\zerohedgehog$.  We have that $U=p([0,s/2)\times I)$ and $V=p\left((s/2,1]\times\{j\}\right)$ are the desired open subsets.\\[3mm]
(2) We have proved that $J(\kappa)$ is not first countable whenever $\kappa \geq \aleph_0$. Since  second countability implies first countability and metrizability implies first countability, we deduce that $J(\kappa)$ is neither second countable nor met\-rizable whenever $\kappa \geq \aleph_0$.\\[3mm]
(3) The quotient hedgehog is a normal space. Note that $X=[0,1]\times I$ is metrizable because so is each of the two factors ($[0,1]$ and $I$ are endowed with the usual and discrete topologies, respectively). In particular, $X$ is normal. By virtue of Fact \ref{projclosed}, one has that $J(\kappa)$ is a continuous image of a normal space under a closed map; and, therefore, it is also normal.\\[3mm]
(4) Since $J(\kappa)$ is normal and Hausdorff (in particular $T_1$) it follows that the quotient hedgehog is also regular.\\[3mm]
(5) Combining the previous paragraphs with Fact~\ref{projclosed},  we have that the quotient hedgehog with infinitely many spines is an example of a \emph{La\v{s}nev space} (that is, the image of a met\-rizable space under a closed map) which is not met\-rizable.\\[3mm]
 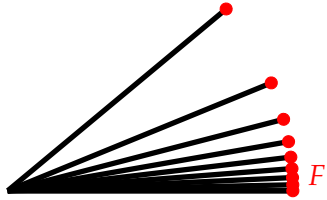
\begin{figure}[htbp]
\begin{center}
\begin{tikzpicture}[scale=0.75]
\draw[line width=2pt] (0,0) -- (5,0);
\foreach \x in {2,...,9}
\draw[line width=2pt,rotate=(60/sqrt(\x)-20)] (0,0) -- (5,0);
\draw[line width=2pt,rotate=(40)] (0,0) -- (5,0);
\filldraw[red,rotate=(40)] (5,0) circle (3pt);
\filldraw[red,rotate=(60/sqrt(2)-20)] (5,0) circle (3pt);
\filldraw[red,rotate=(60/sqrt(3)-20)] (5,0) circle (3pt);
\filldraw[red,rotate=(60/sqrt(4)-20)] (5,0) circle (3pt);
\filldraw[red,rotate=(60/sqrt(5)-20)] (5,0) circle (3pt);
\filldraw[red,rotate=(60/sqrt(6)-20)] (5,0) circle (3pt);
\filldraw[red,rotate=(60/sqrt(7)-20)] (5,0) circle (3pt);
\filldraw[red,rotate=(60/sqrt(8)-20)] (5,0) circle (3pt);
\filldraw[red,rotate=(60/sqrt(9)-20)] (5,0) circle (3pt);
\draw[red,rotate=(60/sqrt(7)-20)] (5.75,0)  node[anchor=east] {$F$};
\end{tikzpicture}
\caption{The discrete subspace $F$}
\label{irudidisk}
\end{center}
\vskip-3mm
\end{figure}
(6) We now show that $J(\kappa)$ is not compact whenever $\kappa \geq \aleph_0$.
 Assume otherwise that $J(\kappa)$ is compact. Let $U=p([0,1)\times I)$, which is a (subbasic) open of $J(\kappa)$. Thus $F=J(\kappa)\smallsetminus U$ is closed, and being a closed set in a compact, we conclude that $F$ is compact.  Note that $F$ inherits the discrete topology from $J(\kappa)$ (see Figure~\ref{irudidisk}). Since $\kappa \geq \aleph_0$, we then have a discrete infinite compact space, a contradiction (recall that a discrete space is compact iff it is finite).\\[3mm]
(7)  The quotient hedgehog is arcwise connected (in particular connected). Indeed, we have
$$J(\kappa) = \{ \zerohedgehog \} \cup  \tbigcup_{i\in I} p\left( [0,1]\times\{i\} \right)$$
where $\zerohedgehog$ is arcwise connected (it is a singleton) and each subset $p\left([0,1]\times \{i\}\right)$ is also arcwise connected (because they are all homeomorphic to the closed unit interval). Since $\zerohedgehog$ is in the intersection of all of them, it follows that $J(\kappa)$ is arcwise connected.\\[3mm]
(8) Note that $J(1) \cong J(2)$. Nevertheless, $J(\kappa) \not\cong J(\lambda)$ for all cardinalities $\kappa \ne \lambda$ with $\lambda$ or $\kappa$ greater than $2$. Indeed, without loss of generality assume that $\kappa < \lambda$ and note that $J(\lambda) \smallsetminus \left\{\zerohedgehog\right\}$ has $\lambda$ connected components. However, $J(\kappa)$ with a point removed has $\kappa$ or $2$ connected components, depending on whether we remove the point $\zerohedgehog$ or some other point. Thus, $J(\kappa)$ and $J(\lambda)$ cannot be homeomorphic.\\[3mm]
(9) If $\kappa \geq \aleph_0$, we shall prove that $J(\kappa)$ fails to be locally compact at $\zerohedgehog$. Note that this shows that a continuous image of a locally compact space is not necessarily locally compact. By way of contradiction, suppose that $N$ is a compact neighborhood of $J(\kappa)$. Then there is a $\{t_i\}_{i\in I}\in (0,1]^I$ such that 
$$\tbigcup_{i\in I} p\left( [0,t_i /2] \times \{i\} \right) \subseteq \tbigcup_{i\in I} p\left( [0,t_i) \times \{i\} \right)\subseteq N.$$ 
Note that $\tbigcup_{i\in I} p\left( [0,t_i /2] \times \{i\} \right)$ is compact because it is a closed subset in a compact.  Since all the intervals $[0,t_i/2]$ are homeomorphic to $[0,1]$, one has that $\tbigcup_{i\in I} p\left( [0,t_i /2] \times \{i\} \right)$ is homeomorphic to  $J(\kappa)$, which is not compact, a contradiction.\\[3mm]
(10) The quotient hedgehog $J(\kappa)$ is separable if and only if $\kappa \leq \aleph_0$.  Assume first that $\kappa\leq \aleph_0$. Then  $X=[0,1]\times I$ has the countable dense subset $X=\left( [0,1]\cap \Q \right) \times I$.  Since the projection map is continuous $p$ and $X$ is separable, from Lemma~\ref{lemmasep} we conclude that $f(X)=J(\kappa)$ is separable.

For the converse, suppose that $\kappa >\aleph_0$. If $J(\kappa)$ were separable, we would have a countable dense subset $D$. Then for each $i\in I$  choose the neighborhood $N_i=p\left( (1/2, 1]\times \{i\} \right)$ of the point $(1,i)\in J(\kappa)$. Then $N_i$ would intersect $D$ in some point, say $d_i$. It follows that $\{d_i\}_{i\in I}$ is an uncountable collection in $D$, against our assumption.\\[3mm]
(11) The quotient hedgehog is a Fr\'echet--Urysohn space. Indeed, let $A\subseteq J(\kappa)$ and $(t,i) \in \overline{A}$. First we deal with the case $t\ne0$. One has 
$$(t,i) \in \overline{A\cap p \left( (0,1)\times\{i\} \right)},$$ for otherwise we would have $U \cap A\cap p \left( (0,1)\times\{i\} \right) =\varnothing$ for some basic open neighborhood $U=p\left( ( t-\varepsilon,t+\varepsilon)\times\{i\} \right)$, and thus $U \cap A=\varnothing$, which contradicts that $(t,i)\in \overline{A}$. Since $p \left( (0,1)\times\{i\} \right)$ is homeomorphic to $((0,1),\tau_u)$ (which is Fr\'echet--Urysohn), it is Fr\'echet--Urysohn. Now, since  
$$(t,i) \in \overline{A\cap p \left( (0,1)\times\{i\} \right)},$$ one deduces that $(t,i)$ is contained in the closure of $A\cap p \left( (0,1)\times\{i\} \right)$ in $p \left( (0,1)\times\{i\} \right)$.  Thus there is a sequence of points $\{x_n\}_{n\in\N}$ in $A\cap p \left( (0,1)\times\{i\} \right)$ converging to $(t,i)$ in $p \left( (0,1)\times\{i\} \right)$ (and therefore in $J(\kappa)$). 

Now suppose that $(t,i)=\zerohedgehog$. First, we shall show that there is an $i_0\in I$ such that 
$$\zerohedgehog \in \overline{A\cap p\left( [0,1]\times\{i_0\} \right)}.$$
 Indeed, by way of contradiction assume that for every $i\in I$ one has $U_i\cap A\cap  p\left( [0,1]\times\{i\} \right)=\varnothing$ for some basic open $U_i =  \underset{j\in I}\tbigcup\,  p( [0,t^{i}_{j} )\times \{j\} )$, where $\{t^{i}_{j}\}_{j\in I}\in (0,1]^I$. Then for every $i\in I$, we have $A\cap p\left([0,t_{i}^{i}) \times\{i\}\right) = \varnothing$. Let
$$U= \tbigcup_{i\in I} p\left([0,t_{i}^{i}) \times\{i\}\right).$$ 
We have that $U$ is an open neighborhood of $\zerohedgehog$ which does not intersect $A$, a contradiction.   
Thus there is an $i_0\in I$ with 
$$\zerohedgehog \in \overline{A\cap p\left( [0,1]\times\{i_0\} \right)}.$$ Now, $p\left( [0,1]\times\{i_0\}\right)$ is homeomorphic to $([0,1],\tau_u)$, and consequently it is Fr\'echet--Urysohn. Since $\zerohedgehog$ is contained in the closure of $A\cap p\left( [0,1]\times\{i_0\}\right)$ in $p\left( [0,1]\times\{i_0\}\right)$, there is a sequence of points $\{y_n\}_{n\in\N}$ in $A\cap  p\left( [0,1]\times\{i_0\}\right)$ converging to $\zerohedgehog$ in $p\left( [0,1]\times\{i_0\}\right)$, and thus in $J(\kappa)$. 
\end{properties}

\subsection{The quotient hedgehog and the Fr\'echet-Urysohn fan} \label{secfre}
In this subsection we present another well-known topological space and we will show that it is closely related to the quotient hedgehog. Such space is constructed as follows (cf.\ \cite{Enc}). Let
$$S= \bigl\{ \textstyle\frac{1}{n} \mid n\in \N \bigr\} \cup \{0\}$$ 
and consider the product $Y= S\times \N$, where $S$ is provided with the usual topology and $\N$ with the discrete topology. Consider an equivalence relation on $Y$ which identifies all the non-isolated points (i.e. $(t,n) \sim (s,m)$ if and only if $t=0=s$ or $(t,n)=(s,m)$). The quotient set with the quotient topology is said to be the \emph{Fr\'echet--Urysohn fan} and we will denote it by $V(S_0)$ (see Figure~\ref{fan}). In view of the description of the Fr\'echet--Urysohn fan, it is clear that $V(S_0)$ is a subspace of the quotient hedgehog $J(\aleph_0)$.

  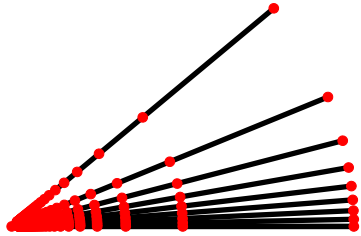
\begin{figure}[htbp]
\begin{center}
\begin{tikzpicture}[scale=0.9]
\draw[line width=2pt] (0,0) -- (5,0);
\foreach \x in {2,...,9}
\draw[line width=2pt,rotate=(60/sqrt(\x)-20)] (0,0) -- (5,0);
\draw[line width=2pt,rotate=(40)] (0,0) -- (5,0);
\foreach \x in {2,...,9}
\foreach \n in {1,...,50}
\filldraw[red,rotate=(60/sqrt(\x)-20)] (5/\n,0) circle (2pt);
\foreach \n in {1,...,50}
\filldraw[red,rotate=(40)] (5/\n,0) circle (2pt);
\filldraw[red,rotate=(40)] (0,0) circle (2pt);
\end{tikzpicture}
\caption{The Fr\'echet--Urysohn fan and the hedgehog}
\label{fan}
\end{center}
\end{figure}

Recall that first countable implies Fr\'echet--Urysohn (see Definition~\ref{defFU}), but the converse is not true in general.  The Fr\'echet--Urysohn fan is the typical example of a (countable, Hausdorff) Fr\'echet--Urysohn space which is not first countable. Indeed, we can easily deduce that $V(S_0)$ is  Fr\'echet--Urysohn, just by noting that $J(\aleph_0)$ is  Fr\'echet--Urysohn and using Proposition~\ref{FUhered}, on the other hand, the argument to conclude that is not first countable is similar to the one for the quotient hedgehog (cf. Proposition~\ref{NotFirst}).
 \subsection{The quotient hedgehog viewed as a subspace of $\bigl([0,1]^{I},\tau_{Box}\bigr)$}
 As we have already proved, there is an order isomorphism between the hedgehog and the following subset of the cube $[0,1]^{I}$:
 \[J(\kappa)\simeq L(\kappa)=\underset{i\in I}\tbigcup\,\left\{\varphi\in [0,1]^{I}\mid \varphi(j)=0\quad \forall j\ne i\right\}.
 \]
Consequently, when $[0,1]^I$ is endowed with an appropriate topology, the quotient hedgehog is homeomorphic to the subspace $L(\kappa)$ of $[0,1]^I$. 
 The adequate topology is the box topology, as we prove in the next lemma.
 
\begin{lemma}\label{homeomBox} The quotient hedgehog $J(\kappa)$ is homeomorphic to $(L(\kappa),\tau_{Box})$.   
\end{lemma}

\begin{proof} The desired homeomorphism is the map $\Phi$ defined in Fact~\ref{orderisom}. We already know that $\Phi$ is bijective, so we are left with the task of showing that  $\Phi$ is continuous and open. For the continuity, let $U= L(\kappa) \cap \tbigprod_{i\in I} U_i$ be an open set in $(L(\kappa),\tau_{Box})$ (where $U_i$ is open in $[0,1]$ for all $i\in I$). We have 
\begin{align*} \Phi^{-1}(U) &= \bigl\{ (t,j) \in J(\kappa) \mid  \Phi(t,j) \in L(\kappa)	\cap  \tbigprod_{i\in I} U_i\bigr\} \\
&= \bigl\{ (t,j) \in J(\kappa) \mid  \Phi(t,j)(i) \in U_i\quad\forall i\in I \bigr\} \\
&= \bigl\{ (t,j) \in J(\kappa) \mid  \Phi(t,j)(i) \in U_i\quad\forall i\ne j \quad \textrm{ and }\quad \Phi(t,j)(j) \in U_j\bigr\} \\
&=\bigl\{ (t,j) \in J(\kappa) \mid  \Phi(t,j)(i) \in U_i\quad\forall i\neq j\bigr\} \cap  \bigl\{ (t,j) \in J(\kappa) \mid  \Phi(t,j)(j) \in U_j\bigr\} \\
&=\bigl\{ (t,j) \in J(\kappa) \mid  0 \in U_i\quad\forall i\neq j\bigr\} \cap  \bigl\{ (t,j) \in J(\kappa) \mid t \in U_j\bigr\} \\
&=\bigl\{ (t,j) \in J(\kappa) \mid  0 \in U_i\quad\forall i\neq j\bigr\} \cap \tbigcup_{j\in I}  p(U_j\times \{j\}). \end{align*}
Let $V=\bigl\{ (t,j) \in J(\kappa) \mid  0 \in U_i\quad\forall i\neq j\bigr\}$. Then
$$V= \begin{cases} 
\varnothing & \textrm{if } 0\not\in U_{i},U_{j} \textrm{ for an } i\neq j;\\
 p([0,1]\times\{j\}) & \textrm{if } 0\not\in U_j \textrm{ for a }j\in I\textrm{ and }  0\in U_i\textrm{ for all } i\neq j;\\
J(\kappa) & \textrm{if } 0\in U_i \textrm{ for all } i\in I.
\end{cases}$$
Therefore, we get
$$\Phi^{-1}(U)= \begin{cases} 
\varnothing & \textrm{if } 0\not\in U_{i},U_{j} \textrm{ for an } i\neq j;\\
p(U_j\times \{j\}) & \textrm{if } 0\not\in U_j \textrm{ for a }j\in I\textrm{ and }  0\in U_i\textrm{ for all } i\neq j;\\
\tbigcup_{j\in I}  p(U_j\times \{j\}) & \textrm{if } 0\in U_j \textrm{ for all } j\in I;
\end{cases}$$
which is open in $J(\kappa)$.

Now we show that $\Phi$ is open, by showing that the images of subbasic open sets of $J(\kappa)$ are open. Recall that
 \[
 {\mathcal{S}}(\kappa) =\bigl\{\underset{i\in I}\tbigcup\, p\left(\left[0,t_{i}\right)\times\{i\}\right)\mid \{t_{i}\}_{i\in I}\in (0,1]^{I}\bigr\}\cup\bigl\{p\left((s,1]\times \{i\}\right)\mid s<1,i\in I\bigr\}
 \] 
is a subbase of $J(\kappa)$. Let $S\in \mathcal{S}(\kappa)$. First assume that $S=\underset{i\in I}\tbigcup\, p\left(\left[0,t_{i}\right)\times\{i\}\right)$ for some $\{t_i\}_{i\in I} \in (0,1]^{I}$. 
Then, 
\begin{align*}
 \Phi(S)&=\Phi\bigl( \tbigcup_{i\in I} p([0,t_i)\times \{i\} ) \bigr) \\& 
=  \tbigcup_{i\in I} \Phi\left(p([0,t_i)\times \{i\} ) \right)= \tbigcup_{i\in I} \bigl\{ \Phi(t,j)\in L(\kappa) \mid (t,j)\in p([0,t_i)\times\{i\}) \bigr\}
 \\& = \tbigcup_{i\in I} \bigl\{ \Phi(t,i)\in L(\kappa)  \mid (t,i)\in p([0,t_i)\times\{i\}) \bigr\} 
 \\& = \tbigcup_{i\in I} \bigl\{ \Phi(t,i)\in L(\kappa)  \mid t\in  [0,t_i)  \bigr\}
 \\& = \tbigcup_{i\in I} \bigl\{ \varphi \in L(\kappa) \mid \varphi(i)=t,\, \varphi(j)=0 \quad \forall j\ne i ,\, t\in [0,t_i) \bigr\}
  \\&= \tbigcup_{i\in I} \bigl\{ \varphi \in L(\kappa) \mid \varphi(i)\in [0,t_i),\, \varphi(j)=0 \quad \forall j\ne i \bigr\}= L(\kappa) \cap \tbigprod_{i\in I} [0,t_i),
 \end{align*}
 which is open in $L(\kappa)$. 
 Assume on the other hand that $S=p\left((s,1]\times \{i\}\right)$ with $s<1$. We have
 \begin{align*} \Phi(S)&=  \bigl\{ \Phi(t,i) \in L(\kappa) \mid  (t,i) \in p\left((s,1]\times \{i\}\right)\bigr\}= 
 \\& =\bigl\{ \varphi \in L(\kappa) \mid \varphi(i)=t,\,\varphi(j)=0 \quad \forall j\ne i,\, t\in (s,1] \bigr\}
 \\& = \bigl\{ \varphi \in L(\kappa) \mid \varphi(i)\in (s,1],\,\varphi(j)=0 \quad \forall j\ne i\bigr\} = L(\kappa) \cap U,
\end{align*}
  where $U= \bigl\{ \varphi \in [0,1]^I \mid \varphi(i) \in (s,1] \bigr\}$ because of the definition of $L(\kappa)$. Set $V_i=(s,1]$ and $V_j=[0,1]$ for all $j\neq i$. Then, $U=\tbigprod_{j\in I} V_j$ is open in $[0,1]^I$ and so $\Phi(S)$ is open in $L(\kappa)$. 
\end{proof}
 
 \section{The Compact Hedgehog}
\label{ch4}
This section is devoted to the study of the second topology on $J(\kappa)$. For this purpose,  recall first that we denote by  $L(\kappa)$ the subset consisting of the axes of the cube $[0,1]^I$ (cf. Subsection~\ref{axescube}). 

\begin{lemma}
The subspace $L(\kappa)$ is closed in $\bigl( [0,1]^I, \tau_{Tych} \bigr)$.
\end{lemma}

\begin{proof}
We shall prove that the complementary
 $$[0,1]^I \smallsetminus L(\kappa)= \{\varphi\in[0,1]^I\mid \varphi(i), \varphi(j)\ne0 \textrm{ for some } i\ne j \textrm{ in } I \}$$
 is open in $\bigl( [0,1]^I, \tau_{Tych} \bigr)$. Indeed, let $\varphi \in [0,1]^I \smallsetminus L(\kappa)$, where $\varphi(i), \varphi(j)\ne0$ for some $i\ne j$ in $I$.  For each $k\ne i,j$ in $I$, set $U_k=[0,1]$ and let $U_i=U_j=(0,1]$. Then $U=\tbigprod_{k\in I}U_k$ is an open set in the product topology containing $\varphi$ and contained in $[0,1]^I\smallsetminus L(\kappa)$. Thus, $[0,1]^I \smallsetminus L(\kappa)$ is open in $\left( [0,1]^I, \tau_{Tych} \right)$.
\end{proof}

\begin{corollary}\label{Lcompact}
The subspace $L(\kappa)$ is compact with the topology inherited from $\bigl( [0,1]^I, \tau_{Tych} \bigr)$.
\end{corollary}

\begin{proof}
By the Tychonoff's theorem, $\left( [0,1]^I, \tau_{Tych} \right)$ is compact, and since $L(\kappa)$ is closed in $\left( [0,1]^I, \tau_{Tych} \right)$, it is also compact.
\end{proof}

Let $( L(\kappa), \tau_{Tych} )$ denote the topology inherited from the product topology of $[0,1]^I$.  Recall that $L(\kappa)$ is in bijection with $J(\kappa)$ via the the mapping $\Phi$ defined in Fact~\ref{orderisom}. Then,  by virtue of Lemma~\ref{bijecthomeom}, one can consider the topology on $J(\kappa)$ which makes the bijection $\Phi$ a homeomorphism. The hedgehog $J(\kappa)$ endowed with such topology will be said to be the \emph{compact hedgehog with $\kappa$ spines}, and we will denote it by
$$\Lambda J(\kappa).$$
Of course, the compact hedgehog is a compact space because of the previous corollary. Now we describe the compact hedgehog by means of a subbase.

\begin{fact}\label{subbasecomp}
The following family is a subbasis of the compact hedgehog $\Lambda J(\kappa)$:
\begin{align*} \hat{\mathcal{S}}(\kappa)&= \bigl\{ J(\kappa) \smallsetminus p([r,1]\times\{i\}) \mid r\in (0,1],i\in I \bigr\}\\&\qquad \qquad \qquad \cup \bigl\{ p\left( (s,1]\times\{i\} \right) \mid s<1 ,i\in I\bigr\}.\end{align*}
\end{fact}

 \begin{figure}[htbp]
\begin{center} 
\begin{tikzpicture}
\draw[line width=2pt] (0,0) -- (5,0);
\foreach \x in {2,...,9}
\draw[line width=2pt,rotate=(60/sqrt(\x)-20)] (0,0) -- (5,0);
\draw[line width=2pt,rotate=(40)] (0,0) -- (5,0);
\draw[line width=3pt,red,rotate=(40)] (0,0) -- (5,0);
\draw[rotate=(40)] (0.2,0)  node[anchor=south east] {$\zerohedgehog$};
\draw[line width=3pt,red, rotate=(60/sqrt(2)-20)] (0,0) -- (2,0);
\draw[line width=3pt,red, rotate=(60/sqrt(3)-20)] (0,0) -- (5,0);
\draw[line width=3pt,red, rotate=(60/sqrt(4)-20)] (0,0) -- (5,0);
\draw[line width=3pt,red, rotate=(60/sqrt(5)-20)] (0,0) -- (5,0);
\draw[line width=3pt,red, rotate=(60/sqrt(6)-20)] (0,0) -- (5,0);
\draw[line width=3pt,red, rotate=(60/sqrt(7)-20)] (0,0) -- (5,0);
\draw[line width=3pt,red, rotate=(60/sqrt(8)-20)] (0,0) -- (5,0);
\draw[line width=3pt,red, rotate=(60/sqrt(9)-20)] (0,0) -- (5,0);
\filldraw[red,rotate=(60/sqrt(2)-20)] (2,0) circle (3pt);
\filldraw[white,rotate=(60/sqrt(2)-20)] (2,0) circle (2pt);
\end{tikzpicture}
\hfill
\begin{tikzpicture}
\draw[line width=2pt] (0,0) -- (5,0);
\foreach \x in {2,...,9}
\draw[line width=2pt,rotate=(60/sqrt(\x)-20)] (0,0) -- (5,0);
\draw[line width=2pt,rotate=(40)] (0,0) -- (5,0);
\draw[line width=3pt,red,rotate=(40)] (3,0) -- (5,0);
\filldraw[red,rotate=(40)] (3,0) circle (3pt);
\filldraw[white,rotate=(40)] (3,0) circle (2pt);
\draw[rotate=(40)] (3,0)  node[anchor=south east] {$(t,i)$};
\end{tikzpicture}
\caption{Subbase of $\Lambda J(\kappa)$}
\label{subbasecompact}
\end{center}
\end{figure}
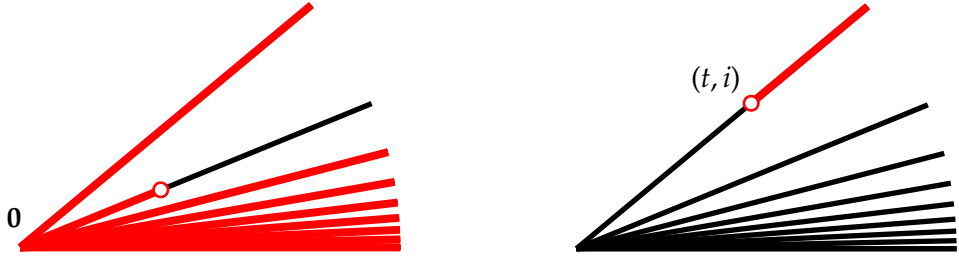

\begin{proof}
A subbase of $\left( L(\kappa), \tau_{Tych} \right)$ is given by 
\begin{align*} \mathcal{S}&=\tbigcup_{i\in I} \bigl\{ L(\kappa) \cap \tbigprod_{j\in I} U_j \mid U_i=[0,b) \textrm{ for some } b\in (0,1] , U_j=[0,1]\quad \forall j\ne i\bigr\}\\&
\cup \tbigcup_{i\in I} \bigl\{ L(\kappa) \cap \tbigprod_{j\in I} U_j \mid U_i=(a,1] \textrm{ for some } a\in [0,1) , U_j=[0,1]\quad \forall j\ne i\bigr\}.
\end{align*}
By Lemma~\ref{bijecthomeom}, a subbasis of $\Lambda J(\kappa)$ is given by 
$$\hat{\mathcal{S}}(\kappa)=\{ \Phi^{-1}(S)\mid S\in \mathcal{S}\}.$$
Let $S\in \mathcal{S}$. First assume that $S=L(\kappa) \cap \tbigprod_{j\in I} U_j$ where $U_i=[0,b)$ and $U_j=[0,1]$ whenever $j\ne i$. In the proof of Lemma~\ref{homeomBox} we have already computed
$$\Phi^{-1}(S)= \tbigcup_{j\in I} p(U_j\times\{j\})=J(\kappa) \smallsetminus p([b,1]\times\{i\}).$$
Similarly, if $S=L(\kappa) \cap \tbigprod_{j\in I} U_j$  with  $U_i=(a,1]$ and $U_j=[0,1]$ for all $ j\ne i$, by the proof of Lemma~\ref{homeomBox} one has
$$\Phi^{-1}(S)=p(U_i\times \{i\} ) = p\left( (a,1]\times\{i\} \right).$$
Hence, the assertion follows.
\end{proof}

In view of the previous fact, we can easily describe a basis of open neighborhoods of $\zerohedgehog$. 

\begin{fact}\label{Basis0comp} The family 
$$\hat{\mathcal{B}_{\zerohedgehog}}(\kappa)=\bigl\{ J(\kappa) \smallsetminus \tbigcup_{j\in J} p([t_j,1]\times\{j\}) \mid J\subseteq I \textrm{ finite, } 0<t_j\leq 1 \bigr\}$$
is a basis of neighborhoods of $\zerohedgehog$ in the compact hedgehog. 
\end{fact}

\begin{proposition}\label{initialforcompact}
The topology of the compact hedgehog $\Lambda J(\kappa)$ is the initial topology for the family of mappings $\bigl\{ \pi_i \colon J(\kappa) \longrightarrow \left( [0,1] ,\tau_u \right) \bigr\}_{i\in I}$.
\end{proposition}

\begin{proof}
First we need to show that the projections $\pi_i \colon \Lambda J(\kappa) \longrightarrow \left( [0,1]  ,\tau_u \right)$  are continuous.  A subbasis of $\left( [0,1]  ,\tau_u \right)$ is given by
$$\mathcal{S} = \bigl\{ [0,b) \mid 0<b\leq 1 \bigr\} \cup \bigl\{ (a,1] \mid 0\leq a <1 \bigr\},$$
and it is enough to show that inverse images of subbasic opens are open.
One has
\begin{align*} [\pi_{i}<b]&= \bigl\{  (t,j)\in J(\kappa) \mid \pi_i (t,j)\in [0,b) \bigr\} \\
&=  \bigl\{  (t,i)\in J(\kappa) \mid \pi_i (t,i) \in [0,b) \bigr\} \\& \qquad \qquad \qquad \qquad \qquad  \cup  \bigl\{  (t,j)\in J(\kappa) \mid j\ne i, \pi_i (t,j) \in [0,b) \bigr\}
\\&=\bigl\{ (t,j)\in J(\kappa) \mid j\ne i\bigr\} \cup \bigl\{ (t,i)\in J(\kappa) \mid t\in [0,b) \bigr\}
\\& = \tbigcup_{j\ne i} p\left([0,1]\times\{j\} \right) \cup  p\left( [0,b) \times \{i\} \right)= J(\kappa) \smallsetminus p\left( [b,1]\times\{i\} \right) \end{align*}
and
\begin{align*} [\pi_{i}>a]&= \bigl\{  (t,j)\in J(\kappa) \mid \pi_i (t,j)\in (a,1] \bigr\} =  \bigl\{  (t,i)\in J(\kappa) \mid \pi_i (t,i) \in (a,1] \bigr\} \\
& =   \bigl\{  (t,i)\in J(\kappa) \mid t \in (a,1] \bigr\}  = p\left( (a,1] \times \{i\} \right) \end{align*}
for each $0<b\leq 1$ and $0\leq a<1$, which are open in $\Lambda J(\kappa)$. Thus one has that $\pi_i \colon\Lambda J(\kappa) \longrightarrow \left( [0,1] ,\tau_u \right)$ is continuous for every $i\in I$.

Let us show that $\Lambda J(\kappa)$ is indeed the coarsest topology making all the maps $\pi_i$ continuous. Assume that $(J(\kappa),\tau)$ is another topology with the property that  $\pi_i \colon (J(\kappa) ,\tau) \longrightarrow \left( [0,1] ,\tau_u \right)$ is continuous for every $i\in I$.  Let $S$ be a subbasic open in $\Lambda J(\kappa)$. Our goal is to show that $S$ is open in $(J(\kappa),\tau)$ too. If $S=J(\kappa)\smallsetminus  p\left( [b,1] \times \{i\} \right)$ for some $0<b\leq1$, one has $S=[\pi_{i}<b]$ which is open in $(J(\kappa),\tau)$  because $\pi_i \colon (J(\kappa) ,\tau) \longrightarrow \left( [0,1] ,\tau_u \right)$ is continuous. Similarly, if $S=p\left( (a,1] \times \{i\} \right)$ for some $0\leq a <1$, continuity of $\pi_i \colon (J(\kappa) ,\tau) \longrightarrow \left( [0,1] ,\tau_u \right)$ implies that $S=[\pi_i>a]$ is open in $(J(\kappa),\tau)$. Hence $\Lambda J(\kappa)$ is coarser than $(J(\kappa),\tau)$.  
\end{proof}

The two topologies that we have introduced so far are related as follows:
\begin{fact}\label{compQuotComp}
The topology of the compact hedgehog is coarser than the topology of the quotient hedgehog. Furthermore, both topologies coincide if and only if $\kappa <\aleph_0$.
\end{fact}

\begin{proof}
Recall that the map $\Phi$ defined in Fact~\ref{orderisom} is a homemorphism \linebreak $\Phi\colon J(\kappa) \longrightarrow ( L(\kappa), \tau_{Box})$ and $\Phi\colon \Lambda J(\kappa) \longrightarrow ( L(\kappa),\tau_{Tych})$. It is a general fact that the box topology is finer than the Tychonoff topology and that they coincide if and only if the number of factors is finite, and hence the assertion follows.
\end{proof}

In what follows we give some of the most important topological properties of the compact hedgehog.

\begin{properties}
(1) As we have already mentioned, $\Lambda J(\kappa)$ is compact.\\[3mm]
(2) Recall that $\Lambda J(\kappa)$ can be seen as a subspace of $([0,1]^I,\tau_{Tych})$, which is Hausdorff, being a product of Hausdorff spaces. Thus $\Lambda J(\kappa)$ is Hausdorff.\\[3mm]
(3) By  combining (1) and (2) and by using Proposition~\ref{comhaus} one gets that $\Lambda J(\kappa)$ is also a normal topological space.  Besides, since $\Lambda J(\kappa)$ is  $T_1$, the compact hedgehog is also regular.\\[3mm]
(4) The compact hedgehog is met\-rizable whenever $\kappa \leq \aleph_0$. Indeed, by virtue of Theorem~\ref{countablemetrizable}, one has that  $[0,1]^I$ is met\-rizable whenever $\kappa\leq \aleph_0$, and since metrizability is hereditary,  $\Lambda J(\kappa)$ is met\-rizable whenever $\kappa \leq \aleph_0$. \\[3mm]
(5) If $\kappa\leq \aleph_0$, since metrizability implies first countability, we deduce that $\Lambda J(\kappa)$ is first countable. Furthermore, whenever $\kappa > \aleph_0$ the compact hedgehog is not first countable. By contradiction, suppose that $\{N_n\}_{n\in\N}$ is a countable basis of neighborhoods of $\zerohedgehog$. For each $n\in \N$, since $N_n\in \NN_x$,  Fact~\ref{Basis0comp} yields  a finite $J_n\subseteq I$ such that 
$$B_n=  J(\kappa) \smallsetminus \tbigcup_{j\in J_n} p([t_{j}^{n},1]\times\{j\}) \subseteq N_n. $$
The set $\tbigcup_{n\in\N} J_n$ is a countable union of finite sets, so it is countable. In particular, there is an element $i_0\in I$ which is not contained in $\tbigcup_{n\in I} J_n$. Set $V=J(\kappa)\smallsetminus p([1/2,1]\times\{i_0\}$, which is an open neighborhood of $\zerohedgehog$. Since $\{N_n\}_{n\in\N}$ is a basis of neighborhoods, there is an $n_0\in \N$ such that 
$B_{n_0}\subseteq N_{n_0} \subseteq V$, which is a contradiction, since $(1,i_0)\in B_{n_0}$ but $(1,i_0)\not\in V$.\\[3mm]
(6) Because of property (5), we deduce that $\Lambda J(\kappa)$ is neither second countable nor met\-rizable whenever $\kappa>\aleph_0$.\\[3mm]
(7) Whenever $\kappa \leq \aleph_0$, one has that $\Lambda J (\kappa)$ is second countable (equivalently separable, since it is met\-rizable, see Lemma~\ref{2ndcountiffsep}).  More precisely, the family 
\begin{align*}\beta (\kappa) &=\bigl\{ J(\kappa) \smallsetminus \tbigcup_{j\in J} p([1/n,1]\times\{j\}) \mid J\subseteq I \textrm{ finite, } n\in\N \bigr\}\\& \cup \bigl\{ p\left( \left(1-1/n,1\right] \times \{i\}  \right)  \mid n\in \N,i\in I\bigr\}   \\&\cup  \bigl\{  p\left( \left(a,b\right) \times \{i\} \right)  \mid a,b\in \Q,0\leq a <b \leq 1,i\in I\bigr\}\end{align*}
is a countable basis (note that the family of finite subsets of a countable set is countable because of Lemma~\ref{finitecount}).\\[3mm]
(8) The same proof as the one for the quotient hedgehog shows that $\Lambda J(\kappa)$ fails to be separable whenever $\kappa >\aleph_0$.\\[3mm]
(9) Despite not being first countable (when $\kappa>\aleph_0$), the compact hedgehog is always a Fr\'echet--Urysohn space. Let $A\subseteq J(\kappa)$ and $(t,i)\in \overline{A}$. Our goal is to construct a sequence of points in $A$ converging to $(t,i)$.  If $t\ne0$, the proof is identical to the one for the quotient hedgehog; so we will only deal with the assertion for $(t,i)=\zerohedgehog$. We distinguish two cases:
\begin{enumerate}[$\bullet$]
\item Assume first that $A$ has points in infinitely many spines of $J(\kappa)$. Select any sequence $\{x_n\}_{n\in \N}\subseteq A$ such that $x_n$ and $x_m$ belong to different spines whenever $n\ne m$.   Let $i_n\in I$ denote the index such that $x_n \in p\left( (0,1]\times\{i_n\}  \right)$.  Then the sequence $\{x_n\}_{n\in\N}$ converges to $\zerohedgehog$. Indeed, let $B=  J(\kappa) \smallsetminus \tbigcup_{j\in J} p([t_j,1]\times\{j\}) $ be a basic neighborhood, where $J\subseteq I$ is finite. Since $J$ is finite, let
$n_B = \tbigvee \{ n\in \N \mid i_n \in J\}+1$.  Then $x_n\in B$ whenever $n\geq n_B$, and hence $\{x_n\}_{n\in\N}$ converges to $\zerohedgehog$. 
\item Suppose now that $A\subseteq p\left( [0,1]\times J\right)$ for some finite $J\subseteq I$. Note that  
$$\zerohedgehog \in \overline{A} = \overline{A\cap p\left( [0,1]\times J\right)},$$ from which follows  that $\zerohedgehog$ is contained in the closure of $A\cap p\left( [0,1]\times J\right)$ in $p\left( [0,1]\times J\right)$.
Besides,  since $J$ is finite, the subspace $p\left( [0,1]\times J\right)$ coincides with the quotient hedgehog (see Fact~\ref{compQuotComp}), which is Fr\'echet--Urysohn. Thus  there is a sequence of points of $A\cap p\left( [0,1]\times J\right)$ converging to $\zerohedgehog$ in $p\left( [0,1]\times J\right)$ (and in $\Lambda J(\kappa)$).\end{enumerate}
\vskip 3pt
(10) If $\kappa \leq \aleph_0$, the space $\Lambda J (\kappa)$ is compact and metrizable, and thus it is automatically totally bounded and complete (see \cite[Theorem~45.1]{Mun}).\\[3mm]
(11) The space $\Lambda J(\kappa)$ is arcwise connected (one can mimic the proof for the quotient hedgehog). Also, note that $\Lambda J(1) \cong \Lambda J(2) \cong [0,1]$. However, $\Lambda J(\kappa) \not\cong \Lambda J(\lambda)$ for all cardinalities $\kappa \ne \lambda$ with $\lambda$ or $\kappa$ greater than 2.
\end{properties}

 \section{The Metric Hedgehog}
 \label{ch5}
In this section we introduce the last topology on the hedgehog $J(\kappa)$. The next fact shows that the hedgehog can also be seen as a metric space.

 \begin{figure}[htbp]
\begin{center}
\begin{tikzpicture}
\draw[line width=2pt] (0,0) -- (5,0);
\foreach \x in {2,...,9}
\draw[line width=2pt,rotate=(60/sqrt(\x)-20)] (0,0) -- (5,0);
\draw[line width=2pt,rotate=(40)] (0,0) -- (5,0);
\draw[line width=3pt,red,rotate=(40)] (2,0) -- (4,0);
\filldraw[red,rotate=(40)] (2,0) circle (3pt);
\filldraw[red,rotate=(40)] (4,0) circle (3pt);
\end{tikzpicture}
\hfill
\begin{tikzpicture}
\draw[line width=2pt] (0,0) -- (5,0);
\foreach \x in {2,...,9}
\draw[line width=2pt,rotate=(60/sqrt(\x)-20)] (0,0) -- (5,0);
\draw[line width=2pt,rotate=(40)] (0,0) -- (5,0);
\draw[line width=3pt,red,rotate=(40)] (2,0) -- (0,0);
\filldraw[red,rotate=(40)] (2,0) circle (3pt);
\draw[line width=3pt,red,rotate=(60/sqrt(3)-20)] (3,0) -- (0,0);
\filldraw[red,rotate=(60/sqrt(3)-20)] (3,0) circle (3pt);
\end{tikzpicture}
\caption{The distance between two pairs of points in $J(\kappa)$}
\label{metriconhedgehog}
\end{center}
\end{figure}
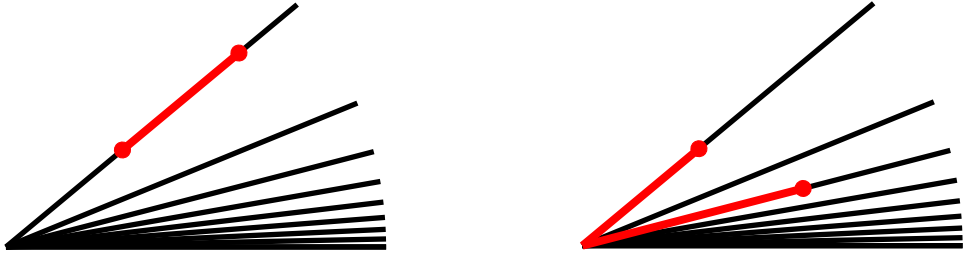

\begin{fact} The map $d\colon J(\kappa)\times J(\kappa)\longrightarrow[0,+\infty)$ given by
 \[d\left((t,i),(s,j)\right) =\begin{cases}|t-s|,&\text{ if  }j=i,\\[0.1cm] t+s,&\text{ if  }j\ne i,\end{cases}\qquad (t,i),(s,j)\in J(\kappa),\]
is a metric on $J(\kappa)$. 
\end{fact}

The hedgehog $J(\kappa)$ equipped with the metric $d$ will be called the \emph{metric hedgehog with $\kappa$ spines}. In what follows, we shall use the symbol 
\[MJ(\kappa)\]
 to denote the metric hedgehog. The metric hedgehog has important applications in Topology. For example, later on we will show that every met\-rizable space can be embedded in a countable  cartesian product of hedgehogs (see Theorem~\ref{Kowalsky}). Another interesting application is given in Theorem~\ref{colnorTh}.\\[2mm]
In view of the metric, we can easily describe the balls of $MJ(\kappa)$. Here and subsequently, we shall denote by $B\left( (t,i) , r\right)$ the open ball of center $(t,i)$ and radius $r$ and  by $\overline{B}\left( (t,i) , r\right)$ the closed ball of center $(t,i)$ and radius $r$. 

\begin{fact}  The open balls of $MJ(\kappa)$  centered at $\zerohedgehog$ are given by
$$B( \zerohedgehog , r )= p\left( [0,r)\times I\right) \quad \textrm{ for all } 0<r\leq 1.$$
At any other point $(t,i)\ne \zerohedgehog$ with $t<1$, we have
$$B\left( (t,i) , r\right)= p\left( (t-r,t+r) \times \{i\} \right) \quad \textrm{ for all } 0<r\leq \min\{t,1-t\}.$$
\end{fact}

  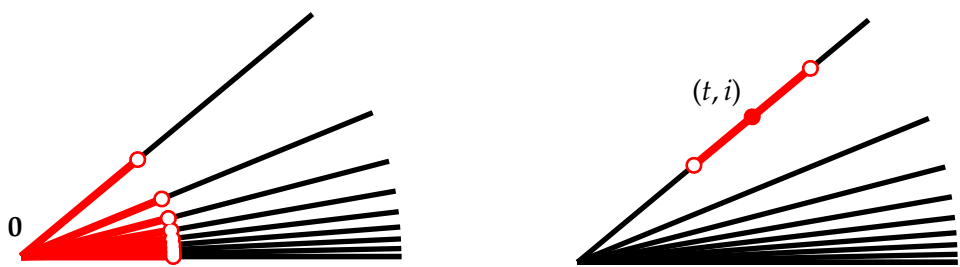
\begin{figure}[htbp]
\begin{center} 
\begin{tikzpicture}
\draw[line width=2pt] (0,0) -- (5,0);
\foreach \x in {2,...,9}
\draw[line width=2pt,rotate=(60/sqrt(\x)-20)] (0,0) -- (5,0);
\draw[line width=2pt,rotate=(40)] (0,0) -- (5,0);
\draw[line width=3pt,red,rotate=(40)] (2,0) -- (0,0);
\filldraw[red,rotate=(40)] (2,0) circle (3pt);
\filldraw[white,rotate=(40)] (2,0) circle (2pt);
\draw[rotate=(40)] (0.2,0)  node[anchor=south east] {$\zerohedgehog$};
\draw[line width=3pt,red, rotate=(60/sqrt(2)-20)] (0,0) -- (2,0);
\draw[line width=3pt,red, rotate=(60/sqrt(3)-20)] (0,0) -- (2,0);
\draw[line width=3pt,red, rotate=(60/sqrt(4)-20)] (0,0) -- (2,0);
\draw[line width=3pt,red, rotate=(60/sqrt(5)-20)] (0,0) -- (2,0);
\draw[line width=3pt,red, rotate=(60/sqrt(6)-20)] (0,0) -- (2,0);
\draw[line width=3pt,red, rotate=(60/sqrt(7)-20)] (0,0) -- (2,0);
\draw[line width=3pt,red, rotate=(60/sqrt(8)-20)] (0,0) -- (2,0);
\draw[line width=3pt,red, rotate=(60/sqrt(9)-20)] (0,0) -- (2,0);
\filldraw[red,rotate=(60/sqrt(2)-20)] (2,0) circle (3pt);
\filldraw[red,rotate=(60/sqrt(3)-20)] (2,0) circle (3pt);
\filldraw[red,rotate=(60/sqrt(4)-20)] (2,0) circle (3pt);
\filldraw[red,rotate=(60/sqrt(5)-20)] (2,0) circle (3pt);
\filldraw[red,rotate=(60/sqrt(6)-20)] (2,0) circle (3pt);
\filldraw[red,rotate=(60/sqrt(7)-20)] (2,0) circle (3pt);
\filldraw[red,rotate=(60/sqrt(8)-20)] (2,0) circle (3pt);
\filldraw[red,rotate=(60/sqrt(9)-20)] (2,0) circle (3pt);
\filldraw[white,rotate=(60/sqrt(2)-20)] (2,0) circle (2pt);
\filldraw[white,rotate=(60/sqrt(3)-20)] (2,0) circle (2pt);
\filldraw[white,rotate=(60/sqrt(4)-20)] (2,0) circle (2pt);
\filldraw[white,rotate=(60/sqrt(5)-20)] (2,0) circle (2pt);
\filldraw[white,rotate=(60/sqrt(6)-20)] (2,0) circle (2pt);
\filldraw[white,rotate=(60/sqrt(7)-20)] (2,0) circle (2pt);
\filldraw[white,rotate=(60/sqrt(8)-20)] (2,0) circle (2pt);
\filldraw[white,rotate=(60/sqrt(9)-20)] (2,0) circle (2pt);
\end{tikzpicture}
\hfill
\begin{tikzpicture}
\draw[line width=2pt] (0,0) -- (5,0);
\foreach \x in {2,...,9}
\draw[line width=2pt,rotate=(60/sqrt(\x)-20)] (0,0) -- (5,0);
\draw[line width=2pt,rotate=(40)] (0,0) -- (5,0);
\draw[line width=3pt,red,rotate=(40)] (2,0) -- (4,0);
\filldraw[red,rotate=(40)] (3,0) circle (3pt);
\filldraw[red,rotate=(40)] (4,0) circle (3pt);
\filldraw[red,rotate=(40)] (2,0) circle (3pt);
\filldraw[white,rotate=(40)] (4,0) circle (2pt);
\filldraw[white,rotate=(40)] (2,0) circle (2pt);
\draw[rotate=(40)] (3,0)  node[anchor=south east] {$(t,i)$};
\end{tikzpicture}
\caption{Open balls centered at $\zerohedgehog$ and at $(t,i)\ne 0$}
\label{ballsmetric}
\end{center}
\end{figure}
On the one hand, all open balls centered at a certain point form a basis of neighborhoods of that point in $MJ(\kappa)$. On the other hand, we can also give a subbasis of the metric hedgehog.
\begin{fact}\label{SubBaseMetric} The following family is a subbasis of the metric hedgehog $MJ(\kappa)$:
$$\tilde{\mathcal{S}}(\kappa) = \bigl\{ p\left([0,r)\times I \right) \mid r\in (0,1] \bigr\} \cup  \bigl\{ p \left((s,1] \times \{i\}\right)  \mid  s< 1,i\in I\bigr\}.$$
\end{fact}

 \begin{figure}[htbp]
\begin{center}
\begin{tikzpicture}
\draw[line width=2pt] (0,0) -- (5,0);
\foreach \x in {2,...,9}
\draw[line width=2pt,rotate=(60/sqrt(\x)-20)] (0,0) -- (5,0);
\draw[line width=2pt,rotate=(40)] (0,0) -- (5,0);
\draw[line width=3pt,red,rotate=(40)] (2,0) -- (0,0);
\filldraw[red,rotate=(40)] (2,0) circle (3pt);
\filldraw[white,rotate=(40)] (2,0) circle (2pt);
\draw[rotate=(40)] (0.2,0)  node[anchor=south east] {$\zerohedgehog$};
\draw[line width=3pt,red, rotate=(60/sqrt(2)-20)] (0,0) -- (2,0);
\draw[line width=3pt,red, rotate=(60/sqrt(3)-20)] (0,0) -- (2,0);
\draw[line width=3pt,red, rotate=(60/sqrt(4)-20)] (0,0) -- (2,0);
\draw[line width=3pt,red, rotate=(60/sqrt(5)-20)] (0,0) -- (2,0);
\draw[line width=3pt,red, rotate=(60/sqrt(6)-20)] (0,0) -- (2,0);
\draw[line width=3pt,red, rotate=(60/sqrt(7)-20)] (0,0) -- (2,0);
\draw[line width=3pt,red, rotate=(60/sqrt(8)-20)] (0,0) -- (2,0);
\draw[line width=3pt,red, rotate=(60/sqrt(9)-20)] (0,0) -- (2,0);
\filldraw[red,rotate=(60/sqrt(2)-20)] (2,0) circle (3pt);
\filldraw[red,rotate=(60/sqrt(3)-20)] (2,0) circle (3pt);
\filldraw[red,rotate=(60/sqrt(4)-20)] (2,0) circle (3pt);
\filldraw[red,rotate=(60/sqrt(5)-20)] (2,0) circle (3pt);
\filldraw[red,rotate=(60/sqrt(6)-20)] (2,0) circle (3pt);
\filldraw[red,rotate=(60/sqrt(7)-20)] (2,0) circle (3pt);
\filldraw[red,rotate=(60/sqrt(8)-20)] (2,0) circle (3pt);
\filldraw[red,rotate=(60/sqrt(9)-20)] (2,0) circle (3pt);
\filldraw[white,rotate=(60/sqrt(2)-20)] (2,0) circle (2pt);
\filldraw[white,rotate=(60/sqrt(3)-20)] (2,0) circle (2pt);
\filldraw[white,rotate=(60/sqrt(4)-20)] (2,0) circle (2pt);
\filldraw[white,rotate=(60/sqrt(5)-20)] (2,0) circle (2pt);
\filldraw[white,rotate=(60/sqrt(6)-20)] (2,0) circle (2pt);
\filldraw[white,rotate=(60/sqrt(7)-20)] (2,0) circle (2pt);
\filldraw[white,rotate=(60/sqrt(8)-20)] (2,0) circle (2pt);
\filldraw[white,rotate=(60/sqrt(9)-20)] (2,0) circle (2pt);
\end{tikzpicture}
\hfill
\begin{tikzpicture}
\draw[line width=2pt] (0,0) -- (5,0);
\foreach \x in {2,...,9}
\draw[line width=2pt,rotate=(60/sqrt(\x)-20)] (0,0) -- (5,0);
\draw[line width=2pt,rotate=(40)] (0,0) -- (5,0);
\draw[line width=3pt,red,rotate=(40)] (3,0) -- (5,0);
\filldraw[red,rotate=(40)] (3,0) circle (3pt);
\filldraw[white,rotate=(40)] (3,0) circle (2pt);
\draw[rotate=(40)] (3,0)  node[anchor=south east] {$(t,i)$};
\end{tikzpicture} 
\caption{Subbase of $MJ(\kappa)$}
\label{subbmet}
\end{center}
\end{figure}
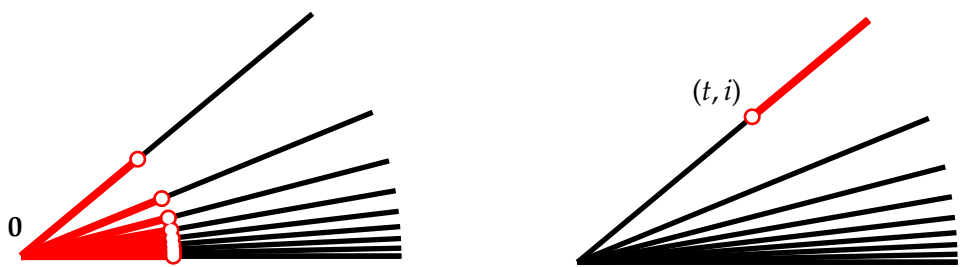

\medskip

Let $\mathcal{S}(\kappa)$ denote the subbase of the quotient hedgehog $J(\kappa)$ defined as in Fact~\ref{subbasequot}.  Then, one clearly has $\tilde{\mathcal{S}}(\kappa) \subseteq \mathcal{S}(\kappa)$, which yields a comparison relation between both topologies.

\begin{fact}\label{comparmetricquot}
The topology of the quotient hedgehog $J(\kappa)$ is finer than the topology of the metric hedgehog $MJ(\kappa)$. 
\end{fact}

Furthermore, in the finite case, the quotient topology and the metric topology coincide.  

\begin{fact}\label{metricquotientsame}
The metric topology on the hedgehog coincides with the quotient topology on the hedgehog if and only if $\kappa<\aleph_0$.
\end{fact}

\begin{proof}
$\Rightarrow)$ Suppose $\kappa \geq \aleph_0$ and select a sequence $\{ i_n \}_{n\in\N}$ of elements of $I$. For each $n\in \N$, set $t_{i_n} = 1/n$; and for $i\not \in \{ i_n \}_{n\in\N}$, set $t_i=1$. Then $\tbigcup_{i\in I} p\left( [0,t_i)\times\{i\} \right)$ is an open neighborhood of $\zerohedgehog$ in the quotient hedgehog which is not open in the metric hedgehog. Thus the quotient hedgehog is strictly finer than the metric hedgehog.\\[2mm]
$\Leftarrow)$ Assume now that $I= \{i_1,\dots i_k\}$ is finite. Because of the previous fact, we only need to show that the metric hedgehog is finer than the quotient hedgehog. Let $S\in \mathcal{S}(\kappa)$ be a subbasic open in the quotient hedgehog, and let us check that $S$ is also open in the metric hedgehog. If $S= p\left( (s,1]\times \{i\} \right)$ for some $s<1$ and $i\in I$, then $S$ is also open in the metric hedgehog. Assume now that $S=\tbigcup_{j=1}^{k}  p([0,\frac{1}{n_j} )\times \{i_j\} )$, with $n_j \in \N$   for each $j=1,\dots,k.$ Let $(t,i_j)\in S$. If $(t,i_j)\ne \zerohedgehog$, we have that $(0,1/n_j)\times\{i_j\}$ is an open neighborhood (in the metric hedgehog) in-between. Otherwise, if $(t,i_j) = \zerohedgehog$, let $n_0 = \max \{ n_1,\dots n_k\}$. Then $p\left([0,1/n_0)\times I\right)$ is an open neighborhood in the metric hedgehog in-between.
\end{proof}

The following proposition is the analogue of Proposition~\ref{initialforcompact} for the metric hedgehog. It will be fundamental in the subsequent subections, since it allows us to check the continuity of hedgehog-valued functions by means of the universal property of the initial topology (cf. Proposition~\ref{univpropinitial}).

\begin{proposition}\label{inittop}
The topology of the metric hedgehog $MJ(\kappa)$ is the initial topology for the family of mappings $\bigl\{ \pi_i \colon J(\kappa) \longrightarrow \left( [0,1] ,\tau_u \right) \bigr\}_{i\in I}$ together with the mapping $\pi_\kappa \colon J(\kappa) \longrightarrow \left( [0,1]  ,\tau_u \right)$.
\end{proposition}

\begin{proof}
First we need to show that the projections $\pi_\kappa \colon MJ(\kappa) \longrightarrow \left( [0,1]  ,\tau_u \right) $ and $\pi_i \colon MJ(\kappa) \longrightarrow \left( [0,1]  ,\tau_u \right)$  are continuous.  A subbasis of $\left( [0,1]  ,\tau_u \right)$ is given by $\mathcal{S} = \bigl\{ [0,b) \mid 0<b\leq 1 \bigr\} \cup \bigl\{ (a,1] \mid 0\leq a <1 \bigr\},$ and it is enough to show that inverse images of subbasic opens are open. Since (see the proof of Proposition~\ref{initialforcompact}) 
\begin{align*} [\pi_{i}<b]= J(\kappa) \smallsetminus p\left( [b,1]\times\{i\} \right) \quad \textrm{and} \quad [ \pi_{i}>a] = p\left( (a,1] \times \{i\} \right) \end{align*}
are open in $MJ(\kappa)$ for every $0<b\leq 1$ and $0\leq a <1$, one has that $\pi_i\colon MJ(\kappa) \longrightarrow ([0,1],\tau_u)$ is continuous for every $i\in I$.  Similarly, for every $0<b\leq 1$ and $0\leq a <1$ we have
\begin{align*} [\pi_{\kappa}<b]&= \bigl\{  (t,j)\in J(\kappa) \mid \pi_\kappa (t,j)\in [0,b) \bigr\} =  \bigl\{  (t,j)\in J(\kappa) \mid t \in [0,b) \bigr\} \\
&  = p\left( [0,b) \times I \right), \end{align*}
and
\begin{align*} [\pi_{\kappa}>a]&= \bigl\{  (t,j)\in J(\kappa) \mid \pi_\kappa (t,j)\in (a,1] \bigr\} =  \bigl\{  (t,j)\in J(\kappa) \mid t \in (a,1] \bigr\} \\
& =   p\left( (a,1] \times  I \right), \end{align*}
which are open in $MJ(\kappa)$.  Thus $\pi_\kappa$ is also continuous.  

Now we have to prove that $MJ(\kappa)$ is indeed the coarsest topology making all the maps in the family continuous. Let $(J(\kappa),\tau)$ be another topology such that $\pi_\kappa$ and $\pi_{i}$ are continuous for every $i\in I$.  Let $S$ be a subbasic open in $MJ(\kappa)$. Our goal is to show that $S$ is open in $(J(\kappa),\tau)$ too.

First let $S=p\left( [0,r)\times I \right)$ with $0<r\leq 1$.  By the previously done calculation, one has $[\pi_\kappa<r] = S$, and since $\pi_\kappa\colon (J(\kappa),\tau) \longrightarrow \left([0,1],\tau_u\right)$ is continuous, we conclude that  $S$ is open in $(J(\kappa),\tau)$. Now take $S=p\left( (s,1]\times \{i\} \right)$ with $0\leq s <1$ and $i\in I$.  Since $\pi_i\colon (J(\kappa),\tau) \longrightarrow \left([0,1],\tau_u\right)$ is continuous, we deduce that $[\pi_i>s] = S$ is open in $(J(\kappa),\tau)$.
\end{proof}

Now we are in position to give the following relation between the three hedgehog spaces.

\begin{corollary}
The topology of the quotient hedgehog is finer than the topology of the metric hedgehog, which is finer than the topology of the compact hedgehog.
\end{corollary}

\begin{proof}
The former assertion is proved in Fact~\ref{comparmetricquot} whereas the latter assertion is a consequence of the previous proposition and Proposition~\ref{initialforcompact}.
\end{proof}

\begin{fact}\label{compareComMet}
The metric topology of the hedgehog $J(\kappa)$ coincides with the compact topology of the hedgehog if and only if $\kappa<\aleph_0$.
\end{fact}

\begin{proof}
$\Rightarrow )$ Assume that $\kappa \geq \aleph_0$, and let us show that $MJ(\kappa)$ is strictly finer than $\Lambda J(\kappa)$.  We have that $U=p\left([0,1) \times I\right)$ is an open subset of $MJ(\kappa)$ which is not open in $\Lambda J(\kappa)$. Indeed, if $U$ were open in $\Lambda J(\kappa)$, Fact~\ref{Basis0comp} yields a finite $J\subseteq I$ such that 
$$B=J(\kappa) \smallsetminus \tbigcup_{j\in J} [t_j,1] \subseteq  U.$$
Since $ \kappa \geq \aleph_0$, select $i\in I$ such that $i\not\in J$, from which follows that $(1,i)\in B$ but $(1,i)\not\in U$, a contradiction.\\[2mm]
$\Leftarrow )$ Assume that $\kappa<\aleph_0$. By the previous corollary, it is enough to show that $MJ(\kappa)$ is coarser tan $\Lambda J(\kappa)$. For that purpose, take the subbasis 
$\tilde{\mathcal{S}}(\kappa)$ of $MJ(\kappa)$ (cf. Fact~\ref{SubBaseMetric}) and the subbasis $\hat{\mathcal{S}}(\kappa)$ of $\Lambda J(\kappa)$  (cf. Fact~\ref{subbasecomp}). It is enough to show that every subbasic open $S$ in $\tilde{\mathcal{S}}(\kappa)$ is also open in $\Lambda J(\kappa)$. If $S= p\left( (r,1]\times\{i\} \right)$, one has  that $S\in \hat{\mathcal{S}}(\kappa)$, so let us assume that $S=p\left( [0,r) \times I \right)$. Then,
$$S=\tbigcap_{i\in I} J(\kappa) \smallsetminus p\left( [r,1] \times \{i\} \right),$$
which is a finite intersection of (subbasic) open subsets of $\Lambda J(\kappa)$, and hence it is open in the compact hedgehog.
\end{proof}
Thus, Facts~\ref{metricquotientsame} and~\ref{compareComMet} show that in the finite case the three hedgehog are just the same space:
\begin{corollary}
Assume that $\kappa < \aleph_0$. Then
$$J(\kappa) = MJ(\kappa) = \Lambda J (\kappa).$$
\end{corollary}

\begin{properties}
(1) The hedgehog $MJ(\kappa)$  is a metric space, and therefore it is first countable, Hausdorff (in particular $T_1$) and normal  (thus regular).\\[3mm]
(2) Note that $J(\aleph_0)$ can be provided with two met\-rizable topologies, namely $MJ(\aleph_0)$ and $\Lambda J(\aleph_0)$. Further, the metrics are non-equivalent, since both topologies are distinct when $\kappa=\aleph_0$ (see Fact~\ref{compareComMet}). \\[3mm]
(3) $MJ(	\kappa)$ is second countable if and only if $\kappa \leq \aleph_0$. Assume first that $\kappa \leq \aleph_0$. A countable basis is given by
\begin{align*}\beta(\kappa) &= \textstyle{\bigl\{ p\left(\left[0,1/n \right)\times I \right) \mid n\in\N \bigr\}\cup \bigl\{ p\left( \left(1-1/n,1\right] \times \{i\} \right)  \mid n\in \N,i\in I\bigr\}} \\
& \cup  \bigl\{  p\left( \left(a,b\right) \times \{i\} \right)  \mid a,b\in \Q,0\leq a <b \leq 1,i\in I\bigr\}.
\end{align*}
Suppose now that $\kappa > \aleph_0$ and by contradiction take a countable basis $\tilde\beta (\kappa)$. For each $i\in I$, consider the open set $U_i = p\left((0,1]\times \{i\}\right)$ in $MJ(\kappa)$. Then there is a $B_i \in \tilde\beta(\kappa)$ such that $B_i \subseteq U_i$. It follows that $\{B_i\}_{i\in I}$ is an uncountable family in $\tilde\beta(\kappa)$, a contradiction.\\[3mm]
(4) Exactly the same argument used for the quotient hedgehog shows that $MJ(\kappa)$ is neither compact nor locally compact whenever $\kappa\geq \aleph_0$.\\[3mm]
(5) Note that the proof of the arcwise connectedness of the quotient hedgehog also applies to the metric hedgehog. Thus, $MJ(\kappa)$ is arcwise connected. Besides, note that $MJ(1) \cong MJ(2) \cong [0,1]$. However, $MJ(\kappa) \not\cong MJ(\lambda)$ for all cardinalities $\kappa \ne \lambda$ with $\lambda$ or $\kappa$ greater than 2.\\[3mm]
(6) Taking into account (3) and Lemma~\ref{2ndcountiffsep}, the metric hedgehog is separable if and only if $\kappa \leq \aleph_0$.\\[3mm]
(7) The metric hedgehog $MJ(\kappa)$ is totally bounded if and only if $\kappa < \aleph_0$. Assume first that $\kappa < \aleph_0$ and let $\epsilon > 0$. Then,
$$\bigl\{B\left( \zerohedgehog, \varepsilon \right)\bigr\} \cup \bigl\{ B( (n\varepsilon/2, i), \varepsilon) \mid i\in I, 2\leq n \leq 2/\varepsilon  \bigr\}$$
is the desired finite cover of open balls of radius $\varepsilon$. Suppose now by way of contradiction  that $\kappa \geq \aleph_0$ and that $MJ(\kappa)$ is totally bounded.  For $\varepsilon = 1/2$, there is a finite cover $\mathcal{B}$ of open balls of radius $\varepsilon$.
It follows that for every $i\in I$ there is a ball $B_i \in \mathcal{B}$ such that $(1,i)\in B_i$. Note that $B_i \subseteq p\left( (0,1] \times \{i\} \right)$, and hence $\{B_i\}_{i\in I}\subseteq \mathcal{B}$ is an infinite family, a contradiction.\\[3mm]
(8) {\rm(Cf.\ \cite[page~277, Problem~4.3.B.\,(c)]{Eng})} The metric hedgehog is complete. Let $\{x_n\}_{n\in \N}$ be a Cauchy sequence in $MJ(\kappa)$. We want to prove that $\{x_n\}_{n\in \N}$ is convergent, which, by virtue of Lemma~\ref{ChConv}, is equivalent to show that $\{x_n\}_{n\in \N}$ has a convergent subsequence.  By contradiction, suppose that $\{x_n\}_{n\in \N}$ has not convergent subsequences. In particular, there are not subsequences of $\{x_n\}_{n\in \N}$ converging to $\zerohedgehog$. Set $S= \{x_n \mid n\in\N\}$. Then there is an $r>0$ such $B\left( \zerohedgehog, r \right) \cap S$ is finite, for otherwise $\zerohedgehog$ is a limit point of $S$, from which we conclude (by Lemma~\ref{lemmaconvergentlimit}) that there is convergent subsequence of $\{x_n\}_{n\in \N}$ to $\zerohedgehog$, a contradiction. Now, by reducing $r$ if necessary, we can assume that $B\left( \zerohedgehog, r \right) \cap ( S \smallsetminus \{\zerohedgehog\} ) = \varnothing$.  We distinguish two cases:
\begin{enumerate}[$\bullet$]
\item Assume first that $S$ has points in infinitely many spines of $MJ(\kappa)$.  Thus we can build a subsequence $\{x_{n_k}\}_{k\in\N}$ of $\{x_n\}_{n\in \N}$ such that $x_{n_k} \in p ((0,1]\times\{i_k\})$ and $x_{n_\ell}\in p((0,1]\times\{i_\ell\})$ with $i_k\ne i_\ell$ whenever $k\ne \ell$. Observe that 
$d\left(x_{n_k},x_{n_\ell}\right)\geq 2r$ whenever  $k\ne\ell$, a contradiction with the fact that $\{x_n\}_{n\in\N}$ is Cauchy. 
\item Suppose now that $S$ has points only in finitely many spines of $MJ(\kappa)$. Then there is a subsequence $\{x_{n_k}\}_{k\in\N}$ of $\{x_n\}_{n\in\N}$ such that $S_1 = \{x_{n_k} \mid k\in\N\} \subseteq p\left( [0,1]\times\{i_0\} \right)$ for some $i_0\in I$. Now, $[0,1]$ is clearly isometric to $p\left( [0,1]\times\{i_0\} \right)$, the former being complete (because it is a closed subset of the real line). Thus, $p\left( [0,1]\times\{i_0\} \right)$ is complete and $\{x_{n_k} \}_{k\in \N}$ (and therefore $\{x_n\}_{n\in\N}$) has a convergent subsequence, a contradiction.  
\end{enumerate}
\end{properties}

Later, we shall need the following property concerning the metric hedgehog (see \cite{Prz}).

\begin{proposition}\label{Rembeds}
The real line $\mathbb{R}$ embeds as a closed subspace in $\bigl( MJ(\aleph_0) \bigr)^2$.
\end{proposition}
\begin{proof}
Throughout the proof we shall denote $Z=\bigl( MJ(\aleph_0) \bigr)^2$.  We will show that $\mathbb{R}$ embeds as a closed subspace in  $Z$ for the index set $I=\Z$. 
We can uniquely represent every $x\in\R$ as $x= 2k_x + t_x = (2\ell_x+1) + s_x$ where $k_x,\ell_x\in \Z$ and $-1< t_x,s_x \leq 1$. Thus we define functions $f\colon \mathbb{R} \longrightarrow MJ(\aleph_0)$ and $g\colon  \mathbb{R} \longrightarrow MJ(\aleph_0)$ given by 
\[f(x) = (1-|t_x|,k_x)\quad\text{and}\quad g(x) = (1-|s_x|, \ell_x).\]
Now we show that $f$ is continuous. Because of Propositions~\ref{univpropinitial} and~\ref{inittop}, it is enough to show that $\pi_{\aleph_0} \circ f$ is continuous and $\pi_k \circ f$ is also continuous for every $k\in \Z$. For every $0<s\leq 1$ we have
\begin{align*}
[\pi_{\aleph_0} \circ f <s] &  =\{ 2n+t\in\R\mid  n\in \Z,\ -1<t\leq 1 ,\ \pi_{\aleph_0} (1-|t|, n) <s\} \\
&= \{ 2n+t \in \R\mid n\in \Z,\ -1<t\leq 1,\ 1-|t| <s \} \\
&= \tbigcup_{n\in \Z} \{ 2n+t \in \R\mid-1<t<-1+s\text{ or } 1-s<t\leq 1 \} \\
&= \tbigcup_{n\in \Z} (2n-1,2n-1+s)\cup (2n+1-s,2n+1] \\
&=  \tbigcup_{n\in \Z} (2n-1,2n-1+s) \cup \tbigcup_{n\in \Z} (2n-1-s,2n-1] \\
& = \tbigcup_{n\in\Z}(2n-1-s,2n-1+s),
\end{align*}
and, for each $k\in\Z$,
\begin{align*}
[ \pi_k \circ f < s] &= \{ 2n+t \in \R \mid  n\in \Z,\ -1<t\leq 1,\ \pi_k( 1-|t|,n) <s \}\\
&=  \{ 2n+t \in \R \mid  n\in \Z,\ -1<t \leq 1 ,\ n\ne k\}\\
&\qquad\qquad\qquad\qquad\qquad \cup  \{ 2k+t \in \R \mid -1<t \leq 1,\  1-|t| <s \}  \\
&= \tbigcup_{n\ne k} (2n-1,2n+1] \cup (2k-1,2k-1+s)\cup (2k+1-s,2k+1] \\
& = \R \smallsetminus [2k-1+s,2k+1-s]
\end{align*}
which are open in the real line. Let now $0\leq s <1$. Similarly, we have that
\begin{align*}
[\pi_{\aleph_0} \circ f >s] &  =\{ 2n+t\in\R\mid  n\in \Z,\ -1<t\leq 1 ,\ \pi_{\aleph_0} (1-|t|, n) >s\} \\
&= \{ 2n+t \in \R\mid n\in \Z,\ -1<t\leq 1,\ 1-|t| >s \} \\
& = \tbigcup_{n\in\Z}(2n-1+s, 2n+1-s),
\end{align*}
and
\begin{align*}
[ \pi_k \circ f > s] &= \{ 2n+t \in \R \mid  n\in \Z,\ -1<t\leq 1,\ \pi_k( 1-|t|,n) >s \}\\
&
=   \{ 2k+t \in \R \mid -1<t \leq 1,\  1-|t| >s \}\\
&  = (2k-1+s, 2k+1-s)
\end{align*}
and again both of them are open in $\R$. Hence $f$ is continuous. One can check the continuity of $g$ similarly.
We will show that the diagonal $h= f \Delta g$ is the desired closed embedding. It suffices to prove that $h(\R)$ is closed in $Z$ \setcounter{ffact}{2}(Fact~\Roman{ffact}), that $h$ restricted onto its image is open \setcounter{ffact}{3}(Fact~\Roman{ffact}) and that $h$ is one-to-one.

\setcounter{ffact}{1}
\begin{ffact}
 Let $\alpha,\beta\in J(\aleph_0)$ and denote $\alpha=(t,n)$.\begin{enumerate}[\normalfont (i)]
\item If $t>0$, one has $(\alpha, \beta) \in h(\R)$ if and only if $\beta = (1-t,n)$ or $\beta= (1-t,n-1)$.
\item One has $(\zerohedgehog, \beta) \in h(\R)$ if and only if there exists $m\in \Z$ such that $\beta= (1,m)$.\end{enumerate}
\end{ffact}

\begin{proof}
(i) Let us first prove the ``only if'' part. Assume that there is an $x= 2m +s \in \R$ such that $f(x)=\alpha$ and $g(x)=\beta$, where $-1<s\leq 1$ and $m\in\Z$. Note that
$$(1-|s|, m) = f(x) = \alpha = (t,n),$$
and since $t>0$, we deduce that $m=n$ and $s=t-1$ or $s=1-t$.  Thus, if $x= 2n+t-1=2(n-1)+1+t$, one has $\beta=g(x)= (1-t,n-1)$, as desired. Assume otherwise that $x=2n+1-t$. If $t=1$, we have $\beta=g(x)=\zerohedgehog=(0,n)$; so let us now suppose that $0<t<1$. 	In the latter case we have $\beta=g(x)=(1-t,n)$, the desired conclusion.

Let us now show the converse. If $\beta=(1-t,n)$, take $x=2n+1-t$. One easily checks that $f(x)=\alpha$ and $g(x)=\beta$, that is, $(\alpha,\beta)\in h(\R)$. If $\beta=(1-t,n-1)$, set $x=2n+t-1$. Then $f(x)=\alpha$ and $g(x)=\beta$, as we wanted to show.\\[3mm]
(ii) For the sufficiency, let $x=2m+s\in \R$ such that $f(x)=\zerohedgehog$ and $g(x)=\beta$, where $-1<s\leq 1$ and $m\in\Z$. Since $(1-|s|, m)=f(x)=\zerohedgehog$, one has $s=1$ or $s=-1$. In the former case, we deduce that $\beta= (1,m)$, whereas in the latter case one has $\beta=(1,m-1)$. 

Now we show the necessity. If $\beta= (1,m)$ for some $m\in \Z$. Let $x=2m+1$. Then $f(x)=\zerohedgehog$ and $g(x)=\beta$, which concludes the proof.\qedhere
\end{proof}

\setcounter{ffact}{2}
\begin{ffact}
$h(\R)$ is closed in $Z$.
\end{ffact} 

\begin{proof}
We shall show that the complementary of $h(\R)$ is open. Let $(\alpha_0,\beta_0)\in Z \smallsetminus h(\R)$, and write $\alpha_0= (t_0,n_0)$ and $\beta_0=(s_0,m_0)$. We distinguish several cases:	\\[3mm]
(i) First assume that $t_0,s_0>0$. By the previous fact, one has $\beta_0\ne (1-t_0,n_0)$ and $\beta_0\ne (1-t_0,n_0-1)$. If $m_0\ne n_0,n_0-1$, define 
$$U=p( (0,1]\times \{n_0\} ) \times p ( (0,1] \times \{m_0\})$$
which is a neighborhood of $(\alpha_0, \beta_0)$ with the property that for each $(\alpha,\beta)\in U$ (write $\alpha=(t,n_0)$), one has $\beta\ne (1-t,n_0),(1-t,n_0-1)$, and hence  the previous fact yields $U\subseteq Z \smallsetminus h(\R)$.

If $m_0=n_0$, since $(s_0,n_0)= \beta_0\ne (1-t_0,n_0)$, one has that $s_0 \ne 1-t_0$, from which follows that the point $(t_0,s_0)\in \R^2$ is not contained in the line $y=1-x$ of $\R^2$. Let then $d>0$ be the euclidean distance between $(t_0,s_0)$ and the line $y=1-x$. We set $r=d/\sqrt2$ and
$$U= \begin{cases} p( (t_0-r,t_0+r) \times\{n_0\} ) \times p( (s_0-r,s_0+r)\times\{n_0\} )  & \textrm{if } t_0,s_0<1;\\
p( (1-r,1] \times\{n_0\} ) \times p( (s_0-r,s_0+r)\times\{n_0\} )  & \textrm{if } t_0=1,s_0<1;\\
p( (t_0-r,t_0+r) \times\{n_0\} ) \times p( (1-r,1]\times\{n_0\} )  & \textrm{if } t_0<1,s_0=1;\\
p( (1-r,1] \times\{n_0\} ) \times p( (s_0-r,1]\times\{n_0\} ) & \textrm{if } t_0=s_0=1.
\end{cases}$$
The neighborhood $U$ of $(\alpha_0,\beta_0)$ satisfies that given $(\alpha, \beta)$ in $U$ (assume that $\alpha=(t,n_0)$ and $\beta=(s,n_0)$), one has $\beta\ne (1-t,n_0),(1-t,n_0-1)$. Indeed, the case $\beta \ne (1-t,n_0-1)$ is clear, so let us verify that $\beta\ne (1-t,n_0)$. By way of contradiction, suppose that $1-t=s$. Then $(t,s)\in \R^2$ is contained in the line $y=1-x$, which yields
$$d \leq d_u \left( (t,s), (t_0,s_0) \right) \leq \sqrt2 \max\{ |t-t_0|,|s-s_0|\} < \sqrt 2 r = d,$$
a contradiction. Thus $U\subseteq Z  \smallsetminus h(\R)$. The case $m_0=n_0-1$ is completely analogous.\\[3mm]
(ii) Suppose now that $t_0>0$ and $s_0=0$. By the previous fact we have  $\beta_0\ne (1-t_0,n_0)$ and $\beta_0\ne (1-t_0,n_0-1)$, from which follows that $1-t_0\ne 0$. The point $(t_0,0)\in \R^2$ is not contained in the  line $y=1-x$ of the plane, so let $d>0$ be the distance from such line to $(t_0,0)$. Define $r=d/\sqrt2$ and 
$$U=
\begin{cases}
p\left( (t_0-r,t_0+r)\times\{n_0\} \right) \times p \left( [0,r)\times I \right) & \textrm{if } t_0<1;\\
p( (1-r,1] \times\{n_0\} )  \times p \left( [0,r)\times I \right) & \textrm{if } t_0=1;
\end{cases}$$
which is a neighborhood of $(\alpha_0,\beta_0)$. Let $(\alpha,\beta)\in U$, and denote $\alpha=(t,n_0)$ and $\beta=(s,m)$. If $\beta=(1-t,n_0)$ or $\beta=(1-t,n_0-1)$, we would have $1-t=s$, that is, $(t,s)\in \R^2$  is contained in the line $y=1-x$, from which we conclude that
$$d \leq d_u \left( (t,s), (t_0,s_0) \right) \leq \sqrt2 \max\{ |t-t_0|,s \} < \sqrt 2 r = d,$$
a contradiction. Thus $U\subseteq Z \smallsetminus h(\R)$. \\[3mm]
(iii) The case $t_0=0$ and $s_0>0$ is similar to the previous one. To finish the proof, note that it is impossible that the equality $t_0=s_0=0$ holds, because of \setcounter{ffact}{2}Fact~\Roman{ffact}.\qedhere
\end{proof}

\setcounter{ffact}{3}
\begin{ffact}
The restriction $h\colon \R\longrightarrow h(\R)$ is open. 
\end{ffact}

\begin{proof}
We shall show that the images of basic open sets under $h$ are open in $h(\R)$. For that purpose, we divide the proof into several simpler cases: \\[3mm]
(i) First, take $(a,b)\subseteq \R$ such that $(a,b)\cap \Z= \varnothing$. Then, one can easily check that $h((a,b))$ is open in $h(\R)$. \\[3mm]
(ii) Now suppose that $(a,b)=(2n-\varepsilon, 2n+\varepsilon)$ for some $n\in\Z$ and $0<\varepsilon <1$. We want to see that $U=h((a,b))$ is open in $h(\R)$. We have that
$$U= h(\R) \cap f((a,b)) \times g((a,b)),$$
where $f((a,b))= p\left( (1-\varepsilon,1]\times\{n\}\right)$ and $g((a,b))=\left( [0,\varepsilon ) \times \{n-1\}\right) \cup p\left( [0,\varepsilon) \times \{n\} \right)$.
Let us check that $U$ can also be written as follows:
$$U= h(\R) \cap f((a,b)) \times p\left( [0,\varepsilon) \times I\right),$$
(which is obviously open in $h(\R)$). It is enough to show the containment
$$  h(\R) \cap f((a,b)) \times p\left( [0,\varepsilon) \times I\right) \subseteq h(\R) \cap f((a,b)) \times g((a,b)),$$
so let $\left( (t,n), \beta )\right) \in h(\R) \cap   f((a,b)) \times p\left( [0,\varepsilon) \times I\right)$. By \setcounter{ffact}{1}Fact~\Roman{ffact}, one has $\beta= (1-t,n)$ or $\beta=(1-t,n-1)$, and hence $(\alpha,\beta)\in h(\R) \cap f((a,b)) \times g((a,b))$. \\[3mm]
(iii) The odd case (i.e. when $(a,b)=(2n+1-\varepsilon,2n+1+\varepsilon)$) can be proved similarly.\\[3mm]
(iv)  Finally, let us show the general case. Indeed, every interval $(a,b)\subseteq \R$ can be written as the union of open intervals of the form described in (i), (ii) and (iii). Since the direct image of a union and the union of direct images coincide, $h((a,b))$ is open in $h(\R)$. \qedhere
\end{proof}
Finally, we check that $h$ is one-to-one. Let $x,y\in\R$ such that $h(x)=h(y)$, i.e. $f(x)=f(y)$ and $g(x)=g(y)$. Write $x=2n+t$ and $y=2m+s$ with $n,m\in\Z$ and $-1< t,s \leq 1$. We only deal with the case $t,s > 0$ since the other cases can be shown similarly. We have $f(x)=f(y)$, i.e.  $(1-|t|,n)=(1-|s|,m)$ and so there are two possible cases. First, assume that $1-|t|=1-|s|=0$. Since $t,s>0$, one has $t=s=1$. Now, since $g(x)=g(y)$ we obtain that $(1,n)=(1,m)$, and thus $n=m$ and $x=y$, as desired. Assume otherwise that $n=m$. Then $1-|t|=1-|s|$, and therefore $t=s$ because $t,s>0$. Thus, $x=y$, as we wanted to show. Therefore, $h$ is the desired closed embedding.
\end{proof}

\subsection{Kowalsky's Hedgehog Theorem}\label{secKo}
One important application of the metric hedgehog is presented in this subsection. More specifically, we aim to give a proof of the Kowalsky's Hedgehog Theorem. Roughly speaking, this result asserts that every metric space  is embeddable in the product of countably many copies of the metric hedgehog.  We will mainly follow the proof given in \cite[pp.~282-283,\ Theorem~4.4.9]{Eng}. \\[2mm]
The first step is to compute the weight of a countable cartesian power of the metric hedgehog. 
\begin{example}\label{weighthedgehog}
The metric hedgehog has weight $\omega \left( MJ(\kappa) \right) = \aleph_0$ if  $\kappa <\aleph_0$ and  $\omega \left( MJ(\kappa) \right) = \kappa$ if $\kappa \geq \aleph_0$. Indeed, assume first that $\kappa \geq \aleph_0$. A basis of cardinality $\kappa$ is given by
\begin{align*} \beta (\kappa) = & \bigl\{ p( [0,1/n) \times I ) \mid n\in \N \bigr\}\\& \cup \bigl\{ p\left((a,b)\times\{i\} \right) \mid a,b\in \Q ,\ 0\leq a < b \leq 1,\ i\in I \bigr\}\\& \cup \bigl\{ p\left( (1/n,1] \times \{i\}  \right) \mid n\in \N,\ i\in I\bigr\}.\end{align*}
Further, it is impossible to give a basis of a strictly smaller cardinality. Assume that $\tilde\beta$ is another basis of $MJ(\kappa)$. For each $i \in I$, $U_i= p\left((0,1]\times \{i\}\right)$ is open, so there is an element $B_i$ of $\tilde\beta$ such that $B_i\subseteq U_i$. Note that all the $B_i$'s are pairwise disjoint and in particular distinct. Thus $|\tilde\beta|\geq \kappa$ and $\omega\left(MJ(\kappa)\right)=\kappa$. 

Now we deal with the case $\kappa < \aleph_0$. Note that the basis $\beta (\kappa)$ defined above is a countable basis of $MJ(\kappa)$. Thus $MJ(\kappa) \leq \aleph_0$. Let $\tilde\beta$ be another basis of $MJ(\kappa)$ and assume by contradiction that $|\tilde\beta| < \aleph_0$. Then, we would have a finite basis for the subspace $p([0,1]\times \{i\})$, which is homeomorphic to the closed unit interval, a contradiction. 
\end{example}

Example~\ref{weighthedgehog} and Lemma~\ref{prodwrithk} yield

\begin{corollary}\label{corwei}
Let $\kappa\geq \aleph_0$. Then the weight of $\bigl(MJ(\kappa)\bigr)^{\aleph_{0}}$ is $\kappa$. 
\end{corollary}

The following results are devoted to show that every metric space has a $\sigma$-discrete basis (see Definition~\ref{sigmadiscret}). 

\begin{definition}
We say that a cover $\{B_j\}_{j\in J}$ is a \emph{refinement} of a cover $\{A_i\}_{i\in I}$ of the same set $X$, if for every $j\in J$ there is an $i \in I$ with $B_j\subseteq A_i$. 
\end{definition}

\begin{theorem}[Stone]{\rm (Cf.\ \cite[page 280, Theorem 4.4.1]{Eng})}\label{stone}
Every open cover of a met\-rizable space has a $\sigma$-discrete open refinement.
\end{theorem}

\begin{proof}
Let $\{U_i\}_{i\in I}$ be an open cover of a met\-rizable space $X$ with metric $d$. The Well-ordering Theorem guarantees the existence of a well-ordering $<$ on $I$.  For every $n\in \N$, we inductively build a family $\mathcal{V}_n = \{ V_{i,n}\}_{i\in I}$ as follows. For each $i\in I$, set $V_{i,n} = \tbigcup B(c,1/2^n)$, where the union is taken over all the $c\in X$ satisfying the following conditions:\\[2mm]
(1) $i$ is the least element of $I$ such that $c\in U_i$,\\[1mm]
(2) $c\not\in V_{i',j}$ for all $j<n$ and $i' \in I$,\\[1mm]
(3) $B(c,3/2^n)\subseteq U_i$.\\[2mm]
The desired $\sigma$-discrete open refinement will turn out to be $\mathcal{V}=\tbigcup_{n\in N}\mathcal{V}_n$. Firstly, the openess of all element in $\mathcal{V}$ is clear (each $V_{i,n}$ is a union of open balls). Further, let us check that $\mathcal{V}$ is a cover of $X$. Indeed, let $x\in X$ and observe that the set 
$$I_x= \{i\in I \mid x\in U_i\}$$
is nonempty because $\{U_i\}_{i \in I}$ is a cover of $X$.  Then, by the well-ordering of $I$, there is a least element in $I_x$, say $i\in I$.  Since $x\in U_i$, which is open, there is an $n\in \N$ such that $B(x,3/2^n)\subseteq U_i$. Thus $x$ satisfies properties (1) and (3).  We distinguish two cases: if property (2) is also satisfied, one has $x \in V_{i,n}$. If (2) does not hold, there is a $j<n$ and a $i' \in I$ such that $x\in V_{i',j}$. Thus, $\mathcal{V}$ is a cover of $X$. Condition (3) yields that $\mathcal{V}$ is a refinement of $\{U_i\}_{i\in I}$, because
$$B(c,1/2^n) \subseteq B(c,3/2^n) \subseteq U_i \implies V_{i,n} = \tbigcup B(c,1/2^n) \subseteq U_i.$$
The only point remaining is to show that $\mathcal{V}$ is $\sigma$- discrete, i.e. that $\mathcal{V}_n$ is discrete for all $n\in \N$. We will first show that the following property is satisfied for each $n\in \N$:
\begin{equation}\label{propertypn} \tag{P$_n$} x_1 \in V_{i_1,n},\, x_2\in V_{i_2,n} \textrm{ and } i_1\ne i_2 \implies d (x_1,x_2)>\frac{1}{2^n}.\end{equation}

Let  $x_1 \in V_{i_1,n}$,  $x_2\in V_{i_2,n}$ and without loss of generality assume $i_1<i_2$. By the definition of $V_{i_1,n}$ and $V_{i_2,n}$, there exist $c_1,c_2\in X$ satisfying properties (1)--(3) such that $x_1\in B(c_1, 1/2^n)$ and $x_2\in B(c_2, 1/2^n)$. Property (3) yields $B(c_1, 3/2^n) \subseteq U_{i_1}$, and from (1) we have $c_2 \not\in U_{i_1}$ (because $i_1<i_2$). Thus, one has $c_2\not\in B(c_1, 3/2^n)$, i.e. $d(c_1,c_2)\geq 3/2^n$. From the triangle inequality we obtain
$$d(x_1,x_2)\geq d(c_1,c_2) - d(c_1,x_1) -d(c_2,x_2) > \frac{3}{2^n} - \frac{1}{2^n} -\frac{1}{2^n} = \frac{1}{2^n},$$
which proves  (\ref{propertypn}). This property shows that $\mathcal{V}_n$ is discrete. Indeed, let $x\in X$. Then $B(x,1/2^{n+1})\in\mathcal{N}_x$ meets $\mathcal{V}_n$ at most once, for otherwise we would have $i_1\neq i_2$ such that $V_{i_1,n}\cap B(x,1/2^{n+1})\ne \varnothing $ and $V_{i_2,n}\cap B(x,1/2^{n+1})\ne \varnothing$. Hence, there would be $x_1\in V_{i_1,n}$ and $x_2 \in V_{i_2,n}$ such that
 $$d(x_1,x_2) \leq d(x_1,x)+d(x,x_2)< \frac{1}{2^{n+1}}+\frac{1}{2^{n+1}}=\frac{1}{2^n},$$
 a contradiction with (\ref{propertypn}). 
\end{proof}

\begin{theorem}{\rm (Cf.\ \cite[page 281, Theorem 4.4.3]{Eng})}\label{discretebasisexists}
Every met\-rizable space has a $\sigma$-discrete basis.
\end{theorem}

\begin{proof}
Let $X$ be a met\-rizable space. For each $n\in \N$, consider the open cover of $X$ given by the $1/n$-balls, namely $\{B(x,1/n)\}_{x\in X}$. By the previous theorem, there is an open $\sigma$-discrete refinement $\beta_n$ of $\{B(x,1/n)\}_{x\in X}$ for each $n\in \N$. It is clear that the union $\beta=\tbigcup_{n\in \N}\beta_n$ is also an open $\sigma$-discrete cover. Let us check that it is indeed a base of $X$. Let $U$ be open in $X$ and take $x \in U$. For each $n\in\N$, $\beta_n$ is an open cover of $X$, so there is an open $B_n\in\beta_n$ such that $x\in B_n$. Further, for every $n\in \N$ one has that $\beta_n$ refines $\{B(y,1/n)\}_{y\in X}$, and thus there is $y_n\in X$ such that $x\in B_n \subseteq B(y_n,1/n)$. 

Since $x\in U$, which is open, there is $n_0\in \N$ with $x\in B(x,1/n_0)\subseteq U$. Note that then 
$$B\left(y_{2n_0}, \textstyle\frac{1}{2n_0}\right) \subseteq B\left(x,\textstyle\frac{1}{n_0}\right).$$
Indeed, let $z\in X$ such that $d(z,y_{2n_0})< 1/(2n_0)$. Therefore, $$d(z,x)\leq d(z,y_{2n_0}) + d(y_{2n_0}, x) <\frac{1}{2n_0} + \frac{1}{2n_0}=\frac{1}{n_0}.$$
Thus, one has $x\in B_{2n_0} \subseteq U$, and $\beta$ is a basis of $X$. 
\end{proof}

We are finally ready to prove the main theorem in this subsection. This result was first published in \cite{KO} by Hans-Joachim Kowalsky in 1961.

\begin{theorem}[Kowalsky's Hedgehog Theorem] \label{Kowalsky} Let $\kappa \geq \aleph_0$. Then the cartesian product  $\bigl(MJ(\kappa)\bigr)^{\aleph_{0}}$ of $\aleph_0$ copies of the metric hedgehog of $\kappa$ spines is universal among all met\-rizable spaces of weight $\kappa$. 
\end{theorem}

\begin{proof}
Let $\kappa \geq \aleph_0$, and denote $Z=\bigl(MJ(\kappa)\bigr)^{\aleph_{0}}$. By Corollary~\ref{corwei}, the weight of $Z$ is $\kappa$, and by Theorem~\ref{countablemetrizable} one has that $Z$ is met\-rizable. Thus $Z$ belongs to the class of all met\-rizable spaces of weight $\kappa$. 

Let $X$ be a met\-rizable space of weight $\kappa$. Theorem~\ref{discretebasisexists} yields a \break $\sigma$-discrete basis $\beta$, that is, $\beta = \dot\tbigcup_{n\in\N} \beta_n$ with each $\beta_n=\{U_i\}_{i\in I_n}$ discrete. By Lemma~\ref{smallerbasis} we can assume that  the cardinality of $I= \dot\tbigcup_{n\in \N} I_n$ is $\kappa$; for otherwise we could take a basis  $\beta'\subseteq \beta$ such that $|\beta'|=\kappa$, whose existence is asserted by the lemma, and clearly $\beta'$ would also be $\sigma$-discrete. Hence we can assume that the index set $I$ coincides with the index set used to define the hedgehog $MJ(\kappa)$.

Let $n\in \N$.  By virtue of Lemma~\ref{lemmabd1} we can assume that the metric in $X$ is bounded by $1$. For each  $i\in I_n$, let $f_i \colon X\longrightarrow [0,1]$ denote the continuous mapping $f_{X \smallsetminus U_i}$ defined in Lemma~\ref{distcontclosed}.  By the same lemma, since $X\smallsetminus U_i$ is closed  one has $X\smallsetminus U_i = f_i^{-1}\left(\{0\}\right)$. Hence $U_i =[f_i>0]$. Let  $\psi_i\colon [0,1]\longrightarrow p\left( [0,1]\times \{i\} \right)$ be the homeomorphism given by $\psi_i (t)=(t,i)$ (note that the expression makes sense because the index sets coincide). Then we define 
 \[\begin{array}{rcclll}
 G_n \colon&\tbigcup_{i\in I_n} \overline{U_i}&\longrightarrow&MJ(\kappa)&&\cr
 &x&\longmapsto&G_n(x)  = (\psi_i \circ f_i)(x) \quad \textrm{if } x\in \overline{U_i}.
 \end{array}\]

Note that $G_n$ is the combined mapping with respect to the closed discrete (see Lemma~\ref{closdiscrete}) cover $\{\overline{U_i}\}_{i\in I_n}$ of $\tbigcup_{i\in I_n} \overline{U_i}$ .  Since each of the functions in the family is continuous, Proposition~\ref{propcombinedmap} yields the continuity of $G_n$. Define
 \[\begin{array}{rcclll}
 F_n \colon&X&\longrightarrow&MJ(\kappa)&&\cr
 &x&\longmapsto&F_n(x)  =\begin{cases} G_n(x) & \textrm{if } x\in \tbigcup_{i\in I_n} \overline{U_i}; \\ \zerohedgehog & \textrm{if } x\in X\smallsetminus \tbigcup_{i\in I_n} U_i.\end{cases}
 \end{array}\]
 We have that $\tbigcup_{i\in I_n} \overline{U_i}$ is closed (by Proposition~\ref{unionclosed}). Clearly, $X\smallsetminus \tbigcup_{i\in I_n} U_i$ is also closed.  Let us check that $F_n$ is a well-defined combined map (i.e. it comes from compatible mappings). Indeed, if $x\in  \tbigcup_{i\in I_n} \overline{U_i} \cap \left(X\smallsetminus \tbigcup_{i\in I_n} U_i \right)$, then $x\in \overline{U_i}$ and $x\in X\smallsetminus U_i$ for some $i\in I_n$.  In other words, we have $G_n(x) = (\psi_i \circ f_i)(x)$ and $f_i(x)=0$. Thus $G_n(x)=\psi_i(0) =\zerohedgehog$ and $F_n$ is well defined. Hence, the Pasting Lemma gives that $F_n$ is continuous. 
 
 In the last step of the proof we use the Diagonal Theorem. By Remark~\ref{remT1} it is enough to show that  the family $\{F_n\}_{n\in\N}$ separates points and closed sets.  Let $x\in X$ and  $F\subseteq X$ a closed subset such that $x\not\in F$. We have $x\in X\smallsetminus F$, which is open, and since $\beta$ is a basis of $X$, there exist $n\in \N$ and $i\in I_n$ such that $x\in U_i \subseteq X\smallsetminus F$, from which we obtain 
 \begin{align*} F_n(x) \in F_n(U_i) &=G_n (U_i) = \psi_i \left(f_i (U_i) \right) \\&= \psi_i \left( f_i \left( f_i^{-1} \left((0,1]\right) \right) \right) \subseteq   \psi_i ( (0,1] ) = p \left((0,1]\times \{i\}\right).\end{align*}
Further, we have  $F_n(F) \tbigcap  p \left((0,1]\times \{i\}\right)=\varnothing$. Indeed, if $F_n(y) \in F_n(F)$ (with $y\in F$) and $F_n(y) \in  p \left((0,1]\times \{i\}\right)$, one necessarily has $y\in U_i$, and thus $y\in F \cap U_i$, a contradiction (note that $U_i\subseteq X\smallsetminus F$).

Hence, it follows that $F_n(x)  \in p \left((0,1]\times \{i\}\right) \subseteq MJ(\kappa)\smallsetminus F_n (F)$. Note that $p \left((0,1]\times \{i\}\right)$ is open in the metric hedgehog, and thus 
$$F_n(x)\in \textrm{int} \left( MJ(\kappa) \smallsetminus F_n(F)\right)= MJ(\kappa) \smallsetminus \overline{F_n(F)},$$ i.e. $F_n(x)\not\in \overline{F_n(F)}$, as desired. By the Diagonal Theorem (see Theorem~\ref{diagonal}) the diagonal mapping is an embedding $\Delta_{n\in \N} F_n \colon X \longrightarrow Z$.
\end{proof}

 Kowalsky's hedgehog theorem yields many interesting corollaries, and in what follows we give some of them.

The following result is the solution of \cite[page~286,\ Problem~4.4.C.\,(a)]{Eng}.

\begin{theorem}\label{app1}
Let $X$ be a met\-rizable topological space of weight $\kappa \leq \mathfrak{c}$. Then there exists a continuous bijection $F \colon X\longrightarrow Y$ onto a separable met\-rizable space $Y$.
\end{theorem}

\begin{proof}
First we will show that there is a continuous one-to-one mapping from $MJ(\mathfrak{c})$ into the plane with the usual topology. Assume that the index set $I$ is $(0,+\infty)$. Consider the  following natural mapping that sends each spine into the corresponding segment in the plane:
 \[\begin{array}{rcclll}
 f\colon&MJ(\mathfrak{c})&\longrightarrow&(\R^2,\tau_u)&&\cr
 &(t,s) &\longmapsto& f(t,s)=\displaystyle\frac{t}{1+s^2} (1,s).
 \end{array}\]
 Clearly $f$ is well-defined and one-to-one. Let us check that $f$ is continuous. Take $(t_0,s_0)\in MJ(\mathfrak{c})$ and $\varepsilon > 0$. First, assume that $t_0 >0$. Let $\delta=\varepsilon \wedge t_0 > 0$. Now, for every $(t,s)\in MJ(\mathfrak{c})$ such that $d\left( (t,s), (t_0,s_0) \right) < \delta$, one necessarily has $s=s_0$  because $\delta < t_0$. An easy calculation shows that
 $$d\left( f(t,s), f(t_0,s_0) \right) = d \left( f(t,s_0), f(t_0,s_0) \right) = | t-t_0| = d\left( (t,s), (t_0,s_0) \right) < \varepsilon,$$
 which proves continuity at $(t_0,s_0)$. Suppose now that $t_0=0$ and let $\delta=\varepsilon$. Then, for every $(t,s)\in MJ(\mathfrak{c})$ such that $d( (t,s), \zerohedgehog)= t <\delta$, one has
 $$d\left( f(t,s), f(\zerohedgehog) \right) = d\left( f(t,s), (0,0) \right) = t <\varepsilon,$$
and thus $f$ is continuous. 

Let $f^*$ denote the product map $f^* \colon \left( MJ(\mathfrak{c})\right)^{\aleph_0} \longrightarrow \left( \R^2 \right)^{\aleph_0}$ assigning to the point $x=\{x_n\}_{n\in\N}$ the point $\{ f(x_n) \}_{n\in \N}$. Since $f^*$ is a product map, it is continuous, and $f^*$ is one-to-one because so is $f$.

To prove the assertion of the theorem, let $X$ be a met\-rizable space of weight $\kappa \leq \mathfrak{c}$. By virtue of Kowalsky's hedgehog theorem, in particular there is a continuous one-to-one mapping $h\colon X \longrightarrow \left( MJ(\kappa) \right)^{\aleph_0}$, and since  $\left( MJ(\kappa) \right)^{\aleph_0}$ embeds in $ \left( MJ(\mathfrak{c}) \right)^{\aleph_0}$, in particular there is a continuous one-to-one mapping $g\colon \left( MJ(\kappa) \right)^{\aleph_0} \longrightarrow\left( MJ(\mathfrak{c}) \right)^{\aleph_0}$. Finally, set  $F=  f^{*} \circ g \circ h \colon X \longrightarrow F(X)$.  We have that $H$ is a bijective continuous map. Further,  $F(X)$ is separable because of Lemma~\ref{SepCountProd} and Corollary~\ref{SepHered}.
\end{proof}

The following result is an extension of the  Kowalsky's hedgehog theorem which  provides a condition for a metric space to be embeddable as a \emph{closed} subspace of $\bigl(MJ(\kappa)\bigr)^{\aleph_{0}}$. The proof is a solution for \cite[page 286, Problem~4.4.B.]{Eng}.

\begin{theorem}\label{app2}
Let $\kappa \geq \aleph_0$. Then every completely met\-rizable space of weight $\kappa$ embeds as a closed subspace of $\bigl(MJ(\kappa)\bigr)^{\aleph_{0}}$.
\end{theorem}

\begin{proof}
Let $X$ be a completely met\-rizable space of weight $\kappa$ with $\kappa \geq \aleph_0$, and denote $Z=\bigl(MJ(\kappa)\bigr)^{\aleph_{0}}$. By  Kowalsky's hedgehog theorem, $X$ embeds in $Z$. Let $Y \subseteq Z$ be the set such that $X$ and $Y$ homeomorphic. Since complete metrizability is a topological property (see Remark~\ref{compmetrtopprop}), one has that $Y$ is completely met\-rizable. Thus, by virtue of Theorem~\ref{complmetrizGdelta}, $Y$ is a $G_\delta$ set. Now, by Lemma~\ref{closedofcartesian}, it follows that $Y$ (and therefore $X$) is homeomorphic to a closed subset of $Z\times \R^{\aleph_0}$. Proposition~\ref{Rembeds} tells us that $\R$ embeds in $\bigl( MJ(\aleph_0) \bigr)^2$ as a closed subset. Further, it is clear that $MJ(\aleph_0)$ embeds as a closed subset of $MJ(\kappa)$ because $\kappa \geq \aleph_0$. Thus, $X$ embeds in $Z \times \bigl(\bigl(MJ(\kappa)\bigr)^2\bigr)^{\aleph_{0}} \cong Z$ as a closed subset. 
\end{proof}

\subsection{The metric hedgehog and collectionwise normality} \label{secNor}

Using induction, it is easy to check that in a normal space any finite family of pairwise disjoint closed sets can be separated by a family open pairwise disjoint open sets:

\begin{proposition} \label{propfiniteclosed}
Let $\{F_n\}_{n=1}^k$ be a finite family of pairwise disjoint closed sets in a normal space $X$.  Then, there is a family $\{U_n\}_{n=1}^k$ of open pairwise disjoint sets such that $F_n \subseteq U_n$ for all $n=1,\dots,k$.
\end{proposition}

Regarding Proposition~\ref{propfiniteclosed}, a natural question in a normal space is whether we can consider larger families of closed disjoint sets or not. Even in the countably infinite case, the answer is negative without additional assumptions:

\begin{example}
Consider the real line with the usual topology, which is a normal space, and take the countable pairwise disjoint family $\{ \{x\} \mid x\in \mathbb{Q} \}$ of rational singletons (they are closed).  It is clear that we cannot find a pairwise disjoint family of open sets each one containing a rational number.  Thus, Proposition~\ref{propfiniteclosed} fails when we replace \emph{finite} with \emph{countable}. \end{example}

Nevertheless, if we replace pairwise disjointness with a stronger condition, we will get a positive result. As one may expect, the required condition is discreteness. 

\begin{theorem}\label{aleph0true}
Let $X$ be a normal space and $\{F_n\}_{n\in \mathbb{N}}$ a discrete family of closed subsets of $X$. Then, there is a family $\{U_n\}_{n\in\mathbb{N}}$ of pairwise disjoint open subsets such that $F_n \subseteq U_n$ for each $n\in \mathbb{N}$.
\end{theorem}

\begin{proof}
Since the family $\{F_n\}_{n\in \N}$ of closed sets is discrete, Proposition~\ref{unionclosed} gives that $\tbigcup_{n\in\N} F_n$ is closed. 

Let us construct the desired family of open sets. For each $n\in\N$, the sets 
$F_{n}$ and  $\tbigcup_{m\neq n} F_m = \left(\tbigcup_{m\in \N} F_m \right) \smallsetminus F_n$ are closed and disjoint. Since $X$ is normal, there exist two open and disjoint sets $U_n$ and $V_n$ satisfying $F_n \subseteq U_n$ and $\tbigcup_{m\neq n} F_m \subseteq V_n$. Finally, define 
$$W_1=U_1 \quad \textrm{ and } \quad W_n=U_n\cap V_1\cap \dots \cap V_{n-1}, \quad \textrm{for each } n > 1.$$
We shall show that $\{W_n\}_{n\in \N}$ is the desired family. Clearly $W_n$ is open for each $n\in\N$. Also $W_n\cap W_m =\varnothing$ whenever $n\neq m$. Indeed, let $n\neq m$ in $\N$ and without loss of generality assume that $n<m$. Then $W_m = U_m \cap V_1\cap \dots\cap V_n \cap V_{n+1}\cap\dots \cap V_{m-1}\subseteq V_n$ and hence $W_n\cap W_m \subseteq U_n\cap U_m =\varnothing$.

We still have to check that $F_n \subseteq W_n$ for all $n\in \N$. Let $n\in \N$. The case $n=1$  is clear, so assume that $n>1$. Note that for all $j=1,\dots,n-1$ we have $F_n\subseteq \tbigcup_{m\neq j} F_m \subseteq V_j$ and also $F_n\subseteq U_n$, and thus we get $F_n \subseteq U_n \cap V_1 \cap \dots \cap V_{n-1}$, that is, $F_{n}\subseteq W_{n}$, which completes the proof.
\end{proof}

However, the analogue of Theorem~\ref{aleph0true} for larger families of closed sets is false in general. An example of such a space is the Bing's Space and its construction can be found in \cite{Eng}. In view of this fact, we introduce the following terminology:

\begin{definition}
Let $X$ be a topological space and $\kappa\ge 2$ some cardinal. We say that $X$ is \emph{$\kappa$-collectionwise normal} if for every discrete family $\{F_i\}_{i\in I}$ of closed subsets with $| I | =\kappa$ there exists a family $\{U_i\}_{i\in I}$ of pairwise disjoint open subsets such that $F_i \subseteq U_i$ for every $i\in I$. 
Further, we say that $X$ is \emph{collectionwise normal} if $X$ is $\kappa$-collectionwise normal for each cardinality $\kappa$.\end{definition}

\begin{remarks}\label{impliesnormal}
(1) Note that $2$-collectionwise normality coincides with normality (cf. Lemma \ref{propertiesdisjointdiscrete}). Further,  take $\kappa \leq \aleph_0$. Then, Proposition~\ref{aleph0true} is telling us that a space is normal if and only if it is $\kappa$-collectionwise normal.\\[3mm]
(2) It is clear that if $\kappa \leq \lambda$ are two cardinalities, then $\lambda$-collectionwise normality forces $\kappa$-collectionwise normality. In particular, taking (1) into account, $\kappa$-collectionwise normality implies normality for every $\kappa$. 
\end{remarks}

\begin{lemma}{\rm (Cf. \cite[page 305, Theorem 5.1.17]{Eng})}\label{alsodisc}
A topological space $X$ is $\kappa$-collectionwise normal if and only if  for every discrete family $\{F_i\}_{i\in I}$ of closed subsets in $X$ with $| I | =\kappa$ there exists a discrete family $\{U_i\}_{i\in I}$ of open subsets in $X$ such that $F_i \subseteq U_i$ for every $i\in I$. 
\end{lemma}

\begin{proof}
The ``if'' part is clear by taking into account Proposition \ref{propertiesdisjointdiscrete}. For the converse, suppose that $X$ is $\kappa$-collectionwise normal and let $\{F_i\}_{i\in I}$ be a discrete family of closed sets in $X$ with $| I | =\kappa$. By $\kappa$-collectionwise normality, there is a pairwise disjoint family $\{U_i\}_{i\in I}$ of open sets in $X$ satisfying $F_i\subseteq U_i$ for all $i\in I$. By Lemma \ref{unionclosed} one has that $\tbigcup_{i\in I} F_i$ is closed. Since $X$ is normal, there are disjoint open sets $U$ and $V$ such that $\tbigcup_{i\in I} F_i \subseteq U$ and $X\smallsetminus \tbigcup_{i\in I} U_i \subseteq V$. One easily checks that the family $\{V_i\}_{i\in I}$  where $V_i = U\cap U_i$ is the desired discrete family.
\end{proof}

\begin{proposition}
Every met\-rizable space is collectionwise normal.
\end{proposition}

We have omitted the proof since this last result is unuseful for our purpose (cf.\ \cite[page\ 333, Theorem 5.4.8]{Eng}).\\[2mm]
We are now interested in determining whether a topological space is  $\kappa$-collectionwise normal or not. The metric hedgehog will be the key to generalize the well known Tietze's extension theorem, and using the hedgehog we will finally provide a characterization for  $\kappa$-collectionwise normal spaces (see Theorem~\ref{colnorTh}). First we need to show that collectionwise normality is hereditary with respect to $F_\sigma$-sets. 

\begin{theorem}\label{Fsigmahereditary}
$\kappa$-collectionwise normality is hereditary with respect to $F_{\sigma}$-sets.
\end{theorem}

\begin{proof}
Let $X$ be $\kappa$-collectionwise normal and $F=\tbigcup_{n\in\N} F_n$ an $F_{\sigma}$-set, where $F_n$ is closed in $X$ for all $n\in\N$. Our goal is to show that $F$ is also $\kappa$-collectionwise normal. Let $\{A_i\}_{i\in I}$ be a discrete family in $F$ of closed sets in $F$ with $| I | =\kappa$. First, for each $n\in \N$, we will inductively build discrete families (in $X$)  $\{U_{i}^{n} \}_{i\in I}$ and $\{V_{i}^{n}\}_{i\in I}$ of open sets in $X$ satisfying
\begin{equation} (A_i \cap F_m) \cup \overline{U_{i}^{m-1}} \subseteq U_{i}^{m} \subseteq \overline{ U_{i}^{m}  } \subseteq V_{i}^{m} \smallsetminus \overline{\tbigcup_{j\ne i} A_j}\quad \forall i\in I,\ \forall m=1,\dots,n\tag{P$_n$}\label{Pnstone} \end{equation}
(we define $U_{i}^{0}=\varnothing$).
For $n=1$, we have to construct discrete families $\{U_{i}^{1}\}_{i\in I}$ and $\{V_{i}^{1} \}_{i\in I}$ of open sets in $X$  satisfying 
$$A_i \cap F_1 \subseteq U_{i}^{1} \subseteq \overline{U_{i}^{1}} \subseteq V_{i}^{1}\smallsetminus \overline{\tbigcup_{j\ne i} A_j} \quad \forall i\in I.$$
Since $\{A_i\}_{i\in I}$ is discrete in $F$, it is straightforward to verify that $\{A_i\cap F_1\}_{i\in I}$ is discrete in $X$. Thus $\{ A_i \cap F_1\}_{i\in I}$ is a discrete family of closed sets in $X$. By $\kappa$-collectionwise normality and Lemma \ref{alsodisc}, there is a discrete  family $\{V_{i}^{1}\}_{i\in I}$ of open sets in $X$ satisfying $A_i\cap F_1\subseteq V_{i}^{1}$ for each $i\in I$.  Further, one has $ A_i\cap F_1 \subseteq V_{i}^{1}\smallsetminus \overline{\tbigcup_{j\ne i} A_j}$ for every $i\in I$. By way of contradiction assume that $x\in A_i\cap F_1$ and $x\in \overline{\tbigcup_{j\ne i} A_j} $. Note that
$$x\in F\cap \overline{\tbigcup_{j\ne i} A_j} = \overline{\tbigcup_{j\ne i} A_j}^{F}=\tbigcup_{j\ne i} A_j,$$
where the last equality follows from Proposition \ref{unionclosed} and the fact that $\{A_i\}_{i\in I}$ is a discrete family of closed sets in $F$. Thus $x\in A_i\cap A_j$ for some $i\ne j$, which contradicts the discreteness of $\{A_i\}_{i\in I}$. Now, since $X$ is normal, $A_i\cap F_1$ is closed in $X$ and  $ V_{i}^{1}\smallsetminus \overline{\tbigcup_{j\ne i} A_j} $ is open in $X$  for each $i\in I$, there is an open set $U_{i}^{1}$ in $X$ such that
$$A_i \cap F_1 \subseteq U_{i}^{1} \subseteq \overline{U_{i}^{1}} \subseteq V_{i}^{1}\smallsetminus \overline{\tbigcup_{j\ne i} A_j}.$$
The observation that $\{U_{i}^{1}\}_{i\in I}$ is also discrete (because so is $\{V_{i}^{1}\}_{i\in I}$ and $U_{i}^{1}\subseteq V_{i}^{1}$ for every $i\in I$) finishes the proof of the case $n=1$.\\[3mm]
Now suppose that we have built discrete families (in $X$) $\{U_{i}^{m} \}_{i\in I}$ and $\{V_{i}^{m}\}_{i\in I}$ of open sets in $X$ satisfying (\ref{Pnstone}) for each $m=1,\dots,n$.  Let us show that the family 
$$\bigl\{ (A_i\cap F_{n+1} ) \cup \overline{U_{i}^{n}}\bigr\}_{i\in I}$$
 is discrete in $X$. We shall check that both conditions in Lemma \ref{charactDisc} are verified:\\[2mm]
(1) Let $x\in X$. Discreteness of $\{\overline{U_{i}^{n}} \}_{i\in I}$ yields $N\in\NN_x$ such that $N$ meets at most one member of $\{\overline{U_{i}^{n}} \}_{i\in I}$. We distinguish two cases: if $x\not\in F_{n+1}$, $M=N\cap(X\smallsetminus F_{n+1})$ is a neighborhood of $x$ in $X$ such that $M$ meets at most one member of $\bigl\{ (A_i\cap F_{n+1} ) \cup \overline{U_{i}^{n}}\bigr\}_{i\in I}$. Otherwise, if $x\in F_{n+1}\subseteq F$, discreteness of $\{A_i\}_{i\in I}$ (in $F$) implies the existence of a neighborhood $F\cap M$ of $x$ in $F$ (where $M$ is a neighborhood of $x$ in $X$) such that $F\cap M$ meets at most one member of the family $\{A_i\}_{i\in I}$. Now, one easily verifies that $M' = N\cap M$ is a neighborhood of $x$ in $X$ which intersects at most 2 members in the family $\bigl\{ (A_i\cap F_{n+1} ) \cup \overline{U_{i}^{n}}\bigr\}_{i\in I}$, which proves that it is locally finite. \\[2mm]
(2) Let $i\ne j$ in $I$. We will show that $\overline{(A_i\cap F_{n+1} ) \cup \overline{U_{i}^{n}} }\cap \overline{ (A_j\cap F_{n+1} ) \cup \overline{U_{j}^{n}} } = \varnothing.$
Indeed, we have
\begin{align*}& \overline{(A_i\cap F_{n+1} ) \cup \overline{U_{i}^{n}} }\cap \overline{ (A_j\cap F_{n+1} ) \cup \overline{U_{j}^{n}} } 
\\& \qquad  = \bigl(  \overline{A_i\cap F_{n+1} } \cup \overline{U_{i}^{n}} \bigr) \cap \bigl( \overline{ A_j\cap F_{n+1} } \cup \overline{U_{j}^{n}}\bigr)  
\\&\qquad \subseteq     \bigl(  (\overline{A_i}\cap F_{n+1} ) \cup \overline{U_{i}^{n}} \bigr) \cap \bigl(( \overline{ A_j}\cap F_{n+1} )\cup \overline{U_{j}^{n}}\bigr)   
\\&\qquad=  \bigl( \overline{A_i}\cap \overline{A_j}\cap F_{n+1} \bigr) \cup \bigl(  \overline{U_{i}^{n}} \cap  \overline{U_{j}^{n}} \bigr) \cup  \bigl( \overline{A_i}\cap F_{n+1} \cap \overline{U_{j}^{n}}\bigr) \cup \bigl( \overline{A_j}\cap F_{n+1} \cap \overline{U_{i}^{n}}\bigr)  .
\end{align*}
Observe that $ \overline{A_i}\cap \overline{A_j}\cap F_{n+1} \subseteq \overline{A_i}^{F}\cap \overline{A_j}^{F} =\varnothing$ because of discreteness of $\{A_i\}_{i\in I}$ in $F$. Similarly, $ \overline{U_{i}^{n}} \cap  \overline{U_{j}^{n}}=\varnothing$ because of discreteness of $\{U_{i}^{n}\}_{i\in I}$ in $X$. Besides, one has  
$$ \overline{A_i}\cap F_{n+1} \cap \overline{U_{j}^{n}} \subseteq \overline{A_i} \cap \overline{U_{j}^{n}} \subseteq \overline{A_i} \cap \bigl(V_{j}^{n} \smallsetminus \overline{\tbigcup_{k\ne j} A_k}\bigr) \subseteq \overline{A_i} \cap (V_{j}^{n} \smallsetminus \overline{A_i})=\varnothing.$$
Similarly, one has $\overline{A_j}\cap F_{n+1} \cap \overline{U_{i}^{n}}=\varnothing$.  Thus, condition (2) is also satisfied.\\[3mm]
Now, $\bigl\{ (A_i\cap F_{n+1} ) \cup \overline{U_{i}^{n}}\bigr\}_{i\in I}$  is a discrete family of closed subsets in $X$ and hence, by $\kappa$-collectionwise normality of $X$ and Lemma \ref{alsodisc}, there is a discrete family (in $X$) $\{V_{i}^{n+1}\}_{i\in I}$ of open sets in $X$ satisfying $(A_i\cap F_{n+1} ) \cup \overline{U_{i}^{n}} \subseteq V_{i}^{n+1}$ for every $i\in I$. In fact, we also have 
$$(A_i\cap F_{n+1} ) \cup \overline{U_{i}^{n}} \subseteq V_{i}^{n+1} \smallsetminus  \overline{\tbigcup_{j\ne i} A_j}.$$
By way of contradiction, assume that $x\in (A_i\cap F_{n+1} ) \cup \overline{U_{i}^{n}} $ and $x\in \overline{\tbigcup_{j\ne i} A_j}$. First, note that necessarily $x\in A_i \cap F_{n+1}$ (since $x\in \overline{U_{i}^{n}}$ would imply  $x\in V_{i}^{n} \smallsetminus  \overline{\tbigcup_{j\ne i} A_j}$, thus $x\not\in  \overline{\tbigcup_{j\ne i} A_j}$). Since $x\in F_{n+1}\subseteq F$, one has 
$$x\in F\cap \overline{\tbigcup_{j\ne i} A_j} = \overline{\tbigcup_{j\ne i} A_j}^{F}= \tbigcup_{j\ne i} A_j,$$
where the last equality follows from Proposition \ref{unionclosed} and the fact that $\{A_i\}_{i\in I}$ is a discrete family of closed sets in $F$. Thus $x\in A_i\cap A_j$ for some $j\ne i$ which contradicts the pairwise disjointness of $\{A_i\}_{i\in I}$.\\[2mm]
By normality of $X$, for  each $i\in I$ there is an open set $U_{i}^{n+1}$ in $X$ verifying
$$(A_i\cap F_{n+1} ) \cup \overline{U_{i}^{n}}\subseteq U_{i}^{n+1}\subseteq \overline{U_{i}^{n+1}} \subseteq V_{i}^{n+1} \smallsetminus  \overline{\tbigcup_{j\ne i} A_j}.$$
Observe that the family $\{U_{i}^{n+1}\}_{i\in I}$ is also discrete because $\{ V_{i}^{n+1} \}_{i\in I}$ is discrete and $U_{i}^{n+1} \subseteq V_{i}^{n+1}$ for every $i\in I$. Thus the existence of the desired families $\{ U_{i}^{n+1}\}_{i\in I}$ and $\{ V_{i}^{n+1} \}_{i\in I}$ is proved. Finally, set 
$$ U_i = F\cap  \tbigcup_{n\in\N} U_{i}^{n}$$
 for each $i\in I$. Clearly $\{U_i\}_{i\in I}$ is a family of open sets in $F$ with the property that $A_i\subseteq U_i$ for every $i\in I$. We finish the proof if we show that the family is also pairwise disjoint. Assume that $i\ne j$ in $I$. By way of contradiction assume that there is $x\in U_i \cap U_j$. Then $x\in U_{i}^{n_1} \cap U_{j}^{n_2}$ for some $n_1,n_2\in \N$. Without loss of generality we may assume that $n_1\leq n_2$.  Property (\ref{Pnstone}) implies that $U_{i}^{n_1}\subseteq U_{i}^{n_2}$, and hence $x\in U_{i}^{n_2}\cap U_{j}^{n_2}$, which contradicts the discreteness of the family $\{U_{i}^{n_2}\}_{i\in I}$. Hence $F$ is $\kappa$-collectionwise normal.
\end{proof}

We now present the announced characterization of $\kappa$-collectionwise normality. We give a fully detailed version of the proof, based on the outlines of the proof provided in \cite{Prz} and \cite[page 337, Problem 5.5.1(c)]{Eng}.

\begin{theorem}\label{colnorTh}
Let $\kappa > 1$ be some cardinality and $X$ a topological space. The following are equivalent:
\begin{enumerate}[\normalfont (i)]
\item $X$ is $\kappa$-collectionwise normal;
\item For every continuous mapping $f\colon A\longrightarrow MJ(\kappa)$ of a closed subspace $A$ of $X$ into the metric hedgehog, there exists a continuous mapping $F\colon X\longrightarrow MJ(\kappa)$ such that $F|_A =f$.
\end{enumerate}
\end{theorem}

\begin{proof}
(i)$\implies$(ii). Let $A$ be a closed subspace of a collectionwise normal space $X$ and let $f\colon A\longrightarrow MJ(\kappa)$ be a continuous mapping. The composition $\pi_\kappa \circ f\colon A\longrightarrow [0,1]$ is continuous (note that $\pi_\kappa$ is continuous  because of Proposition~\ref{inittop}). Since we are in a normal space (see Remark~\ref{impliesnormal}),  Tietze's extension theorem yields the existence of a continuous $G\colon X \longrightarrow [0,1]$ such that $G|_A = \pi_\kappa \circ f$. For each $i\in I$, we set $F_i = f^{-1} \left( p\left( (0,1]\times\{i\} \right) \right).$ We will now show that $\{F_i\}_{i\in I}$ is a discrete family of closed sets in $G^{-1}\left( \left(0,1\right] \right)$. First note that for each $i\in I$ one has
\begin{align*} F_i &= \bigl\{ x\in A \mid f(x) \in  p\left( (0,1]\times\{i\} \right)\bigr\} \\
& =  \bigl\{ x\in A \mid f(x) \in  p\left( [0,1]\times\{i\} \right), (\pi_\kappa \circ f)(x)\in (0,1]\bigr\}  \\
& =  \bigl\{ x\in A \mid f(x) \in  p\left( [0,1]\times\{i\} \right),G(x)\in (0,1]\bigr\}  \\
& = f^{-1} \left( p\left( [0,1]\times\{i\} \right) \right)\cap G^{-1}\left( (0,1] \right),
\end{align*}
which shows that $F_i$ is closed in $G^{-1}((0,1])$. Now we prove the discreteness. Let $x\in G^{-1} ((0,1])$. We distinguish two cases. First, if $x\not\in A$, we have that $x \in N= G^{-1} ( (0,1] ) \cap (X\smallsetminus A)$, which is an  open neighborhood of $x$ in $G^{-1} ( (0,1] )$. Since $F_i \subseteq A$ for all $i\in I$, it follows that $N$ does not intersect any member in the family $\{F_i\}_{i\in I}$.  Now assume that $x\in A$. Since $x\in G^{-1}((0,1])$, we have  $f(x) \in p\left( (0,1]\times \{i_0\} \right)$ for some $i_0\in I$. Hence, $x\in N=F_{i_0}$, which is an open neighborhood of $x$ in $G^{-1}((0,1])$. Observe that  $F_i \cap N =\varnothing$ for all $i\ne i_0$. Hence $\{F_i\}_{i\in I}$ is discrete in $G^{-1}( (0,1])$. \\[2mm]
Note that $G^{-1}( (0,1] )= \tbigcup_{n\in \N} G^{-1}\left( [1/n, 1] \right)$ is an $F_\sigma$-set, and by Theorem~\ref{Fsigmahereditary} it follows that  $G^{-1} ((0,1])$ is $\kappa$-collectionwise normal (because so is $X$). Therefore, there exists a family $\{U_i\}_{i\in I}$ of pairwise disjoint open sets in $G^{-1}( (0,1])$ (and thus they are also open in $X$, because by continuity of $G$, one has that $G^{-1}( (0,1])$ is open in $X$) such that $F_i \subseteq U_i$ for every $i\in I$. 

Note that the sets $A$ and $X\smallsetminus \tbigcup_{i\in I}U_i$ are closed in $X$. We also have
\begin{align*} A \cap \bigl( X \smallsetminus \tbigcup_{i\in I}U_i \bigr) & =  A \smallsetminus \tbigcup_{i\in I}U_i 
\subseteq A \smallsetminus \tbigcup_{i\in I}F_i \\& = f^{-1} \bigl( MJ(\kappa) \smallsetminus \tbigcup_{i\in I} p \left( (0,1] \times \{i\} \right)\bigr) = f^{-1} \left( \left\{\zerohedgehog \right\}\right). \end{align*}

Hence, if $x$ is in the intersection above, one has $(\pi_\kappa \circ f)(x)=0$, and so it is ensured that the following combined map is well-defined:
 \[\begin{array}{rcclll}
h \colon&A\cup \bigl(X\smallsetminus \tbigcup_{i\in I} U_i\bigr)&\longrightarrow&[0,1]&&\cr
 &x&\longmapsto&h(x)  = \begin{cases} (\pi_\kappa \circ f)(x)  & \textrm{if } x\in A;\\
 0 & \textrm{if } X\smallsetminus \tbigcup_{i\in I} U_i.
 \end{cases}
 \end{array}\]
 Clearly, $h$ is continuous in each of the closed sets $A$ and $X\smallsetminus \tbigcup_{i\in I_n} U_i$, and therefore the Pasting Lemma yields the continuity of $h$. By virtue of the Tietze's extension theorem again, we extend $h$ to a continuous mapping $H\colon X \longrightarrow [0,1]$. Finally, since the family $\{U_i\}_{i\in I}\cup\{ X\smallsetminus \tbigcup_{i\in I} U_i\}$ is a pairwise disjoint cover of $X$, we define a new map as follows:
  \[\begin{array}{rcclll}
F \colon&X&\longrightarrow&MJ(\kappa)&&\cr
 &x&\longmapsto&F(x)  = \begin{cases} (H(x),i)  & \textrm{if } x\in U_i;\\
 \zerohedgehog & \textrm{if } x\in X\smallsetminus \tbigcup_{i\in I} U_i.
 \end{cases}
 \end{array}\]
 
 We will show that $F$ is the desired continuous extension of $f$.  By Proposition~\ref{univpropinitial}, we know that $F$ is continuous if and only if $\pi_\kappa \circ F$ and $\pi_i \circ F $ are continuous for every $i\in I$. Note that $\pi_\kappa \circ F = H$ and $\pi_i \circ F = \chi_{U_i} \wedge H$ for every $i\in I$. Since $H$ is already continuous, it is enough to show that $\chi_{U_i} \wedge H$ is continuous for each $i\in I$. Let $i \in I$ and $t\in [0,1)$. We have
$$[\chi_{U_i} \wedge H> t]= \{ x\in X \mid \chi_{U_i} (x) >t ,\ H(x)>t\} = U_i \cap [H>t],$$
which is open in $X$ because it is the intersection of two open sets.  

Similarly,  for every $i\in I$ and $t\in (0,1]$ one has
\begin{align*} [\chi_{U_i} \wedge H< t] &=  \{ x\in X \mid \chi_{U_i} (x) <t \textrm{ or } H(x)<t\}= (X\smallsetminus U_i) \cup [H<t] \\& =\tbigcup_{j\ne i} U_j \cup \bigl( X\smallsetminus \tbigcup_{j\in I} U_j \bigr) \cup [H<t].
\end{align*}
Observe that $H\left( X\smallsetminus \tbigcup_{j\in I} U_j  \right) = h\left( X\smallsetminus \tbigcup_{j\in I} U_j  \right) = \{0\}$, and therefore we have $X\smallsetminus \tbigcup_{j\in I} U_j  \subseteq [H<t]$. Thus, we obtain
 $$ [\chi_{U_i} \wedge H< t] = \tbigcup_{j\ne i} U_j \cup [H<t],$$
which is open. Thus $F$ is continuous. Finally, we check that $F|_A =f$. Let $x\in A$. If $x\in X\smallsetminus \tbigcup_{i\in I} U_i$,  we have $f(x)=\zerohedgehog = F(x)$. Otherwise, if $x\in U_i$ for some $i\in I$, we get $F(x) = (H(x),i) = ((\pi_\kappa \circ f)(x), i) = f(x)$, as desired. \\[3mm]
(ii)$\implies$(i).  Let $\{F_i\}_{i\in I}$ be a discrete family of closed subsets with $| I | =\kappa$. We define 
 \[\begin{array}{rcclll}
f \colon& \tbigcup_{i\in I} F_i&\longrightarrow&MJ(\kappa)&&\cr
 &x&\longmapsto&f(x)  = (1,i) \quad \textrm{if } x\in F_i. 
 \end{array}\]
 
 By Proposition~\ref{propcombinedmap}, we have that $f$ is continuous, and thus it admits a continuous extension $F\colon X \longrightarrow MJ(\kappa)$.
 For every $i\in I$, set $U_i = F^{-1} \left( p( (0,1]\times\{i\} \right)$.  Clearly $\{U_i\}_{i\in I}$ is a pairwise disjoint family of open sets such that $F_i \subseteq U_i$ for every $i\in I$.
\end{proof}

\begin{corollary}
Let $X$ be a topological space. The following are equivalent:
\begin{enumerate}[\normalfont (i)]
\item $X$ is collectionwise normal;
\item For each cardinal $\kappa$ and every continuous mapping $f\colon A\longrightarrow MJ(\kappa)$ of a closed subspace $A$ of $X$ into the metric hedgehog, there exists a continuous mapping $F\colon X\longrightarrow MJ(\kappa)$ such that $F|_A =f$.\end{enumerate}
\end{corollary}

\begin{remark}
We already know that $MJ(2) \cong [0,1]$. Thus, by letting $\kappa=2$,  Theorem~\ref{colnorTh} yields Tietze's extension theorem as a particular case.
\end{remark}

\begin{remark}For every $2\leq  \kappa \leq \aleph_0$, we know that $\kappa$-collectionwise normality is equivalent to normality. Further, if $\kappa > 2$,  one has that the spaces $MJ(\kappa)$ and $MJ(\lambda)$ are non-homeomorphic for every cardinality $\lambda$. Thus, if we take $2< \kappa \leq \aleph_0$ in Theorem~\ref{colnorTh}, we get infinitely many different (topologically) Tietze-type characterizations of normality.
\end{remark}
\newpage

\section{A summary of topological properties of the three hedgehogs}
\vskip+10mm
The following table summarizes all the previously proved properties of the three hedgehogs.\\[10mm]

\begin{table}[htbp]
\begin{adjustbox}{center}
\renewcommand{\arraystretch}{1.5}
\begin{tabu}{|r|[2pt]c|c|c|[2pt]c|c|c|[2pt]c|c|c|}
\hline \multicolumn{1}{|c|[2pt]}{ } &
\multicolumn{3}{c|[2pt]}{Quotient}&
\multicolumn{3}{c|[2pt]}{Metric}&
\multicolumn{3}{c|}{Compact}
 \\ \hline
 & $\kappa < \aleph_0$ & $\kappa = \aleph_0$ & $\kappa > \aleph_0$ & $\kappa < \aleph_0$ & $\kappa = \aleph_0$ & $\kappa > \aleph_0$ & $\kappa < \aleph_0$ & $\kappa = \aleph_0$ &$\kappa > \aleph_0$ \\\tabucline[2pt]{-}
 Arcwise connected & + & + &+ &+&+&+& +&+ &+  \\ \hline
 Compact & + & - & - & + & - & - & + & + & + \\ \hline
  Complete & + &  &  &+ & + & + & +  & + &  \\ \hline
 First countable & + & - & - & + & + & + &+ & + & - \\  \hline
  Fr\'echet--Urysohn & + & + & + & + & + & + & + & + & +  \\  \hline
 Hausdorff & + & + & + & + & + & + & +& +& + \\ \hline
 Locally compact & +& - & - & + & - & - &+ &+ &+  \\ \hline
 Metrizable & +  & - & - & +  & + & + & + &+  &-  \\ \hline
 Normal  & + & + &+ &+&+&+& +& +&+  \\ \hline 
 Regular  & + & + &+ &+&+&+&+ &+ &+  \\ \hline 
 Second countable & + & - & - & + & + & -& + &+ &- \\ \hline
 Separable & + & + & - & + & + & - & + & + & - \\ \hline
  $T_1$ & + & + & + & + & + & + & +&+ &+  \\  \hline
 Totally bounded & + & & & + & - & - & + &+ &  \\\hline
\end{tabu}
\end{adjustbox}
\label{tab:multicol}
\end{table}

\end{document}